\documentclass[a4paper,10pt]{article}

\usepackage{amsmath,amssymb,amsthm}
\usepackage{color}

\usepackage[utf8]{inputenc}
\usepackage[T1]{fontenc}
\usepackage{dsfont}
\usepackage{hyperref}
\usepackage{fancyhdr,  graphicx, psfrag, color ,amsfonts,layout, amssymb}
\usepackage{euscript,epsfig}

\newcommand{\norm}[1]{\left|\!\left|#1\right|\!\right|}%Norm of a vector v : ||v||
\newcommand{\bd}{\partial}%Just a short way to write partial derivatives and boundary
\newcommand{\inter}{\stackrel{\circ}} %interior of a set

\newcommand{\calP}{{\mathcal{P}}}

\newcommand{\calC}{\mathcal{C}}
\newcommand{\calL}{\mathcal{L}}
\newcommand{\calE}{{\mathcal{E}}}

\newcommand{\calB}{{\mathcal{B}}}

\newcommand{\calO}{\mathcal{O}}

\newcommand{\calH}{{\mathcal{H}}}

\newcommand{\bbN}{\mathbb{N}}
\newcommand{\bbH}{\mathbb{H}}
 
\newcommand{\bbR}{\mathbb{R}}

\newcommand{\bbZ}{\mathbb{Z}}

\newcommand{\e}{\varepsilon}
\newcommand{\ds}{\displaystyle}

\renewcommand{\phi}{\varphi}
\renewcommand{\tilde}{\widetilde}
\renewcommand{\epsilon}{\varepsilon}

\DeclareMathOperator*{\supess}{sup\,ess}

\newtheorem{theo}{Theorem}[section]
\newtheorem{lemm}[theo]{Lemma}
\newtheorem{prop}[theo]{Proposition}

\newtheorem{rema}[theo]{Remark} 
\newtheorem{defi}[theo]{Definition} 
\newtheorem{coro}[theo]{Corollary} 
\newtheorem{fact}[theo]{Fact}

\title{Regularity of entropy, geodesic currents and entropy at infinity}
\author{Barbara Schapira, Samuel Tapie}
\date{Version \today}

\begin{document}   

\maketitle

\begin{abstract}
In this work, we introduce the notion of {\em entropy at infinity}, and define a wide class of   noncompact manifolds with negative curvature, those which admit a critical gap between entropy at infinity and topological entropy. 
We call them {\em strongly positively recurrent manifolds} (SPR), and provide many examples. We show that dynamically, they behave as 
compact manifolds. In particular, they admit a finite measure of maximal entropy. 

Using the point of view of currents at infinity, we show that on these SPR manifolds
the topological entropy of the geodesic flow varies in a $C^1$-way along (uniformly) 
$C^1$-perturbations of the metric. 
This result generalizes former work of Katok (1982) and Katok-Knieper-Weiss (1991) in the compact case. 
\end{abstract}

(\footnote{MSC Classification 37A10, 37A35, 37A40, 37B40, 37D35,37D40})

%%%%%%%%%%%%%%%%%%%%%%%%%%%%%%%%%%%%%%%%%%%%%%%

\section{Introduction}
 
\subsection{Variation of the topological entropy\,: An overview}

The initial motivation of this work was to answer the following simple question. 
Consider a hyperbolic surface of finite volume and 
a smooth compact perturbation of the metric. 
Does the topological entropy of the geodesic flow vary regularly ? 
More generally, what happens for a smooth perturbation of the metric 
of a noncompact negatively curved Riemannian manifold ? 

The answer has been known on compact manifolds since almost thirty 
years \cite{KKW91,KKPW,Flaminio95}, 
and has been extended to the convex-cocompact case in \cite{Tap11}. 
A similar argument gives the regularity of the topological entropy 
for a perturbation of an Anosov flow, cf  \cite{KKW91}. 

Compactness of the underlying space is crucial in the above results, 
and no result was known untill now for manifolds 
with a non-compact non-wandering set. 
Even the case of a smooth compact  perturbation of the metric of a 
finite volume hyperbolic surface was not accessible with their arguments. 
Let us recall the two main steps of their argument to understand why.

The key step is the following inequality, due to Katok in \cite{Katok82} 
for surfaces, extended in \cite{KKW91} to all dimensions.

\begin{theo}[Katok 1982 ; Katok-Knieper-Weiss 1991]
Let $g_1, g_2$ be Riemannian metrics with negative sectional curvature 
on the same compact manifold $M$. Then the entropies of their geodesic flows satisfy
\begin{equation}\label{eq:CompEnt1}
h_{top}(g_1) \leq h_{top}(g_2)\,\times\, \int_{S^{g_1}M} \norm{v}^{g_2} d\overline m_{BM}^{g_1}(v) ,
\end{equation}
where $\ds \norm{v}^{g_2} = \sqrt{g_2(v,v)}$ and $\ds \overline m_{BM}^{g_1}$ 
is the normalized Bowen-Margulis measure on the $g_1$-unit tangent bundle $S^{g_1}M$ 
for the $g_1$-geodesic flow.
\end{theo}

Reversing the role of $g_1$ and $g_2$ also provides a lower bound for $h_{top}(g_1)$, 
and a first order power expansion gives the following smoothness result.

\begin{theo}[Katok-Knieper-Weiss 1991]
Let $(g_\lambda)_{\lambda\in (-1, 1)}$ be a $\calC^2$-family 
of $\calC^2$ Riemannian metrics with negative sectional curvature 
on the same compact manifold $M$. Then $\lambda\mapsto h_{top}(g_\lambda)$ 
is $\calC^1$, and its derivative is given by
\begin{equation}\label{eq:VarEnt1}
 \left.\frac{d}{d\lambda}\right|_{\lambda = 0} h_{top}(g_\lambda) =
 -h_{top}(g_0)\,\times\,\int_{S^{g_0}M} \left.\frac{d}{d\lambda}\right|_{\lambda = 0}
 \norm{v}^{g_\lambda}d\overline m_{BM}^{g_0}(v).
\end{equation}
where $\ds m_{BM}^{g_0}$ is the normalized Bowen-Margulis measure on 
the $g_0$-unit tangent bundle $S^{g_0}M$ for the $g_0$-geodesic flow.
\end{theo}

In the previously quoted works, the proofs of (\ref{eq:CompEnt1}) 
strongly use the compactness of the non-wandering set. 
In the first part of our paper, we use a different approach 
to generalize it to the non-compact setting. 
This improves it even in the compact case, 
providing an explicit transformation rule for the entropies, 
equality which immediately implies (\ref{eq:CompEnt1}), 
and has other interesting consequences. 

The previously known proofs of (\ref{eq:VarEnt1}) use 
the compactness of $M$ for a crucial point: 
to ensure the finiteness and the continuity of the normalized 
Bowen-Margulis measures $m_{BM}^{g_\lambda}$ in the weak-*topology as $\lambda$ varies. 
Neither finiteness of the Bowen-Margulis measure nor its continuity under a variation 
of the metric can be ensured in general. 
Maybe the most striking fact of our work is that we introduce a new wide class of manifolds, 
which we call {\em SPR manifolds}, SPR meaning {\em strongly / stably positively recurrent}. 
The terminology {\em Stably positively recurrent} has been introduced 
by Gurevic-Savchenko \cite{GS} in the context of countable Markov shifts. 
Sarig \cite{Sarig01} modified it, in the same context, into {\em strongly positively recurrent}, 
terminology which has been used later by other authors as Buzzi \cite{BBG}. 
See also the very recent work of Velozo \cite{Velozo}, who follows also this terminology.  
Both terminologies are meaningful, and had not yet been considered in a geometric context.  

The class of SPR manifolds that we define here has the remarkable property 
that the Bowen-Margulis measure is finite, and moreover stays finite 
and varies continuously along small perturbations. 
In particular, under uniformly $\calC^2$ variation of 
such SPR Riemannian metrics, 
the topological entropy is $\calC^1$ and its derivative 
is given by (\ref{eq:VarEnt1}).
 
These SPR manifolds include finite volume hyperbolic manifolds, 
and more generally almost
all known examples where the geodesic flow admits 
a (finite) measure of maximal entropy, 
as  geometrically finite negatively curved manifolds 
with spectral gap \cite{DOP}, 
Schottky product examples from \cite{Peigne}, 
and unpublished examples of Ancona \cite{Ancona}.  
The class of SPR manifolds is much larger than only 
the above mentioned examples. 
We postpone the extensive study of SPR manifolds to a later paper \cite{STdufutur}. 
Therefore, the second half of our paper will be devoted to the presentation of a geometrical setting, 
as large as possible, where this finiteness and continuity of Bowen-Margulis measures can be ensured.

Let us now present our main results with more details.

%%%%%%%%%%%%%%%%%%%%%%%%

\subsection{Invariant measures and change of Riemannian metrics}

Let $(M, g_1)$ be a complete Riemannian manifold, 
and $g_2$ be another Riemannian metric on $M$ such that 
there exists $C>1$ with $\ds \frac 1 C g_1 \leq g_2 \leq C g_1$. 
We assume moreover that both $g_1$ and $g_2$ have pinched negative sectional curvatures 
with uniformly bounded derivatives: 
this implies that $g_1$-geodesics are $g_2$-quasi-geodesics 
and the visual boundary of the universal cover $(\tilde M, g_1)$ 
is canonically identified with the visual boundary of $(\tilde M, g_2)$; 
we will denote it by $\bd \tilde M$. We will use extensively 
this correspondance to compare the dynamics of the geodesic flows on $S^{g_1}M$ and $S^{g_2}M$. 

Let $\Gamma = \pi_1(M)$ acting on the universal cover $\tilde M$, 
let $m$ be a locally finite measure on $S^{g_1}M$, invariant by the geodesic flow $(g_1^t)_{t\in \bbR}$, 
and $\tilde m$ its lifts to $S^{g_1}\tilde M$.
We write $\bd^2 \tilde M = (\bd \tilde M\times\bd \tilde M)\backslash$Diag. 
In $g_1$-Hopf coordinates (cf Section \ref{sec:Hopf}), 
$S^{g_1}\tilde M \simeq \bd^2\tilde M \times \bbR$,
and $\tilde m$ has a local product structure of the form $d\tilde m = d\mu \times dt$, 
where $\mu$ is a $\Gamma$-invariant \emph{geodesic current} on $\bd^2 \tilde M$. 
We write therefore $m = m_\mu^{g_1}$. 

We can now define a measure $\tilde m_\mu^{g_2}$ on $S^{g_2}\tilde M$, 
given in $g_2$-Hopf parametrization by the same local product formula
 $\ds \tilde m_\mu^{g_2} = d\mu \times dt$: by $\Gamma$-invariance, 
this induces a locally finite measure $m_\mu^{g_2}$ on $S^{g_2}M$, 
which is invariant for the geodesic flow $(g_2^t)_{t\in \bbR}$. 
The ergodic properties of $(S^{g_1}M, g_1^t, m_\mu^{g_1})$ and 
$(S^{g_2}M, g_2^t, m_\mu^{g_2})$ are strongly related.

Well known facts imply that if $m_\mu^{g_1}$ and $m_\mu^{g_2}$ are finite 
then one is ergodic or conservative if and only if the other is. 
The reader may believe that, since $\ds \frac 1 C g_1 \leq g_2 \leq C g_1$, 
then $m_\mu^{g_1}$ is finite if and only if $m_\mu^{g_2}$ is. 
We will indeed show that it is the case and relate the masses and entropies of these measures.

\medskip

In this purpose, let us introduce the \emph{instantanous geodesic stretch} 
$\calE^{g_1\to g_2} : S^{g_1}\tilde M \to \bbR$ defined for all $v\in S^{g_1}\tilde M$ by
$$
\calE^{g_1\to g_2}(v) = 
\left.\frac{d}{dt}\right|_{t=0^+}\calB^{g_2}_{v_+^{g_1}}(\pi v, \pi g_1^t \tilde v) 
= \left.\frac{d}{dt}\right|_{t=0^+}\calB^{g_2}_{v_+^{g_1}}(o, \pi g_1^t \tilde v),
$$
where  $\calB^{g_2}_{v_+^{g_1}}(.,.)$ is the Busemann function for $g_2$ based at 
the end point of the $g_1$-geodesic generated by $v$. 
By  $\Gamma$ invariance, it induces a map $\calE^{g_1\to g_2}: S^{g_1}M \to \bbR$.
We will see in Section \ref{subsec:stretch} that this is the derivative 
along $g_1$-geodesics of a natural \emph{Morse correspondance} 
$\Psi^{g_1\to g_2} : S^{g_1}M\to S^{g_2} M$ between the $g_1$ and $g_2$ geodesic flows.
This implies the following.

\begin{prop}\label{exact-change-variable}
For all $m_\mu^{g_2}$-measurable map $G : S^{g_2}M\to \bbR$, 
the map $G\circ \Psi^{g_1\to g_2}$ is $m_\mu^{g_1}$-measurable and
$$
\int_{S^{g_2}M} G\,dm^{g_2}_\mu=\int_{S^{g_1}M} G\circ \Psi^{g_1\to g_2} \times \calE^{g_1\to g_2} \,dm^{g_1}_\mu\,.
$$
In particular, the masses of $m_\mu^{g_2}$ satisfies
$$
\norm{m_\mu^{g_2}} = \int_{S^{g_1}M} \calE^{g_1\to g_2} \,dm^{g_1}_\mu\,.
$$
\end{prop}

Some other versions of the geodesic stretch has already been considered 
in \cite{FF93} of \cite{Kni95}; we explain in Section \ref{subsec:stretch} 
the relationship with these references and the interest of our new definition.
We then introduce in Section \ref{sec:Entropy} a notion of \emph{local entropy} 
for invariant measures, which is an analogous in the non-compact setting 
to Brin-Katok entropy, and which coincides with the classical metric entropy 
for Gibbs measures. This also allows us to relate the local entropies of 
$(S^{g_1}M, g_1^t, m_\mu^{g_1})$ and $(S^{g_2}M, g_2^t, m_\mu^{g_2})$.

\begin{theo}
Under the previous notations, the local entropies of $(g_1^t, m_\mu^{g_1})$ 
and $(g_2^t, m_\mu^{g_2})$ are related as follows.
 $$
h_{loc}(m^{g_2}_\mu,g_2)=I_\mu(g_2, g_1)\,.\, h_{loc}(m^{g_1}_\mu,g_1)\,.
%=\int_{S^{g_2}M}\mathcal{E}^{g_2\to g_1}(v)\frac{dm_\mu^{g_2}(v)}{\norm{m_\mu^{g_2}}}\,. \,h_{loc}(m^{g_1}_\mu, g_1)\,.
$$ 
\end{theo}

In particular, the combination of the previous theorem with 
the variational principle implies that 
$$
h_{top}(g_2)\leq \int_{S^{g_2}M}\mathcal{E}^{g_2\to g_1}(v)\frac{dm_\mu^{g_2}(v)}{\norm{m_\mu^{g_2}}} h_{top}(g_1),
$$
which is an optimal improvement of \ref{eq:CompEnt1}.

Eventually, in Section \ref{sec:Gibbs}, we show that $m_\mu^{g_1}$ 
is a Gibbs measure for the potential $G: S^{g_1}\to \bbR$ if and only if 
$m_\mu^{g_2}$ is a Gibbs measure for the potential $G\circ \Psi^{g_2\to g_1}\times \calE^{g_2\to g_1}$, 
and give some applications of this last fact to a comparison 
between the length spectrum of $(M,g_1)$ and $(M,g_2)$.

%%%%%%%%%%%%%%%%%%%%%%%%%%%%%%%%%%%%%%%%%%%%%%%%%%%%%%%%%%%%%%%%%%%

\subsection{Entropy at infinity, SPR manifolds and Bowen-Margulis measures}

Let $(M,g)$ be a Riemannian manifold with pinched negative sectional 
curvatures whose derivatives are uniformly bounded. 
We introduce a notion of {\em entropy at infinity} (see section \ref{sec:SPR}), 
which measures the highest possible complexity of the dynamics 
outside a  compact set in the manifold.
 
 We call the Riemannian manifold $(M,g)$ {\em strongly positively recurrent}, shortly SPR, 
if the entropy at infinity is strictly smaller than the topological entropy 
of the geodesic flow. 
This SPR property implies that the geodesic flow  admits a measure of maximal entropy, 
that this fact remains true under a nice small perturbation of the metric, and that these measures vary continuously in the narrow topology. 
Let us summarize the main results that we establish here on the SPR property. 
We refer to \cite{STdufutur} for further study of this property and its consequences. 

\begin{theo}\label{examples-SPR}
Let $(M,g)$ be a Riemannian manifold with pinched negative curvature. 
\begin{enumerate}
\item The SPR property implies that the geodesic flow admits an invariant probability measure of maximal entropy $m_{BM}^g$, 
the so-called {\em Bowen-Margulis measure}. In the terminology of \cite{PS16},
the SPR property implies that the  geodesic flow is positively recurrent. 
\item Geometrically finite manifolds with critical gap (see \cite{DOP}) have the SPR property\,,
\item Topologically infinite examples of \cite{Ancona} have the SPR property,
\item Schottky product examples of \cite{Peigne} have the SPR property.

\end{enumerate}
\end{theo}

As mentionned above, this SPR property is stable in the following sense.

\begin{theo} \label{theo:SPRStable}
Let $(M,g_0)$ be a SPR manifold with pinched negative curvature and bounded 
derivatives of the curvature. Let $(g_\varepsilon)_{\e\in(-1,1)}$ 
be a $C^1$-uniform variation of the metric. 
Then there exists $\e_0>0$ such that for all $\e\in (-\e_0, \e_0)$, 
the manifold $(M,g_0)$ is SPR. 
Moreover, the Bowen-Margulis measures $(m_{BM}^{g_\e})$ vary continuously 
in the narrow topology at $\e=0$.
\end{theo}

This allows us to show the following regularity property for the topological entropy, 
which answers our initial question. We refer to section \ref{sec:SPR} 
 for technical details on the assumptions. 
Denote by $h_{top}(g)$ the topological entropy of 
the geodesic flow of the metric $g$.

\begin{theo} \label{theo:main} Let $(M,g_0)$ be a SPR manifold with pinched negative curvature 
and bounded derivatives of the curvature.
 Let $(g_\varepsilon)$ be a $C^1$-uniform variation of the metric with negative sectional curvatures. 
Then the map $\varepsilon\to h_{top}(g_\varepsilon)$ is $C^1$ near $\epsilon=0$, with derivative at $0$
$$
{\frac{{\rm d}\, h_{top}(g_\e)}{{\rm d}\,\e}}_{|\e=0}=-h_{top}(g_0)\int_{S^{g_0}M} \left.\frac{{\rm d}\, \norm{v}^{g_\e}}{{\rm d}\,\e}\right|_{\e=0 }\,d\frac{m_{BM}^{g_0}}{\|m_{BM}^{g_0}\|}\,,
$$ 
the measure $\frac{m_{BM}^{g_0}}{\|m_{BM}^{g_0}\|}$ being the invariant probability measure of maximal entropy for the $g_0$-geodesic flow. 
\end{theo}

Let us emphasize the fact that this theorem is valid in a much greater 
generality than what we thought initially possible.  
On the one hand, SPR manifolds are a very general and interesting class of manifolds, 
much larger than the well known and well studied class of finite volume, 
or even geometrically finite hyperbolic manifolds, 
as illustrated by Theorem \ref{examples-SPR}. 
It may be an optimal class to get such result in the sense that we guess that
phase transitions for the entropy can happen when the manifold is not SPR (see \cite{STdufutur}). 

On the other hand, we allow much more general perturbations than 
only compact ones since we deal with noncompact $C^2$-perturbations of our metric, 
as soon as they are not too wild at infinity.

\medskip

The paper is organized as follows. 
In Section \ref{sec:Hopf}, we develop the point of view of geodesic currents at infinity, 
which allows us to associate to an invariant measure $m_\mu^{g_1}$ 
for the geodesic flow for $(M,g_1)$ an invariant measure $m_\mu^{g_2}$ 
for the geodesic flow on $(M,g_2)$, and compare their ergodic properties.

In Section \ref{sec:Entropy}, we introduce different notions of entropy 
and develop methods of section \ref{sec:Hopf} to relate the entropies 
of $m_\mu^{g_1}$ and $m_\mu^{g_2}$.  

In Section \ref{sec:Gibbs}, we recall general fact about Gibbs measures 
on non-compact manifolds, we show that $m_\mu^{g_1}$ is a Gibbs measure
 if and only if $m_\mu^{g_2}$ is and give applications to the length spectrum.

In Section \ref{sec:ConvGeod} we show some continuity results for geodesics, 
Busemann functions and non-normalized Bowen-Margulis measures which will be needed in the sequel.

In Section \ref{sec:DiffEnt}, we first show that for a fixed geodesic current $\mu$ on $\bd^2\tilde M$, the metric entropy $\ds \e \mapsto h\left(g^t_\e, m_\mu^{g_\e}\right)$ is $\calC^1$ under a $\calC^1$-uniform variation of the Riemannian metrics $g_\e$. We then show in a very similar proof that, if under a $\calC^1$-uniform variation of Riemmanian metrics the normalized Bowen-Margulis measures $\overline{m}_{BM}^{g_\e}$ vary continuously in the narrow topology, then the topological entropy is also $\calC^1$.

Eventually, in Section \ref{sec:SPR}, we introduce entropy at infinity and SPR manifolds, we show that they have finite Bowen-Margulis measure, and that under a small $\calC^1$-uniform variation of Riemmanian metrics they remain SPR.
On the way, we  give some properties of the entropy at infinity of independent interest.

Theorem \ref{examples-SPR} follows from results of section \ref{ssec:SPR-Ex}, where we provide many examples of SPR manifolds. Theorem \ref{theo:SPRStable} is a reformulation of the second part of
Theorem \ref{theo:SPR-FinBM}. At last, our main variational formula for the topological entropy, Theorem \ref{theo:main}, 
follows from Theorems \ref{theo:VarEnt} and \ref{theo:SPRStable} (or \ref{theo:SPR-FinBM}).

\subsection*{Acknowledgments}

The authors acknowledge the support of the ANR grant  ANR JCJC 0108-GEODE (2010-2015), the grant PEPS CNRS Egalit\'e CHATS (Integer Project 2015), and the Centre Henri Lebesgue ANR-11-LABX-0020-01.

%%%%%%%%%%%%%%%%%%%%%%%%%%%%%%%%%%%%%%%%%%%%%%%%%%%%%%%%%%%%%%
%%%%%%%%%%%%%%%%%%%%%%%%%%%%%%%%%%%%%%%%%%%%%%%%%%%%%%%%%%%%%%
 
\section{Hopf parametrization and geodesic currents}\label{sec:Hopf}

%%%%%%%%%%%%%%%%%%%%%%%%%%%%%%%%%%%%%%%%%%%%%%%%%%%%%%%%%%%%%%
\subsection{Hopf parametrization and geodesic flow}\label{subsec:hopf}

Let $(M,g_0)$ be a complete manifold with pinched negative sectional curvatures 
satisfying $-b^2 \leq K_{g_0} \leq -b^2 <0$, and derivatives of the curvature bounded. 
Let $\tilde M$ be its universal cover, equipped with the lifted metric 
which we will still denote by $g_0$, 
and let $\bd_{g_0} \tilde M$ be its visual boundary.
 Let $\Gamma = \pi_1(M)$ be the fundamental group, 
acting properly by diffeomorphisms on $\tilde M$. 
Denote by $p_\Gamma$ indistinctly the projection $\tilde M\to M$ 
and its linear tangent map $T\tilde M\to TM$. 
A metric $g$ on $M$ (or equivalently, a $\Gamma$-equivariant metric $g$ on $\tilde M$) 
will be called \emph{admissible} if it has pinched negative sectional curvature, 
if the derivatives of the curvature are bounded, 
and if there exists a constant $C_1(g_0,g)>1$ such that at all $x\in M$,
    \begin{equation}\label{eq:QIsom}
     \frac 1 {C_1(g_0,g)}\, g_0 \leq g \leq C_1(g_0,g)\, g_0.
    \end{equation}

    This implies that $g$-geodesics are $g_0$-quasi-geodesics, 
which are contained in the $C_2(g_0,g)$-neighbourhood of $g_0$-geodesics, 
where $C_2(g_0,g)$ only depends on $C_1(g_0,g)$ (see \cite[Th.1.7 p401]{BH99}). 
In particular the visual boundary $\bd_{g} \tilde M$ of $(\tilde M, g)$ 
is canonically identified to the visual boundary of $(\tilde M, g_0)$, 
and they will therefore both be denoted by $\bd \tilde M$. 
Moreover, this identification is H\"older continuous w.r.t the visual %\footnote{visual or Gromov???}
 distances induced by both $g_0$ and $g$, so that $\bd\tilde M$ has a natural H\"older structure. 

The {\em limit set} $\Lambda_\Gamma \subset \bd \tilde M$ is 
the set of accumulation points of any orbit $\Gamma.x$ on the boundary. 
It does not depend either of the chosen admissible metric. 
 The {\em radial limit set} $\Lambda_\Gamma^{r}\subset\Lambda_\Gamma$ 
is the set of endpoints of geodesics which, on the quotient manifold $M$, 
return infinitely often to some compact set. 
It does not depend either on the chosen admissible metric. 
    
    \medskip
    
Let us fix once for all a point $o\in \tilde M$. 
Let $g$ be any admissible metric on $M$, and   $d^g$ 
the distance induced by $g$ on $M$ and $\tilde M$.
Denote by $S^gM$ (resp. $S^g\tilde M$) the unit tangent bundle of $(M,g)$ (resp.  $(\tilde M,g)$), and $\bd^2 \tilde M = (\bd \tilde M\times \bd \tilde M)\backslash\mbox{Diag}$. 
We  write $\pi : TM \to M$ and $\pi: T\tilde M \to \tilde M$ the projections from the tangent bundle to its base, 
and by $(g^t)_{t\in \bbR}$ the geodesic flow on $S^gM$ or $S^g\tilde M$. 
For any $v\in S^g \tilde M$,  write $v_-^g$ and $v_+^g$ for the endpoints
 in $\bd \tilde M$ of the geodesic $\{\pi g^t v ; t\in \bbR\}$. 
    
\begin{rema}\rm 
We keep track in our notations of the metric $g$ 
since we will soon compare these quantities for two different admissible metrics $g_1$ and $g_2$.
\end{rema}

\noindent
For all $\xi\in \bd \tilde M$, let $\calB^g_\xi$ be the Busemann function at 
$\xi$ defined, 
for any $x,y\in \tilde M$, by
$$
\calB^g_\xi(x,y) = \lim_{z\to \xi} d^g(x,z) - d^g(y,z).
$$

\noindent
The map 
$$
\ds H^g : v \mapsto \left(v_-^g, v_+^g, \calB_{v_+^g}(o, \pi v)\right)
$$ 
is a H\"older homeomorphism from $S^g \tilde M$ to 
$\bd^2 \tilde M\times \bbR$, called the \emph{Hopf parametrization} of the unit tangent bundle. 
    
The action of $\Gamma$ by (differentials of) isometries 
on $S^g\tilde M$ can be written in these coordinates as 
$$
\gamma.(v^g_-,v^g_+,t)=\left(\gamma.v^g_-,\gamma.v^g_+,t+\calB^g_{v^g_+}(o,\gamma^{-1}.o)\right)\,.
$$
Let us emphasize the fact that this action of $\Gamma$ on 
$\bd^2\tilde M\times \bbR$, and more specifically on the third factor,  depends strongly on the cocycle $\calB^g$, and therefore on the metric $g$.

%%%%%%%%%%%%%%%%%%%%%%%%%%%%%%%%%%%%%%%%%%%%%%%%%%%%%%%%%%%%%%
  \subsection{Geodesic currents and invariant measures}
  
 In the coordinates given by the Hopf parametrization of $S^g\tilde M$,
 the geodesic flow $(g^t)$ acts by translation on the last factor\,: 
for all $ v\in S^g\tilde M$, and 
$s\in \bbR$, 
$$
\mbox{if}\quad H^g(v) = (v_-, v_+, t) \quad \mbox{then}\quad   H^g(g^s v) = (v_-, v_+, t+s)\,.
$$
Therefore, any positive Radon measure $m$ on $S^g M$ invariant by the flow 
lifts to a measure $\tilde m$ on $S^g \tilde M$ of the form 
$\tilde m = (H^g)^*(\mu\times dt)$, 
where $dt$ is the Lebesgue measure on $\bbR$, 
and $\mu$ is a $\Gamma$-invariant locally finite positive measure 
on $\bd^2 \tilde M$. 
    
\begin{defi}[Geodesic current]
A \emph{$\Gamma$-invariant geodesic current}, or simply \emph{geodesic current}, 
is a $\Gamma$-invariant  positive Radon  %\footnote{Verifier difference localement fini/ Radon } 
measure on $\bd^2 \tilde M$. 
\end{defi}

Given any geodesic current $\mu$ and any admissible metric $g$ on $M$, 
we will denote by $m^g_\mu$ the unique measure on $S^gM$ 
invariant by the geodesic flow $(g^t)$ 
whose lifts on $S^g\tilde M$ is $\tilde m^g_\mu = (H^g)^*(d\mu\times dt)$. 
The {\em non-wandering set}  $\Omega^g\subset S^gM$  of 
the geodesic flow $(g^t)$ is 
$$
\Omega^g = (H^g)^{-1}\left((\Lambda_\Gamma\times \Lambda_\Gamma)\backslash \mbox{Diag}\times \bbR\right)\,.
$$
It was shown in Eberlein \cite{Eberlein} that for the geodesic flow of a negatively curved manifold, this definition coincides with the usual definition of the nonwandering set of a flow.  

It follows from (\ref{eq:QIsom}) and  \cite[Thm 1.7 p401]{BH99} that $\Omega^g$ is compact (i.e. (M,g) is \emph{convex-cocompact}) if and only if $\Omega^{g_0}$ is. We will mainly be interested in the case where $\Omega^g$ is \emph{not compact}.
    
    The measure $m^g_\mu$ is locally finite, but may have infinite mass as soon as $(M,g)$ is not convex-cocompact. 
We will use all over this paper the fact that many properties of the measure $m_\mu^g$ only depend on the geodesic current $\mu$ and not on the chosen admissible metric $g$.

Recall first that an invariant measure is {\em ergodic} if every invariant set either 
has measure zero or its complementary set has measure zero. 
An invariant measure is  {\em periodic} if it is (proportional to) the Lebesgue measure on a periodic orbit. 
The measure $m$  is {\em conservative} if it satisfies the conclusion of Poincar\'e recurrence Theorem\,: for all sets $A$ of positive measure $m(A)>0$, and $m$-almost all vectors $v$, the orbit $(g^tv)$ returns infinitely often in $A$. 
The measure $m$ has a {\em product structure} if the associated geodesic current is equivalent to a product of measures on $\partial \tilde M$. 
The measure $m$ is {\em strongly mixing} if 
it is finite and satisfies $m(A\cap g^t B)\to m(A)m(B)$ 
when $t\to \pm \infty$ for all Borel sets $A,B$. 
It is {\em weakly mixing} if it is finite and 
$\frac{1}{T}\int_0^T\left|m(A\cap g^t B)-m(A)m(B)\right|$
 goes to $0$ when $T\to \pm \infty$ for all Borel sets $A,B$.

 First well known properties are given in the following proposition.
    \begin{prop}\label{easy-comparison}
     Let $\mu$ be a geodesic current, let $g_1$ and $g_2$ be two admissible metrics on $M$. Then 
     \begin{enumerate}
     \item the measure $m_\mu^{g_1}$ is supported by a 
(finite number of) closed geodesic(s) if and only if $m_\mu^{g_2}$ is ;
      \item the measure $m_\mu^{g_1}$ is ergodic for 
the geodesic flow $(g_1^t)$ if and only if $m_\mu^{g_2}$ 
is ergodic for the geodesic flow $(g_2^t)$ ;
      \item the measure $m_\mu^{g_1}$ is conservative for 
the geodesic flow $(g_1^t)$ if and only if $m_\mu^{g_2}$ is 
conservative for the geodesic flow $(g_2^t)$ ;
\item  the measure $m_\mu^{g_1}$ has a local product structure
iff the measure $m_\mu^{g_2}$ has a local product structure. 
     \end{enumerate}
    \end{prop}
    
\begin{proof} The measure $m_\mu^{g_1}$ is supported by a 
closed geodesic if and only if 
$\mu$ is carried by the $\Gamma$-orbit of a couple 
$(\xi_-, \xi_+)\in \bd^2 \tilde M$ where $\xi_-$ and $\xi_+$ 
are the fixed point of a hyperbolic element $\gamma\in \Gamma$. 
Since this property does not depend on $g_1$, it shows 1. 

The measure $m_\mu^{g_1}$ is ergodic for the geodesic flow $(g_1^t)$ 
if and only
$\mu$ is ergodic under the action of $\Gamma$ on $\bd^2 \tilde M$ 
(cf for instance \cite[p. 19]{Roblin}). 
This property only depends on $\mu$, which shows 2. 

The measure $m_\mu^{g_1}$ is conservative for the geodesic flow $(g_1^t)$ 
if and only $\mu$ gives full measure to 
$\Lambda^r_\Gamma\times \Lambda^r_\Gamma$   \cite[proof of (b) page 19]{Roblin} 
where $\Lambda_\Gamma^r$ is the radial limit set, 
which does not depend on the (admissible) metric $g_i$. 
This shows~3.
\end{proof}

One should note that in general an invariant measure $m_\mu^g$, 
even with finite total mass, has no reason to be a probability measure.

We will see further nontrivial relationships between $m_\mu^{g_1}$ and $m_\mu^{g_2}$ later. It would be interesting to know if this kind of result can be extended to mixing property. Explicit examples of mixing measures have all a local product structure. But there exist mixing measures without such a product structure.

%%%%%%%%%%%%%%%%%%%%%%%%%%%%%%%%%%%%%%%%%%%%%%%%%%%%%%%%%%%

  \subsection{Geodesic stretches}\label{subsec:stretch}
  
  Let $g_1$ and $g_2$ be two admissible metrics. 
  For all $v\in S^{g_1}M$, define  the quantity  
  \begin{equation}\label{eq:GeodStretch1}
   e^{g_1 \to g_2}(v) = 
\inf_{t>0} \frac{d^{g_2}(\pi \tilde v, \pi g_1^t \tilde v)}{t},
  \end{equation}
  where $\tilde v$ is a lift of $v$ to $S^g \tilde M$. 
This does not depend on the choice of $\tilde v$. 
Knieper showed in \cite{Kni95} that if $m$ is any invariant measure for $(g_1^t)$, then for  $m$-almost every $v\in S^{g_1}M$, 
   \begin{equation}\label{eq:GeodStretch2}
 e^{g_1 \to g_2}(v) = 
\lim_{t\to+\infty} \frac{d^{g_2}(\pi \tilde v, \pi g_1^t \tilde v)}{t}.
    \end{equation}
    
This asymptotic geodesic stretch has been studied by many authors, among which 
\cite{FF93}, \cite{Kni95}, \cite{Glorieux}.  
Sambarino uses a different point of view of reparametrization of the geodesic flow (see for example \cite{Sambarino}) which is very close from our point of view below. 

Recall that, $\xi\in\partial \tilde M$ being fixed, 
the Busemann function $\calB^{g}_\xi(\dot,\dot)$ is $C^2$ 
on $\tilde M^2$ \cite[prop. 3.1]{HIH77}.  
Therefore, for all $v\in S^{g_1} M$, we can define
  \begin{equation}\label{eq:GeodStretch3}
  \calE^{g_1\to g_2}(v) = 
\left.\frac{d}{dt}\right|_{t=0^+}\calB^{g_2}_{v_+^{g_1}}(\pi v, \pi g_1^t \tilde v) 
= \left.\frac{d}{dt}\right|_{t=0^+}\calB^{g_2}_{v_+^{g_1}}(o, \pi g_1^t \tilde v),
  \end{equation}
  where $\tilde v\in S^{g_1} \tilde M$ 
is any lift of $v$, and $\calB^{g_2}_{v_+^{g_1}}(.,.)$ 
is the Busemann function for $g_2$ based at 
the end point of the $g_1$-geodesic generated by $v$. 
This definition was inspired by Ledrappier's paper \cite{Led94}. 
In his notations, our geodesic stretch satisfies 
$\calE^{g_1\to g_2}(v)=\alpha^{g_2}(v)$, 
where $\alpha^{g_2}$ is the harmonic $1$-form on the $g_1$-stable foliation 
associated to the Busemann cocycle of the metric $g_2$.

\begin{figure}[ht!]\label{fig:geod-stretch} 
\begin{center}
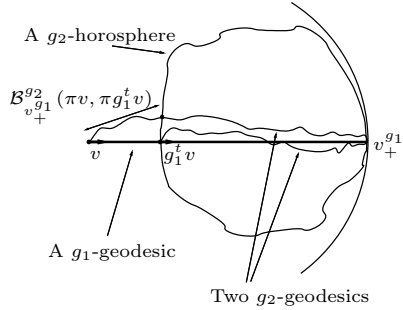 
\caption{Geodesic stretch} 
\end{center}
\end{figure} 

  \begin{defi}[Geodesic stretch]
  The maps $e^{g_1 \to g_2} : S^{g_1}M\to \bbR$ and 
$\calE^{g_1 \to g_2} : S^{g_1}M\to \bbR$ will be called respectively 
the \emph{asymptotic} and \emph{instantaneous geodesic stretch} of $g_2$ with respect to $g_1$.
  \end{defi}
 
 Anyway, we will most of the time call them both without distinction {\em  geodesic stretch}. 

 By construction, for all $v\in S^{g_1}M$, 
$\calE^{g_1\to g_1}(v) = e^{g_1\to g_1}(v) = 1$. 
Observe that there is no obvious relation from the definition between 
$e^{g_1\to g_2}$ (resp. $\calE^{g_1\to g_2}$) and 
$e^{g_2\to g_1}$ (resp. $\calE^{g_2\to g_1}$). 

If $m$ is ergodic, then $e^{g_1\to g_2}$ is $m$-almost everywhere constant. 
Of course its value   strongly depends on the measure $m$. 
On the opposite, the map $\calE^{g_1\to g_2}$ is defined everywhere 
and does not depends on the chosen measure. 
It is in general non-constant, 
globally H\"older on $S^{g_1}M$ \cite[Appendix of Brin]{Bal95},\cite[thm 7.3]{PPS}, 
and $C^1$ along $g_1$-geodesics (as Busemann functions are $C^2$, see \cite{HIH77}). 
We will need the following basic estimate.
 
 \begin{lemm}\label{lem:Ineq-Knieper}
  Let $g_1$ and $g_2$ be two admissible metrics, and $m$ any $g_1$-invariant measure. 
For $m$-almost all $v\in S^{g_1} \tilde M$,
$$
e^{g_1\to g_2}(v) \leq \int_{S^{g_1}M}\norm{v}^{g_2}\,dm\,,
$$
whereas for all $v\in S^{g_1} \tilde M$,
$$   \calE^{g_1\to g_2}(v) \leq \norm{v}^{g_2}\,.
$$
\end{lemm}
\begin{proof}
 The first estimate was shown in \cite[p.44]{Kni95}.
 The second follows from  triangular inequality. Indeed, 
for all $t\ge 0$, $\calB_{v^{g_1}_+}^{g_2}(\pi(v), \pi(g_1^t v))\le d^{g_2}(\pi(v),\pi(g_1^tv))$, and these two quantities vanish at $t=0$ so 
that their derivatives at $t=0$ satisfy the same inequality.
Moreover, $d^{g_2}(\pi(v),\pi(g_1^t v))$ is smaller than 
the $g_2$-length of the curve $(\pi(g_1^sv))_{0\le s\le t}$, 
whose derivative at zero is exactly $\|v\|_{g_2}$. 
%\footnote{B: pour ce que veut faire Samuel, ajouter un lemme plus precis disant que la derivee en $t=0$ de la distance $d^{g_2}(\pi v,\pi g_1^t v)$ est exactement $\|v\|_{g_2}$. }
 \end{proof}

 Lemma \ref{lem:GeodStretch} and  Corollary \ref{coro:GeodStretch} below 
justify the common name of {\em geodesic stretch} given to the  two maps $e^{g_1\to g_2}$ 
and $\calE^{g_1\to g_2}$.
Before stating them, recall a well known feature of negative curvature. 
On a geodesic space $X$, each triangle $(x,y,z)$ admits an interior triangle $(p,q,r)$ such that
$d(r,x)=d(q,x)$, $d(q,z)=d(p,z)$ and $d(p,y)=d(r,y)$. 
If $g$ is a metric with negative curvature, there exists a universal constant $\Delta(g)$ such that 
for any geodesic triangle $(x,y,z)$ in $\tilde{M}$, the associated interior triangle has sides smaller than $\Delta(g)$. 

 \begin{lemm}\label{lem:GeodStretch}
  There exists $C_3=C_3(g_1,g_2)>0$, depending only on the constant $C_2(g_1,g_2)$ and the hyperbolicity constant $\Delta(g_2)$,  such that for all $\tilde v\in S^{g_1}\tilde M$ and for all $T>0$,
  $$
\left|d^{g_2}(\pi\tilde v,\pi g_1^T \tilde v)-\calB^{g_2}_{v_+^{g_1}}(\pi \tilde v,\pi g_1^T \tilde v)\right|=
\left|d^{g_2}(\pi \tilde v, \pi g_1^T \tilde v) - \int_0^T \calE^{g_1\to g_2}(g_1^t \tilde v)dt \right|\leq C_3(g_1,g_2)\,.
$$
  \end{lemm}

  \begin{proof}
 Let $\tilde v\in S^{g_1}\tilde M$ and $T>0$ be fixed. 
We write $x = \pi \tilde v$, $x_T = \pi g_1^T \tilde v$, 
%and let $y_T$ be the projection of $x_T$ on the $g_2$-geodesic $(x,v_+^{g_1})^{g_2}$ 
and $z_T$ be the intersection between 
the $g_2$-geodesic $(x,v_+^{g_1})^{g_2}$ and the $g_2$-horosphere
centered at $v_+^{g_1}$ passing thorough $x_T$. 

We will need at several occasions the following estimate. 
\begin{fact}\label{fact} With above notations, $d^{g_2}(x_T,z_T)\le 2C_2(g_1,g_2)+ \Delta(g_2)$. 
\end{fact}

Let us first prove this fact. 
Consider the $g_2$-geodesic triangle $x,x_T,v_+^{g_1}$ and its interior triangle, 
say $p\in (x_T,v_+^{g_1}),q\in (x,x_T),r\in (x,v_+^{g_1})$. 
\begin{figure}[ht!]\label{fig:lemme26}
\begin{center}
 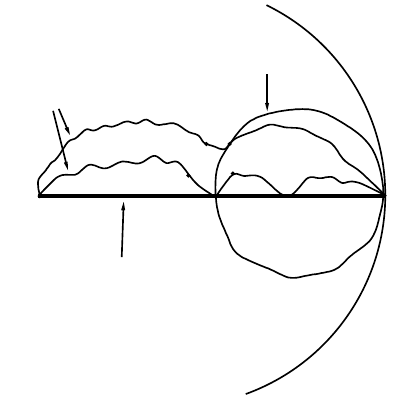 
\caption{Proof of lemma \ref{lem:GeodStretch}} 
\end{center}
\end{figure} 

Then   by definition of $z_T$, $d^{g_2}(z_T,r)=d^{g_2}(x_T,q)$, 
so that $d^{g_2}(x_T,z_T)\le 2d^{g_2}(x_T,q)+d^{g_2}(q,r)$.  
Now, the definition of $(p,q,r)$ implies $d^{g_2}(x_T,q)\le d^{g_2}(x_T,(x,v_+^{g_1}))\le C_2(g_1,g_2)$. 
The fact follows. \\

 By  definition (\ref{eq:GeodStretch3}),
 $$
\int_0^T \calE^{g_1\to g_2}(g_1^t \tilde v) dt =
 \calB^{g_2}_{v_+^{g_1}}(\pi \tilde v, \pi g_1^T \tilde v) = d^{g_2} (x, z_T)\,.
$$
Thanks to the above fact, we get
$$
|d^{g_2}(x,x_T)-d^{g_2}(x,z_T)|\le d^{g_2}(x_T,z_T)\le 2C_2(g_1,g_2)+\Delta(g_2)\,.
$$
 The result of the lemma follows, with $C_3(g_1,g_2)=2C_2(g_1,g_2)+\Delta(g_2)$. 
   \end{proof}

  \begin{coro}\label{coro:GeodStretch}
  Let $m$ be an ergodic probability measure on $S^{g_1}M$, invariant by the geodesic flow $(g_1^t)$. Then 
  $$
\int_{S^{g_1}M} e^{g_1\to g_2}(v) dm(v) = \int_{S^{g_1}M} \calE^{g_1\to g_2}(v)dm(v)\,.
$$
  Moreover, for $m$-almost every $v\in S^{g_1}M$ and all lifts $\tilde v\in S^{g_1}\tilde M$ of $v$,
$$
\lim_{T\to +\infty} \frac{d^{g_2}(\pi \tilde v, \pi g_1^T \tilde v)}{T} = 
\lim_{T\to +\infty} \frac{1}{T} \int_0^T \calE^{g_1\to g_2}(g^t_1 \tilde v)dt
= \int_{S^{g_1}M} \calE^{g_1\to g_2}(w) dm(w)\,.
$$
  \end{coro}
  
  \begin{proof}
  It follows from the previous lemma that for all $\epsilon>0$, 
there exists $T_0>0$ such that for all $T\geq T_0$ and all $\tilde v\in S^{g_1}\tilde M$,
  $$
\frac{1}{T} \left| d^{g_2}(\pi \tilde v, \pi g_1^T \tilde v) - \int_0^T \calE^{g_1\to g_2}(g^t_1 \tilde v)dt\right| \leq \epsilon.
$$
  Since $m_\mu^{g_1}$ is ergodic, for $m_\mu^{g_1}$-almost all vector $v\in S^{g_1}$,
    $$
\int_{S^{g_1}M} \calE^{g_1\to g_2}(v) dm(v) = 
\lim_{T\to +\infty} \frac{1}{T} \int_0^T \calE^{g_1\to g_2}(g^t_1 \tilde v)dt
$$
    and
      $$
\int_{S^{g_1}M} e^{g_1\to g_2}(v) dm(v) =
 \lim_{T\to +\infty} \frac{d^{g_2}(\pi \tilde v, \pi g_1^T \tilde v)}{T},
$$
    which concludes the proof of the corollary.
  \end{proof}

 Let us emphasize the fact that the measures which we will consider will usually have finite mass, 
but may not be probability measures. We will denote by $\|m\|$ 
the mass of a finite measure $m$ on $TM$. 
 
 \begin{defi}[Geodesic stretch with respect to a geodesic current]
  Let $\mu$ be a geodesic current on $\bd^2 \tilde M$ such that $m^g_\mu$ is finite. 
We will call \emph{(average) geodesic stretch of $g_2$ relative to $g_1$ with respect to $\mu$} the quantity
$$
I_\mu(g_1, g_2)\, = \,
\frac{1}{\|m_\mu^{g_1}\|} \int_{S^{g_1} M} \calE^{g_1\to g_2}(v)dm_\mu^{g_1}(v)\, = \,
\frac{1}{\|m_\mu^{g_1}\|} \int_{S^{g_1} M} e^{g_1\to g_2}(v)dm_\mu^{g_1}(v)\,.
$$
 \end{defi}

  By Corollary \ref{coro:GeodStretch}, $I_\mu(g_1, g_2)$ coincides with the definition of the geodesic stretch studied in \cite{Kni95} (note that Knieper only considers invariant \emph{probability} measures).

 When $(M,g)$ has finite volume and $\mu$ is the Liouville geodesic current of $g_1$, then 
  $$
I_\mu(g_1, g_2). \mbox{Vol}(S^{g_1}M) = i(g_1,g_2),
$$ 
  where  $i(g_1, g_2)$ is the \emph{intersection} between the metrics $g_1$ and $g_2$ studied in \cite{FF93}.  

It follows from the definition that for all geodesic current $\mu$ such that $m^g_\mu$ is finite, $\ds I_\mu(g_1, g_1) = 1$.

\begin{rema}[Geodesic stretches and Thurston metric]\label{rem:Thurston} \rm Given two negatively curved metrics $g_1$ and $g_2$ on a compact surface $S$, the Thurston distance 
$d_{Th}(g_1,g_2)$ is defined as the supremum over all periodic orbits of the ratios of their lengths\,:
$$d_{Th}(g_1,g_2)=\sup_{\gamma}\left( \frac{\ell^{g_2}(\gamma)}{\ell^{g_1}(\gamma)},\frac{\ell^{g_1}(\gamma)}{\ell^{g_2}(\gamma)}\right)\,.
$$
With our notations, this distance could also be defined as the following supremum
$$
d_{Th}(g_1,g_2)=\sup_{\mu}\left(I_\mu(g_1,g_2),I_\mu(g_2,g_1)\right)
$$
over all currents $\mu$ associated to ergodic measures. 
Indeed,  considering periodic measures immediately shows that Thurston distance is smaller than the above supremum. In the other direction, the density of periodic measures in the set of ergodic measures gives the above equality. 
\end{rema}

%%%%%%%%%%%%%%%%%%%%%%%%%%%%%%%%%%%%%%%%%%%%%%%%%%%%%%%%%%%%%%%%%%%%

\subsection{Morse correspondances and geodesic stretches}\label{subsec:Morse}

To compare dynamics of the geodesic flows on $S^{g_1}M$ and $S^{g_2}M$, 
it is natural to   consider their dynamics modulo the $\Gamma$-action on 
$S^{g_1}\tilde M$ and $S^{g_2}\tilde M$. Hopf coordinates are a good motivation to consider the   map 
$$
\widetilde{\Phi}^{g_1\to g_2}:=(H^{g_2})^{-1}\circ H^{g_1} : 
S^{g_1}\tilde M \to S^{g_2} \tilde M \,.
$$
It is a H\"older homeomorphism, but it is unfortunately not $\Gamma$-equivariant, as both $\Gamma$-actions on each unit tangent bundle $S^{g_i}\tilde M$ are different. 
In other words, as said earlier, on $\bd^2\tilde M\times \bbR$,
 these $\Gamma$-actions involve different cocycles on the $\bbR$ component. 

Despite its non-invariance, this map is sometimes useful, because it has the nice property to commute with both geodesic flows. But we need to find another map from 
$S^{g_1}\tilde M$ to $S^{g_2}\tilde M$ which will be $\Gamma$-
equivariant. We proceed  as follows. 
For all $v\in S^{g_1}\tilde{M}$, let $w=\tilde\Psi^{g_1\to g_2}(v)$ be the unique vector in $S^{g_2}\tilde{M}$ on the $g_2$-geodesic joining $v^{g_1}_-$ to $v^{g_1}_+$ satisfying $\calB^{g_2}_{v_+^{g_1}}(\pi(v),\pi(w))=0$.

\begin{figure}[ht!]\label{fig:morse}
\begin{center}
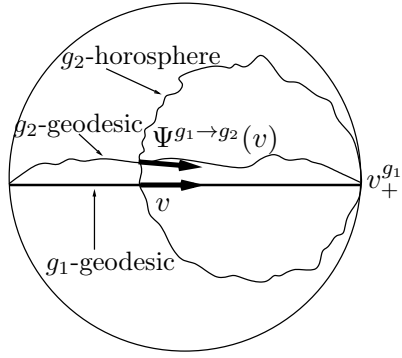 
\caption{Morse correspondance} 
\end{center}
\end{figure} 

\begin{lemm} The map $\Psi^{g_1\to g_2}$ is H\"older continuous. 
Moreover, for all $v\in S^{g_1}\tilde{M}$, we have 
$$d^{g_2}(v,\Psi^{g_1\to g_2}(v))\le C_3(g_1,g_2)\,,
$$ where $C_3(g_1,g_2)$ is the constant given by Lemma 
\ref{lem:GeodStretch}.

\end{lemm}
\begin{proof} It is H\"older continuous as composition of the maps $\Phi^{g_1\to g_2}$ and 
some time $g_2^t$ of the geodesic flow, with $t=t(v)$ depending H\"older-continuously of $v$.

The bound on $d^{g_2}(v,\Psi^{g_1\to g_2}(v))$ has already been proved in Fact \ref{fact}. 
\end{proof}

By construction, the correspondance $\tilde\Psi^{g_1\to g_2}$ is $\Gamma$-invariant. 
We denote by $\Psi^{g_1\to g_2}$ the induced map from $S^{g_1}M$ to $S^{g_2}M$. 
It is a homeomorphism homotopic to identity sending $(g_1^t)$-orbits to $(g_2^t)$-orbits, i.e. a {\em $(g_1,g_2)$-Morse correspondance} in the sense of \cite{FF93}. 

%This Morse correspondance $\Phi^{g_1\to g_2}$ commutes with the geodesic flow. 
%As a consequence, it  is not sensitive to geodesic stretch. 
%It will be more   convenient to understand the geodesic stretch as an infinitesimal reparametrization of the flow.  

By definition of both correspondances, the following lemma holds. 
 It says that the geodesic flows $(g_1^t)$ and $(g_2^s)$ on the unit tangent bundles $S^{g_i}\tilde M$ are conjugated by $\Phi^{g_1\to g_2}$, and conjugated 
up to reparametrization by the Morse correspondance $\Psi^{g_1\to g_2}$. 

\begin{lemm}\label{lem:Morse-Psi}\label{lem:inverses-Morse}  With the above notations, 
we have for all $v\in S^{g_1}\tilde{M}$ 
\begin{enumerate}
\item $\ds
\Phi^{g_1\to g_2}\circ g_1^t(v)=
g_2^t\circ\Phi^{g_1\to g_2}(v)\,.
$
\item $\ds
\Phi^{g_2\to g_1}=\left(\Phi^{g_1\to g_2}\right)^{-1}\,.
$
\item $\ds
\Psi^{g_1\to g_2}\circ g_1^t(v)=g_2^{s^{g_1\to g_2}(t,v)}\circ \Psi^{g_1\to g_2}(v),
$
with $\ds s^{g_1\to g_2}(t,v)=\calB_{v_+^{g_1}}^{g_2}(\pi(v),\pi(g_1^t v))$. 
\item 
$\ds \Psi^{g_1\to g_2}(v)=g_2^{\tau^{g_1\to g_2}(v)}\circ \Phi^{g_1\to g_2}(v)\,$, 
with $$\tau^{g_1\to g_2}(v)= \calB^{g_2}_{v_+^{g_1}}(o,\pi(v))-\calB^{g_1}_{v_+^{g_1}}(o,\pi(v)).$$
\item $\ds
\Psi^{g_2\to g_1}\circ \Psi^{g_1\to g_2}(v)=g_1^{\sigma^{g_1\to g_2}(v)}(v)\,,
$
with $\displaystyle\sigma^{g_1\to g_2}(v)=\calB_{v^{g_1}_+}^{g_1}\left(\pi(v),\pi(\Psi^{g_1\to g_2} v)\right)$.
\end{enumerate}
\end{lemm}
%$\Phi^{g_1\to g_2} = (H^{g_2})^{-1}\circ H^{g_1}$. 

Let us emphasize that $\Phi^{g_1\to g_2}$ and its inverse are not $\Gamma$-invariant, $\Psi^{g_1\to g_2}$ and its inverse are $\Gamma$-invariant, the map $\tau^{g_1\to g_2}$ is not $\Gamma$-invariant, whereas $\sigma^{g_1\to g_2}$ and the cocycle $s^{g_1\to g_2}(t,v)$ are $\Gamma$-invariant.  

\begin{proof} 
The fact that $\Phi^{g_1\to g_2}$ commutes with the geodesic flows 
of $g_1$ and $g_2$ is immediate by definition of Hopf coordinates. 
The property about its inverse is also obvious. 

By definition of $\Psi^{g_1\to g_2}$, the vectors $\Psi^{g_1\to g_2}(g_1^t v)$, for $t\in\bbR$,
 all lie on the $g_2$-geodesic joining $v_-^{g_1}$ to $v_+^{g_2}$. 
The only question is to compute 
$$
s^{g_1\to g_2}(t,v)=
\calB^{g_2}_{v_+^{g_1}}(\pi(\Psi^{g_1\to g_2}(v)),\pi(\Psi^{g_1\to g_2}(g_1^tv)))\,.
$$ 
By definition of $\Psi^{g_1\to g_2}$,
$$
\calB^{g_2}_{v_+^{g_1}}(\pi\Psi^{g_1\to g_2}(g_1^t v),\pi(g_1^tv))=0=
\calB^{g_2}_{v_+^{g_1}}(\pi \Psi^{g_1\to g_2}(v),\pi(v))\,.
$$ 
Using the cocycle properties of $\calB^{g_2}_{v_+^{g_1}}$, 
we deduce immediately that $s^{g_1\to g_2}(t,v)$ is the algebraic 
$g_2$-distance $\calB^{g_2}_{v_+^{g_1}}(\pi(v),\pi(g_1^t v))$.%\footnote{Trop long??}

The next   affirmation follows from the computation 
$$
\tau^{g_1\to g_2}(v)=\calB^{g_2}_{v_+^{g_1}}(\Phi^{g_1\to g_2}(v),\Psi^{g_1\to g_2}(v))=
\calB^{g_2}_{v_+^{g_1}}(o,\pi(v))-\calB^{g_2}_{v_+^{g_1}}(o,\Phi^{g_1\to g_2}(v))
$$
$$=
\calB^{g_2}_{v_+^{g_1}}(o,\pi(v))-\calB^{g_1}_{v_+^{g_1}}(o,v)\,.
$$ 
The last statement follows easily from the previous one. 
\end{proof} 

 %Let $v\in S^{g_1}\tilde M$, and $(g_1^tv)_{t\in \bbR}$ be 
%the associated $g_1$-geodesic. 
%Let $w = \Phi^{g_1\to g_2}(v)\in S^{g_2}\tilde M$. 
%By definition of the Hopf coordinates, $w\in S^{g_2}\tilde M$ 
%is the unique vector with 
%$$
%w^{g_2}_\pm = v^{g_1}_\pm = \xi_\pm \quad \mbox{and} \quad \calB^{g_1}_{\xi_+}(o, \pi v) =
% \calB^{g_2}_{\xi_+}(o, \pi w)\,.
%$$ 

%%%%%%%%%%%%%%%%%%%%%%%%%%%%%%%%%%%%%%%%%%%%%%%%%%%%%%%%%%%%%%%%%ù

\subsection{Change of mass}

We will need the following variant of  Lemma \ref{lem:GeodStretch}, 
which shows once more that $\calE^{g_1\to g_2}$ behaves 
asymptotically as the infinitesimal reparametrization of the flow 
given by Morse correspondance $\Psi^{g_1\to g_2} : S^{g_1}\tilde M \to S^{g_2}\tilde M$.

\begin{prop}\label{prop:Change-Var-Flow}
 Let $G : S^{g_2}M\to \bbR$ be a continuous map and $\tilde G : S^{g_2}\tilde{M}\to \bbR$ be its ($\Gamma$-invariant) lift to $S^{g_2}\tilde M$. 
Then for all $v\in S^{g_1} \tilde M$, $T\ge 0$, 
and $w=\Psi^{g_1\to g_2}(v)$, we have
$$
\int_{0}^{s^{g_1\to g_2}(T,v)} \tilde G(g_2^s w)\,ds=
\int_0^T G\circ\Psi^{g_1\to g_2}(g_1^t v)\times \calE^{g_1\to g_2}(g_1^t v)\,dt\,, 
$$
with $s^{g_1\to g_2}(T,v)=\calB_{v_+^{g_1}}^{g_2}(\pi(v),\pi(g_1^T v))$ as in Lemma \ref{lem:Morse-Psi}. 

If moreover $G$ is bounded, then there exists $C=C(G,g_1,g_2)$ such that for 
all $v\in S^{g_1} \tilde M$, $T\ge 0$, and $w=\Psi^{g_1\to g_2}(v)$, we have
$$
\left|\int_{0}^{d^{g_2}(v,g_1^Tv)} \tilde G(g_2^sw),ds-
\int_0^{T} \tilde G\circ \Psi^{g_1\to g_2}(g_1^t v)\times\calE^{g_1\to g_2}(g^t_1 v)\,dt  \right|
\le C\,.
$$ 

If $G$ is not bounded,  then for all compact sets $K\subset S^{g_1}M$ 
there exists another constant $C'=C'(G,K,g_1,g_2)$ such that for all  $v\in S^{g_1}\tilde M$ and $T\in\bbR$ such that 
 both $v$ and $g_1^T v$ project inside $K$,  
$\tilde K\subset S^{g_1} \tilde M$, we have
$$
\left|\int_{0}^{d^{g_2}(v,g_1^Tv)} \tilde G(g_2^sw),ds-
\int_0^{T} \tilde G\circ \Psi^{g_1\to g_2}(g_1^t v)\times\calE^{g_1\to g_2}(g^t_1 v)\,dt  \right|
\le C'\,.
$$ 
\end{prop}

The geodesic stretch $\calE^{g_1\to g_2}$ can therefore be understood 
as the instantaneous reparametrization of the flow $(g_1^t)$ in the correspondance $\Psi^{g_1\to g_2}$. 

\begin{proof} 
The first equality is a simple change of variable using Lemma \ref{lem:Morse-Psi}. 
The second follows using Lemma \ref{lem:GeodStretch} and the fact that $F$ is bounded.  Indeed, 
$$
 \left|\int_{0}^{d^{g_2}(v,g_1^Tv)} G(g_2^sw),ds-\int_0^{T} F\circ \Psi^{g_1\to g_2}(g_1^t v)\times\calE^{g_1\to g_2}(g^t_1 v)\,dt \right| 
$$  
\begin{eqnarray*}
&=& \left|\int_{s(T,v)}^{d^{g_2}(v,g_1^Tv)}G(g_2^sw)\,ds\right|\\
&\le & \|G\|_\infty\times \left|d^{g_2}(v,g_1^Tv)-\calB^{g_2}_{v_+^{g_1}}(\pi(v),\pi(g_1^tv))\right|\\
&=& \|G\|_\infty\times \left|d^{g_2}(v,g_1^Tv)-\int_0^T \calE^{g_1\to g_2}(g_1^tv)\,dt\right|\\
&\le & C_1\|G\|_\infty\,.
\end{eqnarray*}

The last assertion is a variation on the second one. If $v$ and $g_1^Tv$ are in a compact set $K$, for any parameter $s$ such that $|s|\le C_3(g_1,g_2)$, $g_1^{T\pm s}v$ belongs to the $C_3(g_1,g_2)$-neighbourhood of $K$, on which $G$ is bounded. The above computation therefore applies verbatim. 
\end{proof}

\begin{rema}\rm Proposition \ref{exact-change-variable} follows immediately: 
given any $m_\mu^{g_2}$-measurable map $G : S^{g_2}M\to \bbR$, 
 the map $G\circ \Psi^{g_1\to g_2}$ is $m_\mu^{g_1}$-measurable and
$G$ on $S^{g_2}M$, and we have
$$
\int_{S^{g_2}M} G\,dm^{g_2}_\mu=\int_{S^{g_1}M} G\circ \Psi^{g_1\to g_2} \times \calE^{g_1\to g_2} \,dm^{g_1}_\mu\,.
$$
\end{rema}

The corollary below follows immediately from the above Remark. It gives a nice interpretation of the 
  geodesic stretch $I_\mu(g_1,g_2)$.
  
%Let us conclude this section by a Proposition whose proof follows exactly from the above arguments. 

%\begin{prop}\label{exact-change-variable} Let $\mu$ be a geodesic current such that 
%$m_\mu^{g_1}$ is ergodic and has finite total mass, 
% denoted by $\|m_\mu^{g_1}\|$. 
%Let $F:S^{g_2}M\to \bbR$ be a bounded continuous function. 
%Then
%$$
%\int_{S^{g_2}M} F\,dm_\mu^{g_2}  = 
%\int_{S^{g_1}M} F\circ\Psi^{g_1\to g_2}\,\frac{dm^{g_1}_\mu}{\|m^{g_1}_\mu\|}\times\int_{S^{g_1}M} F\,d m_\mu^{g_1} \,.
%$$
%\end{prop}
%\begin{proof} The proof follows exactly the lines of the above proof of Proposition
%\ref{prop:transformation-mass}. 
%The only change is to replace $1$ by $F$ in the statement and the proof of Lemma \ref{lem:lengths}.  
%The statement of this lemma becomes 
%\begin{eqnarray*}\label{lengths-ponderes}
%\frac{1}{\ell^{g_1}(\gamma_k)}\int_{\gamma_k^{g_2}}F\circ\Psi^{g_1\to g_2} d\ell_{\gamma_k}^{g_2}
%\to \int F\times \calE^{g_1\to g_2} \frac{dm^{g_1}_\mu}{\|m^{g_1}_\mu\|}\,.
%\end{eqnarray*}
%\end{proof}

\begin{coro}[Mass transformation law]\label{prop:transformation-mass}
Let $\mu$ be a geodesic current such that $m_\mu^{g_1}$ is ergodic and has finite total mass, 
 denoted by $\|m_\mu^{g_1}\|$. 
Then
$$
\|m_\mu^{g_2}\| = 
I_\mu(g_1, g_2)\times\| m_\mu^{g_1}\|\,.
$$
   In particular $m_\mu^{g_1}$ has finite mass if and only if $m_\mu^{g_2}$ has finite mass. 
Moreover, when it is the case,
$$
I_\mu(g_1, g_2) = 
\frac{1}{I_\mu(g_2, g_1)} = 
\frac{\|m_\mu^{g_2}\|}{\|m_\mu^{g_1}\|}\,.
$$
  \end{coro}

 \begin{rema}\label{rem:ChangeMass} \rm  
  The previous formula is very natural if $I_\mu(g_2, g_1)$ is interpreted as the average dilation of the reparametrization of the flow via the Morse correspondance $\Psi^{g_1\to g_2}$. 
Indeed, in the case where $(g_1^t)$ and $(g_2^t)$ are suspension flows over a (fixed) compact basis for distinct ceiling functions, the above formula is well known \cite{Abramov}.
  \end{rema}

%%%%%%%%%%%%%%%%%%%%%%%%%%%%%%%%%%%%%%%%%%%%%%%%%%%%%%%%%%%%ù
\subsection{Periodic orbits and geodesic stretch}

In this section we relate geodesic stretch and lengths of periodic orbits. 
The results will not be useful in the sequel of the paper, but are enlightening about the geodesic stretch. 
 
 For $i = 1, 2$, for any hyperbolic element $\gamma\in \Gamma$, 
let $\gamma^{g_i}$ be the closed $g_i$-geodesic associated to the conjugacy class of $\gamma$. 
Let $\ell^{g_i}(\gamma)$ be its $g_i$-length, 
and $d\ell_{\gamma}^{g_i}$ be the Lebesgue measure along the geodesic $\gamma^{g_i}$. 
Observe that, up to normalizing constants, the periodic measure $d\ell_{\gamma}^{g_i}$, $i=1,2$, induce
the same current at infinity. 
 
 Since $m_\mu^{g_1}$ is finite and ergodic, there exists a sequence $(\gamma_k)_{k\in \bbN}$ 
of hyperbolic elements such that in the weak topology,
 $$
\lim_{k\to \infty} \frac{d\ell_{\gamma_k^{g_1}}}{\ell^{g_1}(\gamma_k)} = 
\frac{m_\mu^{g_1}}{\|m_\mu^{g_1}\|}\,,
$$
see for instance \cite[Lemma 2.2]{CS10}. 
This convergence holds a priori in the dual of continuous functions with compact support. 
But as all measures involved above are probability measures, 
this convergence also holds in the dual of bounded continuous functions of $S^{g_1}M$.
  
We can moreover suppose that 
$\ds \lim_{k\to \infty} \ell^{g_1}(\gamma_k) = +\infty$.
%The proof will follow easily from the two lemmas below. 

The following proposition shows that the same happens on $S^{g_2}M$, 
and that the ratio of lengths of periodic orbits in both metrics allows to recover the geodesic stretch. 

\begin{prop} Let $(M,g_i)$, $i=1,2$, be two admissible Riemannian structures with pinched negative curvature. 
Let $\mu$ be a geodesic current such that both measures $m^{g_i}_\mu$ are finite. 
Let $(\gamma_k)$ be a sequence of hyperbolic elements such that 
$\frac{d\ell^{g_1}_{\gamma_k}}{\ell^{g_1}(\gamma_k)}$ converges weakly to 
 $\frac{m^{g_1}_\mu}{\|m^{g_1}_\mu\|}$ in the dual of bounded continuous functions. 
Then $\frac{d\ell^{g_2}_{\gamma_k}}{\ell^{g_2}(\gamma_k)}$ converges weakly to  
$\frac{m^{g_2}_\mu}{\|m^{g_2}_\mu\|}$ in the dual of bounded continuous functions. 

Moreover, the ratios  of lengths satisfy 
$$
\lim_{k\to +\infty}\frac{\ell^{g_2}(\gamma_k)}{\ell^{g_1}(\gamma_k)}
 = I_\mu(g
_1,g_2)\,.
$$
\end{prop}

The proof is separated in two lemmas. The first one asserts that viewed on $S^{g_2}M$, 
the sequence of periodic probability measures associated to $(\gamma_k)$  also converges to 
$\frac{m_\mu^{g_2}}{\|m_\mu^{g_2}\|}$ in the dual of bounded continuous functions. 
The second says that the ratio of lengths $\ell^{g_2}(\gamma_k)/\ell^{g_1}(\gamma_k)$ converges to the
average geodesic stretch $I_\mu(g_1,g_2)$.

 \begin{lemm}\label{lem:narrow}
 With the previous notations, for the same sequence $(\gamma_k)$, 
in the dual of continuous bounded functions of $S^{g_2} M$,
  $$
\lim_{k\to \infty} \frac{d\ell_{\gamma_k}^{g_2}}{\ell^{g_2}(\gamma_k) } = 
\frac{m_\mu^{g_2}}{\|m_\mu^{g_2}\|}\,.
$$
 \end{lemm}
 
 \begin{proof} 
First, as the sequence of probability measures $\frac{d\ell_{\gamma_k}^{g_1}}{\ell^{g_1}(\gamma_k)}$ converges 
to the probability measure $\frac{m_\mu^{g_1}}{\|m_\mu^{g_1}\|}$, 
the $\Gamma$-invariant lift of 
$\frac{d\ell_{\gamma_k}^{g_1}}{\ell^{g_1}(\gamma_k)}$ to $S^{g_1}\tilde M$ 
converges in the dual of continuous functions with compact support towards
$\frac{\tilde m_\mu^{g_1}}{\|m_\mu^{g_1}\|}$. 
Using Hopf coordinates, we deduce that 
 the geodesic current on $\partial^2\widetilde{M}$ associated through $H^{g_1}$ to
 $\frac{d\ell_{\gamma_k^{g_1}}}{\ell^{g_1}(\gamma_k)}$ converges weakly 
(in the dual of continuous functions with compact support) to $\mu$. 
Using the same reasoning in the other direction, 
we obtain that  the sequence of probability measures $\frac{d\ell_{\gamma_k}^{g_2}}{\ell^{g_2}(\gamma_k)}$ converges weakly 
(in the dual of continuous functions with compact support) to some multiple of $m^{g_2}_\mu/\|m^{g_2}_\mu\|$. 

It is not exactly the desired result. To get the convergence towards the probability measure $m^{g_2}_\mu/\|m^{g_2}_\mu\|$, and in the dual of bounded continuous functions, we need to avoid a possible loss of mass at infinity. 
To establish this convergence, it is necessary and sufficient to prove that 
$\frac{d\ell_{\gamma_k}^{g_2}}{\ell^{g_2}(\gamma_k)}$ does not diverge.
In other words, we want to check that 
for all $\varepsilon>0$, there exists a compact set $\mathcal{K}_\varepsilon\subset S^{g_2}M$, 
such that for all $k\ge 0$ large enough, 
$$
\frac{d\ell_{\gamma_k}^{g_2}}{\ell^{g_2}(\gamma_k)}(\mathcal{K}_\varepsilon)\ge 1-\varepsilon\,.
$$
It  follows easily from the fact that there exists a constant $C=C(g_1,g_2)$ such that
 any $g_2$-geodesic of $\tilde M$ stays in a $C(g_1,g_2)$-neighbourhood of the $g_1$-geodesic with same 
endpoints at infinity. Let us write the detail of the argument. 

Choose first some $\varepsilon>0$, and some compact set $K_1\subset M$ such that 
$\frac{m^{g_1}_\mu(S^{g_1}K_1)}{\|m^{g_1}_\mu\|}\ge 1-\varepsilon/2$. 
By convergence of $\frac{d\ell_{\gamma_k}^{g_1}}{\ell^{g_1}(\gamma_k)}$, for all $k\ge k_0$ large enough, 
we also have $\frac{\ell^{g_1}_{\gamma_k}(S^{g_1}K_1)}{\ell^{g_1}(\gamma_k)}\ge 1-\varepsilon$. 

Now, choose a relatively compact preimage $\tilde K_1\in \tilde M$, its $g_2$-convex closure $\tilde K_2$
%containing a $C(g_1,g_2)$-neighbourhood of $\tilde K_1$ for both metrics $g_1$ and $g_2$, 
and $\tilde K_3\supset \tilde K_2$ a larger compact convex set of $\tilde M$ 
containing a $2C(g_1,g_2)$-neighbourhood of $\tilde K_2$ for both metrics $g_1$ and $g_2$.

Consider a lift $\tilde{\gamma}_k^{g_1}$ of the $g_1$-geodesic $\gamma_k^{g_1}$ which intersects $\tilde K_1$, 
and the associated lift $\tilde{\gamma}_k^{g_2}$
of the $g_2$-geodesic $\gamma_k^{g_2}$, at distance at most $C(g_1,g_2)$ from $\tilde{\gamma}_k^{g_1}$. 
Let $a,b$ be two points on $\tilde \gamma_k^{g_2}$ such that the length 
$\ell^{g_2}_{\gamma_k}\left((a,b)\right)=\ell^{g_2}(\gamma_k)$. 
We want to estimate the proportion of $g_2$-length of $[a,b]$ outside $\Gamma.\mathcal{K}=\Gamma.S^{g_2}K_3$.

\begin{figure}[ht!]\label{fig:periodic}
\begin{center}
 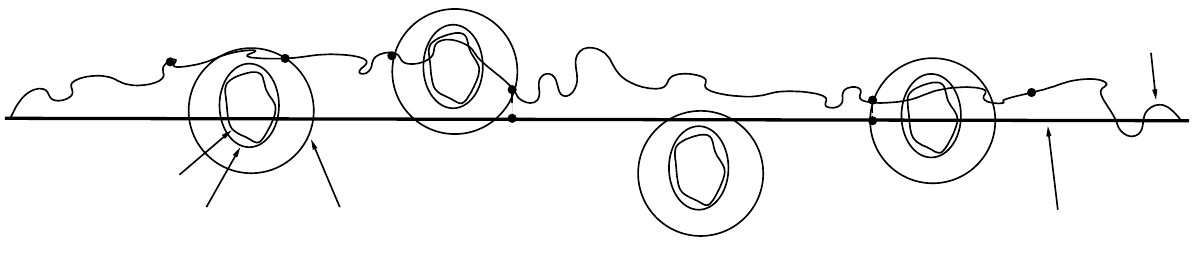 
\caption{Proof of Lemma \ref{lem:narrow}} 
\end{center}
\end{figure} 

By convexity of $\tilde K_3$, we can write $(a,b)\cap(\Gamma.\tilde K_3)^c$ 
as the disjoint union $\sqcup (a_i,b_i)$ of finitely many intervals. 
Thus, we have to show
that 
$$
\frac{\ell^{g_2}_{\gamma_k}(\mathcal{K}^c)}{\ell^{g_2}(\gamma_k)}=
\frac{\sum_i\ell^{g_2}_{\gamma_k}(a_i,b_i)}{\ell^{g_2}(\gamma_k)}\le \varepsilon\,.
$$

Choose two points $c_i$ and $d_i$ on $\widetilde{\gamma}_k^{g_1}$ 
whose projections (for the metric $g_2$)
on the $g_2$-geodesic $\widetilde{\gamma}_k^{g_2}$ are exactly $a_i$ and $b_i$. 
Such  points are not necessarily unique but  always exist: take $c_i$ 
in the intersection of $\tilde\gamma_k^{g_1}$ with 
the hyperplane orthogonal to $\tilde\gamma_k^{g_2}$ at $a_i$. 
Denote by $(c_i,d_i)$ the $g_1$-geodesic segment on $\tilde\gamma_k^{g_1}$, 
and let $\lambda>0$ be such that
$\frac{1}{\lambda}g_2\le g_1\le \lambda g_2$. 
We have 
$$
\ell^{g_2}_{\gamma_k}(a_i,b_i)=d^{g_2}(a_i,b_i)\le 
d^{g_2}(c_i,d_i)\le \ell^{g_2}(c_i,d_i)\le \sqrt{\lambda}\ell^{g_1}(c_i,d_i)\,.
$$
We deduce that
 $$
\frac{\ell^{g_2}_{\gamma_k}(\mathcal{K}^c)}{\ell^{g_2}(\gamma_k)}\le 
\lambda \sum_i\frac{\ell^{g_1}(c_i,d_i)}{\ell^{g_1}(\gamma_k)}\le 
\lambda \frac{\ell^{g_1}((S^{g_1} \Gamma.\tilde K_1)^c)\cap\widetilde{\gamma}_k^{g_1})}{\ell^{g_1}(\gamma_k)}\,,
$$
the last inequality coming from the fact that,
 as $a_i$ and $b_i$ are in the boundary of $\Gamma.\tilde K_3$, and 
$c_i$ and $d_i$ are at distance at most $C(g_1,g_2)$ resp. from $a_i$ and $d_i$, 
they cannot belong to $\Gamma.\tilde K_2$, 
so that the  segment $(c_i,d_i)$ does not intersect $\Gamma.\tilde K_1$. 
This proves that 
$$
\frac{\ell^{g_2}_{\gamma_k}(\mathcal{K}^c)}{\ell^{g_2}(\gamma_k)}\le \lambda\varepsilon\,, 
$$
which concludes the proof (up to changing $\varepsilon$ in $\varepsilon/\lambda$). 
 \end{proof}

 Moreover, the lengths $\ell^{g_1}(\gamma_k)$ and $\ell^{g_2}(\gamma_k)$ are related as follows.
 
 \begin{lemm}\label{lem:lengths}
 With the previous notations,   for the same sequence $(\gamma_k)$, 
 $$
\lim_{k\to +\infty} \frac{\ell^{g_2}(\gamma_k)}{\ell^{g_1}(\gamma_k)} = 
I_\mu(g_1, g_2)\,.
$$
 \end{lemm}
 
 \begin{proof}
For all $k\in \bbN$, let $v^{g_1}_k$ (resp. $v^{g_2}_k$) be a tangent vector to $\gamma^{g_1}_k$  
(resp.  $\gamma^{g_2}_k$) such that $\ds d^{g_2}(\pi v^{g_1}_k, \pi v^{g_2}_k)\leq C(g_1,g_2)$, 
where $C(g_1,g_2)>0$ is the  constant coming from (\ref{eq:QIsom}). 
Let $\tilde v^{g_1}_k\in S^{g_1}\tilde M$ and $\tilde v^{g_2}_k\in S^{g_2}\tilde M$ 
be lifts of $v^{g_1}_k$ and $v^{g_2}_k$ such that again, 
$\ds d^{g_2}(\pi \tilde v^{g_1}_k, \pi \tilde v^{g_2}_k)\leq C(g_1,g_2)$. 
It follows from Proposition \ref{prop:Change-Var-Flow} applied to $F\equiv 1$
 that there exists $C_1>0$, only depending on $C(g_1,g_2)$ and the bounds on the curvature, such that 
$$
 \left | \ell^{g_2}(\gamma_k) - \int_0^{\ell^{g_1}(\gamma_k)} \calE^{g_1\to g_2}(g_1^t \tilde v) dt \right|
 \leq C_1\,.
$$
Therefore,
$$
\left|\frac{\ell^{g_2}(\gamma_k)}{\ell^{g_1}(\gamma_k)} - 
\frac{1}{\ell^{g_1}(\gamma_k)} \int_0^{\ell^{g_1}(\gamma_k)} \calE^{g_1\to g_2}(g_1^t \tilde v) dt \right|
\leq \frac{C_1}{\ell^{g_1}(\gamma_k)}\,.
$$
By Lemma \ref{lem:narrow}, as $\calE^{g_1\to g_2}$ is bounded and continuous,  we know that 
$$ \frac{1}{\ell^{g_1}(\gamma_k)} \int_0^{\ell^{g_1}(\gamma_k)} \calE^{g_1\to g_2}(g_1^t \tilde v) dt \to I_\mu(g_1,g_2)\,,$$ so that the conclusion follows.  
 \end{proof}
 
% Let us finish the proof of the mass transformation law. 
%For all $\phi : TM\to \bbR$ bounded and continuous, 
%we have by the previous lemmas, and by definition of the Morse correspondance $\Phi^{g_1\to g_2}$, 
% \begin{eqnarray*}
%\frac{m_\mu^{g_2}(\phi)}{\|m_\mu^{g_2}\|} & =&
 %\frac{1}{\|m_\mu^{g_2}\|} m_\mu^{g_1}(\varphi\circ \Phi^{g_1\to g_2}) \\
%&=&\frac{\|m_\mu^{g_1}\|}{\|m_\mu^{g_2}\|}\lim_{k\to \infty} \frac{1}{\ell^{g_1}(\gamma_k)}\int_0^{\ell^{g_1}(\gamma_k)} \varphi\circ\Phi^{g_1\to g_2} d\ell^{g_1}_{\gamma_k}\\
%&=& \frac{\|m_\mu^{g_1}\|}{\|m_\mu^{g_2}\|}\lim_{k\to \infty} \frac{\ell^{g_2}(\gamma_k)}{\ell^{g_1}(\gamma_k)}
%\frac{1}{\ell^{g_2}(\gamma_k)}\int_0^{\ell^{g_2}(\gamma_k)} \varphi  d\ell^{g_2}_{\gamma_k}\\
%&=& \frac{\|m_\mu^{g_1}\|}{\|m_\mu^{g_2}\|}\times I_\mu(g_1,g_2) \times \frac{m_\mu^{g_2}(\phi)}{\|m_\mu^{g_2}\|}\,.
%\end{eqnarray*}
%The conclusion follows.  
% \end{proof}

%%%%%%%%%%%%%%%%%%%%%%%%%%%%%%%%%%%%%%%%%%%%%%%%%%%%%%%%%%%%%%
%%%%%%%%%%%%%%%%%%%%%%%%%%%%%%%%%%%%%%%%%%%%%%%%%%%%%%%%%%%%%%

\section{Entropy of finite measures}\label{sec:Entropy}

In this section, given two admissible metrics $g_1$ and $g_2$ as before, and a geodesic current
$\mu$ on $\partial^2\tilde M$, we wish 
to compare the entropies of the measures $m_\mu^{g_1}$ and $m_\mu^{g_2}$. 
Theorem \ref{entropy-transformation} establishes that their ratio is the average geodesic stretch between $g_1$ and $g_2$
w.r.t $\mu$, but in the reverse direction compared to the relation between their masses, which leads to 
Corollary \ref{coro:invariant}, 
which states that the product of the entropy of $m_\mu^{g_i}$ by its mass $\|m_\mu^{g_i}\|$ 
remains constant under an admissible change of metric. 

First, we will recall some definitions and relations between dynamical balls 
(subsection \ref{dyn-balls-shadows}).
 In subsection \ref{sub:entropies},  we compare two notions of entropy of a measure, 
the Kolmogorov-Sinai entropy and the local Brin-Katok entropy, recalling well and
 less known results of Brin-Katok and Riquelme. 
It allows us to prove Theorem \ref{entropy-transformation} 
and Corollary \ref{coro:invariant} in subsection \ref{sub:proof}.

%%%%%%%%%%%%%%%%%%%%%%%%%%%%%%%%%%%%%%%%%%%%%%%%%%%%%%%%%%%%%%ùùù
\subsection{Dynamical balls and shadows}\label{dyn-balls-shadows}

If $(\phi^t)$ is a continuous dynamical system on a metric space $(X,d)$, 
a dynamical ball is a ball for the dynamical distance 
$d_T(x,y)=\sup_{0\le t \le T} d(\phi^t x,\phi^t y)$. 

We will restrict ourselves to  geodesic flows associated 
to a Riemannian metric $g$ on $S^gM$. 
For such geometric dynamical systems, it is more convenient to work 
with the Riemannian distance induced by 
the metric $g$ on $M$ or $\tilde{M}$  instead of the distance coming 
from the Sasaki metric on $TM$ or $T\tilde{M}$. 
We refer to  \cite[p.70]{Bal95} and \cite[p.19-20]{PPS} 
for a discussion about the fact that it is 
the good thing to do in this case. 
%On $S^g \tilde M$, we will consider the distance 
%$$
%d^g(v,w) = \sup_{-1 \leq t \leq 0}d^g(\pi g^tv, \pi g^t w)\,.
%$$
%Since $g$ has pinched negative curvatures, this distance is bi-Lipschitz to the Riemannian distance given by the Sasaki metric on $S^g\tilde M$, see. The distance $d^g$ is well adapted to study the geodesic flow. It is $\Gamma$-invariant and induces which we will still denote by $d^g$ on $S^gM$.

For all $\e,T>0$ and $v\in S^g\tilde{M}$, we will call \emph{dynamical ball} 
of center $v$, diameter $\e$ and length $T$ the set
$$
B^{g}(v,T,\varepsilon) =
\{w\in S^g \tilde{M},\,d^g(\pi(g^t v),\pi(g^t w))\le \varepsilon,\,\textrm{for all } 0\le t\le T\}\,.
$$
Note that $B^{g}(v,0,\varepsilon)$ is the $\e$-ball with center $v$ for the distance $d^g$ defined above. 

\begin{rema}\label{erratum}\rm On the quotient, for $v\in S^g M$, one can either consider the quotient dynamical ball $B^g(v,T,\varepsilon)=p_\Gamma(B^g(\tilde{v},T,\varepsilon))$, $\tilde{v}$ being any lift of $v$ to $S^g\tilde{M}$. There is also a more dynamical definition, as 
$$
B^g_{dyn}(v,T,\epsilon) =
\{w\in S^g  M ,\,d^g(\pi(g^t v),\pi(g^t w))\le \varepsilon,\,\textrm{for all } 0\le t\le T\}\,.
$$
Of course, if $\tilde{v}\in S^g\tilde{M}$ and $v=p_\Gamma(\tilde{v})\in S^gM$, 
one has the obvious inclusion 
\begin{equation}\label{inclusion}
B^g(v,T,\varepsilon))=p_\Gamma(B^g(\tilde{v},T,\varepsilon)\subset B^g_{dyn}(v,T,\varepsilon)\,.
\end{equation}
One can easily see
 that this inclusion is an equality when the injectivity radius of $M$ 
is uniformly bounded from below, 
as soon as $\varepsilon$ is small enough.
 However, when the injectivity radius of $M$ is not bounded from below, 
one can build examples where this inclusion is not an equality \cite{Bellis, Velozo}. 

It turns out that in many cases, the most natural dynamical ball 
to consider is the small ball $p_\Gamma(B^g(\tilde{v},T,\varepsilon)$. 
Therefore, we will call it the {\em small dynamical ball} and denote it by 
$B^g(v,T,\epsilon)$. 

This problem  has not been emphasized in \cite{PPS}, where only these small dynamical balls are considered
(see \cite[3.15]{PPS}).  However, in various definitions of local entropies, the large dynamical balls have to be considered. 
\end{rema}

%Each $B^{g}(v,T,\varepsilon)$ is contained in the $\e$-neighbourhood of the geodesic
 %$(g^tv)_{0\leq t \leq T}$ for the distance $d^g$ defined above.

We will also need the following variant, for $v\in S^g\tilde M$ and $T,T'>0$\,:
$$
B^g(v;T,T',\e)=\{w\in S^g \tilde M, d^g(\pi(v),\pi(w))\le \e,\,\,\textrm{for all }  -T'\le t\le T \}\,.
$$
Observe that $B^{g}(v;T,T',\e)=g^{T'}(B^g(g^{-T'}v,T+T',\e))$. As mentioned in the above remark \ref{erratum}, we consider on $S^g\tilde{M}$ the small dynamical balls $B^g(v;T,T',\e)=p_\Gamma(B^g(\tilde{v};T,T',\e))$. \\

Recall the following well known fact in negative curvature.

\begin{lemm}\label{uniform-negative} Let $(M,g)$ be a manifold with pinched negative curvature. 
For all $0<a<b$, there exists a constant $c=c(a,b)>0$ such that for all
vectors $v,w\in S^g\widetilde{M}$, and all $T>2c$, if $d^g(\pi(g_t v),\pi(g_t w))\le b $ for
all $0<t<T$, then $d^g(g_t v, g_t w)\le a$ for all $c<t<T-c$.
\end{lemm}

\begin{proof} This is an exercise using standard comparison results. 
Note that the constant $c(a,b)$ also depends on the upper bound of the curvature. 
\end{proof}

\begin{lemm}\label{lem:dynamical-balls-inclusion}  
Let $(M,g)$ be a manifold with pinched negative curvature. 
For all $0<\e_1<\e_2$, there exists $C(g,\e_1,\e_2)>0$ such that
for all $v\in S^g\tilde{M}$ and $T, T'>0$, we have
$$
B^g(v;T+C(g,\e_1,\e_2),T'+C(g,\e_1,\e_2),\e_2)
\subset B^g(v;T, T',\e_1)\subset B^g(v;-T,T',\e_2)\,.
$$
\end{lemm}

\begin{proof} The right inclusion is obvious.
The left one comes from Lemma \ref{uniform-negative} above.
\end{proof}

The shadow $\calO^g_x(B^g(y,R))$ of the ball $B^g(y,R)$ 
viewed from $x$ w.r.t. the metric $g$
is the set of positive endpoints in $\partial\widetilde{M}$ of $g$-geodesic rays starting from $x$
and intersecting $B^g(y,R)$. \\

Recall Lemma 3.17 from \cite{PPS}.
\begin{lemm}[\cite{PPS}] \label{dynamical-ball-product-shadows}
For all $r,\alpha>0$ and $T, T'>0$, and $v\in S^g\widetilde{M}$ such that $\calB^{g}_{v_+^g}(\pi(v),o)=0$, if $x_t$ denotes the footpoint of $g_t(v)$,  we have
$$
B^g(v; T,T',r)\subset (H^g)^{-1}\left(\calO^g_{x_{-T'}}(B^g(x_{T},2r)\times \calO^g_{x_{T}}(B^g(x_{-T'},2r))\times]-r,r[\right)\,, \quad \mbox{and}
$$
$$(H^g)^{-1}\left(\calO^g_{x_{-T'}}(B^g(x_{T},r)\times \calO^g_{x_{T}}(B^g(x_{-T'},r))\times]-\alpha,\alpha[\right)
\subset B^g(v; T,T',2r+2\alpha)\,.
$$
\end{lemm}

%In this paper, we want to differentiate the entropy along a variation of metrics.
%Before differentiating it, we want to compare entropies of different metrics.
%This will be done thanks to the comparison of dynamical balls of different metrics,
%and more precisely of shadows of different metrics.

When $g_1$ and $g_2$ are two admissible negatively curved metrics on $M$, recall that
any $g_1$-geodesic between any two points
is at distance at most $C_2(g_1,g_2)$ of the  $g_2$-geodesic joining
the same endpoints, and vice versa,
for some constant $C_2(g_1,g_2)$ depending only on $g_1$ and $g_2$.
This leads immediately to the following lemma.

\begin{lemm}\label{comparison-shadows} Let $g_1$ and $g_2$ be two 
admissible negatively curved metrics on $M$, and $x,y$ two points on $\tilde{M}$. 
Then
$$
\calO^{g_1}_x(B^{g_1}(y,R))
\subset \calO^{g_2}_x(B^{g_2}(y,R+ C_2(g_1,g_2))
\subset \calO^{g_1}_x(B^{g_1}(y,R+2C_2(g_1,g_2))\,.
$$
\end{lemm}

These lemmas will have the following very convenient corollary. 

\begin{coro}\label{coro:comparison} Let $g_1$ and $g_2$ be two 
admissible negatively curved metrics on $M$. 
For all $\varepsilon>0$, there exists $C>0$ and $\varepsilon'=\varepsilon'(\varepsilon)>0$ 
such that for all  $v\in S^{g_1}M$, 
we have
$$
B^{g_2}(\Psi^{g_1\to g_2}(v),S+C,S'+C,\varepsilon)\subset
 \Psi^{g_1\to g_2}(B^{g_1} (v,T,T',\varepsilon))
\subset B^{g_2} (\Psi^{g_1\to g_2}(v),S,S',\varepsilon')\,,
$$
where  $S=\calB^{g_2}_{v_+^{g_1}}(\pi(v),\pi(g_1^Tv))$, $S'=\calB^{g_2}_{v_+^{g_1}}(\pi(v),\pi(g_1^{-T'}v))$  and $\varepsilon'=\varepsilon(5+C_1(g_1,g_2))+C_2(g_1,g_2)+2C_3(g_1,g_2)$.
\end{coro}

\begin{proof} As the sets considered in the above statement are typically small,
 we can prove them on $T\tilde{M}$ instead of $TM$. 
Without loss of generality, we can assume that $\pi(v)=o$. Indeed, 
all lemmas stated above are valid with $o$ an arbitrary point, for example the basepoint of $v$. 
In particular, we have $\Psi^{g_1\to g_2}(v)=\Phi^{g_1\to g_2}(v)$.

 We start with the right inclusion. Given $u\in B^{g_1}(v,T,T',\varepsilon)$,
 we want to control the distance 
$d^{g_2}(g_2^s\Psi^{g_1\to g_2}u,g_2^s\Psi^{g_1\to g_2}v)$. 
As $u\in B^{g_1}(v,T,T',\varepsilon)$,
$\tau^{g_1,g_2}(u)\le \varepsilon(1+C_1(g_1,g_2))$ 
so that 
$d^{g_2}(\Phi^{g_1\to g_2}(u),\Psi^{g_1\to g_2}(u))\le \varepsilon(1+C_1(g_1,g_2))$. 
Therefore, $\Psi^{g_1\to g_2}B^{g_1}(v;T,T',\varepsilon)$ is included in the 
$\varepsilon(1+C_1(g_1,g_2))$-neighbourhood of $\Phi^{g_1\to g_2}(B^{g_1}(v;T,T',\varepsilon)$. \\

Let $w=\Psi^{g_1\to g_2}(v)=\Phi^{g_1\to g_2}(v)$. 
Denote by $w_s$ (resp. $z_s$) the basepoint $\pi(g_2^s w)$, 
for $s\in\bbR$,  of $g_2^sw$. (resp. of $g_2^sz$).
Let $S=\calB^{g_2}_{v_+^{g_1}}(o,\pi(g_1^T v)$ and 
$S'=-\calB^{g_2}_{v_+^{g_1}}(o,\pi(g_1^{-T'}v))$. 
By lemma \ref{lem:GeodStretch}, we know that  
$$
|S-d^{g_2}(\pi(v),\pi(g^1_T v))|\le  C_3(g_1,g_2)\quad 
\mbox{and} \quad|S'-d^{g_2}(\pi(v),\pi(g^1_{-T'} v))|\le  C_3(g_1,g_2)\,.
$$
Moreover, the distances $d^{g_2}(w_S,v_T)$ and $d^{g_2}(w_{-S'},v_{-T'})$ are uniformly bounded. 
Indeed, by Lemma \ref{lem:Morse-Psi}, 
$\Psi^{g_1\to g_2}(g_1^T v)=g_2^S\Psi^{g_1\to g_2}(v)$ so that 
$d^{g_2}(w_S,v_T)=d^{g_2}(\pi(\Psi(g_1^Tv),\pi(g_1^Tv))\le C_3(g_1,g_2)$.

Lemma  \ref{comparison-shadows} and elementary geometric 
considerations in negative curvature give the inclusion

\begin{eqnarray*}
\calO^{g_1}_{v_{-T'}} (B^{g_1}(v_T,2\epsilon) \times  \calO^{g_1}_{v_T}(B^{g_1}(x_{-T'},2\epsilon)) & \times ]-\epsilon,\epsilon[  \subset     \\
\calO^{g_2}_{w_{-S'}}(B^{g_2}(w_S,2\epsilon+ C_2(g_1,g_2)+2C_3(g_1,g_2))\times \quad \quad &\\
  \calO^{g_2}_{y_S}(B^{g_2}(y_{-S'},2\epsilon+C_2(g_1,g_2)+2C_3(g_1,g_2)))&\times ]-\epsilon,\epsilon[\hfill\,.
\end{eqnarray*}
Lemma \ref{dynamical-ball-product-shadows} implies the right inclusion 
$\Phi^{g_1\to g_2}B^{g_1}(v;T,T',\epsilon)\subset B^{g_2}(w;S,S',4\epsilon+C_2(g_1,g_2)+2C_3(g_1,g_2))$.
The relation between $\Phi^{g_1\to g_2}$ and $\Psi^{g_1\to g_2}$ gives 
$$
\Psi^{g_1\to g_2}B^{g_1}(v;T,T',\varepsilon)\subset
B^{g_2}(\Psi^{g_1\to g_2}(v),S,S',\varepsilon(5+C_1(g_1,g_2))+C_2(g_1,g_2)+2C_3(g_1,g_2))\,.
$$

We proceed in the same way for the left inclusion, 
but we need in addition the help of Lemma \ref{lem:dynamical-balls-inclusion}.
 
Reasoning similarly as above gives the inclusion  
$$
B^{g_2}(w;S,S',\varepsilon)\subset 
\Psi^{g_1\to g_2}\left(B^{g_1}(v;T,T',(4\varepsilon+ C_2(g_1,g_2)+2C_3(g_1,g_2))(1+C_1(g_1,g_2))\right)\,.
$$
As $T,T',\varepsilon$ are arbitrary, using lemma \ref{lem:dynamical-balls-inclusion}, 
we obtain easily the existence of a constant $C>0$ such that 
$$
B^{g_2}(w;S+C,S'+C,\varepsilon)\subset \Psi^{g_1\to g_2}\left(B^{g_1}(v;T,T',\varepsilon \right)\,.
$$
\end{proof}

%%%%%%%%%%%%%%%%%%%%%%%%%%%

\subsection{Kolmogorov-Sinai, Brin-Katok and topological entropies}\label{sub:entropies}

The Kolmogorov-Sinai entropy of a dynamical system $T$ w.r.t an invariant probability measure
$\mu$ is the supremum over all measurable partitions of the
 exponential growth rate of the complexity of a partition, when iterated
by $T$, and measured by $\mu$. 
By Shannon-McMillan-Breiman Theorem, it also equals (the supremum over all partitions of) 
the exponential decay rate of
a typical atom of the iterated partition.

Instead of iterating a measurable partition, when $X$ is a metric space, endowed with the Borel
$\sigma$-algebra, one can consider exponential decay rate of the measure of typical
dynamical balls, which will give us a notion of local entropy, introduced by \cite{BK83}.

When $T$ is a continuous map on a {\em compact} space $X$, Brin-Katok \cite{BK83} showed that this
Kolmogorov-Sinai entropy coincides with the exponential decay of dynamical balls,
also called the local entropy. This equality also holds when $T$ is a 
lipschitz map of a noncompact manifold,
as has been verified in \cite[Thm 1.32]{Riq-these}.

We shall not define the classical Kolmogorov-Sinai entropy, denoted by $h_{KS}(T,m)$, 
because we do not really use it in this work.
But we recall below some definitions of local entropy and the statements of Brin-Katok and Riquelme. 

For $(\phi^t):X\to X$ a dynamical system and $m$ a finite invariant measure, define
the lower local entropy 
\begin{equation}\label{local-inf}
\underline{h}_{loc}(T,m))=
  \,\,\underset{x\in X}{\textrm{ess inf}}\,
 \lim_{\epsilon\to 0}\liminf_{T\to\infty}-\frac{1}{T}\log m(B_{dyn}(x,T,\epsilon))\,,
\end{equation}
 and the upper local entropy relative to compact sets

\begin{equation}\label{local-sup-comp}
\overline{h}_{loc}^{comp}(T,m))=
\sup_{K}\,\,\underset{x\in K}{\textrm{ess sup}}\, \lim_{\epsilon\to 0}\limsup_{T\to\infty, \phi^T  x\in K}-\frac{1}{T}\log m(B_{dyn}(x,T,\epsilon))
\end{equation}

For the geodesic flow in negative curvature, 
dynamical balls should be defined relatively to a distance on $S^gM$, 
but, as mentioned in the above subsection, the "natural" Sasaki distance 
on $S^gM$ is equivalent to the distance 
$d(v,w)=\sup_{-1\le t\le 0} d^g(\pi(g_t v),\pi(g_t w))$, so that, 
when studying asymptotic quantities as entropy,  
we can use the distance $d^g$ on $M$ instead of the Sasaki distance on $S^gM$. 

The following result is essentially due to Brin-Katok and Riquelme. 

\begin{theo}[Brin-Katok \cite{BK83}, Riquelme  \cite{Riq-these} \cite{Riq-Ruelle-geod} ]\label{th-entropies} 
Let $M$ be a Riemannian manifold with pinched negative curvature. 
\begin{equation}\label{entropies-egales}
h_{KS}(m,g)=\underline{h}_{loc}(m,g)=\overline{h}_{loc}^{comp}(m,g) \,.
\end{equation}
\end{theo}

\begin{proof} This result is due to Brin-Katok in the compact case. 
Their proof of the inequality $h_{KS}(m,g)\le \underline{h}^{loc}(m,g)$ extends verbatim to the noncompact case. 
In \cite[Th.1.32]{Riq-these}, Riquelme proved the equality 
$h_{KS}(m,g)=\underline{h}^{loc}(m,g)$ for any Lipschitz dynamical system. 
In \cite[Th 1.41]{Riq-these}, he established the inequality 
$\underline{h}^{loc}(m,g)\le \overline{h}^{loc}_{comp}(m,g) $, 
and the inequality   $  \overline{h}^{loc}_{comp}(m,g)\le  h_{KS}(m,g)$ 
is established in the proof of \cite[Th 1.42]{Riq-these}. 
\end{proof}

As observed in Remark \ref{erratum} there are two notions of dynamical balls 
and the small ones are more relevant for us. 
Therefore, we define what we will call the local entropy, denoted by $h^{loc}_{\Gamma}(m,g)$ in the sequel, as follows. 

\begin{equation}\label{local-entropy-for-us}
h^{loc}_{ \Gamma}(m,g))=
\sup_{K\subset S^g M}\,\,\underset{v\in K}{\textrm{ess sup}}\, \lim_{\epsilon\to 0}\limsup_{T\to\infty, g^T  v\in K}-\frac{1}{T}\log m(B^g  (v,T,\epsilon))
\end{equation} 

It follows from Theorem \ref{th-entropies} and inclusion \ref{inclusion} that 
\begin{equation}\label{entropies-different}
h_{KS}(m,g)=h^{loc}_{comp}(m,g)\le h^{loc}_{ \Gamma}(m,g),\end{equation}
 with equality as soon as $M$ has an injectivity radius bounded from below or $m$ has compact support.

\begin{rema}\rm Let us emphasize that all  definitions 
of  entropies above are sensitive to the scaling of the metric but 
 not sensitive to
the scaling of the measure. 
In particular, if $m$ is finite but not a probability measure, 
then 
$$
h^{loc}_{ \Gamma}(m,g)=h^{loc}_{ \Gamma}(\lambda m, g)\quad \mbox{and}\quad h^{loc}_{ \Gamma}(m,\lambda g)=\frac{h^{loc}_{ _\Gamma}(m,g)}{\sqrt{\lambda}} \,.
$$
\end{rema}

\begin{rema}\rm Observe that contrarily to Kolmogorov-Sinai entropy, 
the above definitions of local entropy make perfectly sense for an infinite invariant ergodic and conservative Radon measure. 
In particular, the Bowen-Margulis measure (see section \ref{BM}) which, when finite, is the measure of maximal entropy of the geodesic flow, always has a local entropy with respect to small dynamical balls and return times into compact sets which coincides with the topological entropy of the geodesic flow, see Proposition \ref{entropie-BM}.
\end{rema}

%%%%%%%%%%%%%%%%%%%%%%%%%%%%%%%%%%%%%%%%%%%%%%%%%%%

Lemma \ref{lem:dynamical-balls-inclusion} allows us to choose some $\e>0$  without  need 
to take the limit when $\e\to 0$. 
Moreover, the invariance of the measure allows to consider shifted dynamical balls. 
It is the result below.

\begin{lemm}\label{eq:LocalEntropy} Let $(M,g)$ be a  manifold with pinched negative curvature, 
and $\mu$ a geodesic current. 
Let $m_\mu^g$ be the
$g$-invariant measure associated to $\mu$ on $S^gM$.
One can compute its   local entropy as
$$
h^{loc}_{\Gamma}(m_\mu^{g},g)=
\sup_K\supess_{v\in K}   \limsup_{T+T'\to \infty, g^{T } v\in K, g^{-T'}v\in K}-\frac{1}{T+T'}\log m^g_\mu(B^{g} (v;T,T',\e))\,.
$$
\end{lemm}

Geometers usually are more interested in topological entropy than  measure-theoretic  entropy. 
We shall not define topological entropy topologically, but through the variational principle. Denote by 
$\mathcal{M}^1(g)$ the set of invariant probability measures for the metric $g$.

The topological entropy of the geodesic flow $(g^t)$, denoted by $h_{top}(g)$, satisfies 
\begin{equation}\label{htop}
h_{top}(g)=\sup_{m\in\mathcal{M}^1(g)} h_{KS}(m,g)\,.
\end{equation}
This variational principle is due first to \cite{Dinaburg, Goodman},
% (see also \footnote{Goodwyn})
 \cite{Misiu} and later Handel-Kitchen \cite{HK} on noncompact spaces. 
It follows from \cite{OP} that this supremum is achieved iff the so-called Bowen-Margulis measure is finite (see later subsection \ref{BM} for details). 
 In this case, it is the unique measure maximizing entropy.

%%%%%%%%%%%%%%%%%%%%%%%%%%%%%%%%%%%%%%%%%%%%%%%%%%%%%%%%%%%%%%%%%%%%%%%%%%%%%ùù

 \subsection{Entropy transformation law}\label{sub:proof}

Our goal is to prove the following result.

\begin{theo}\label{entropy-transformation} Let $(M,g_i)$, $i=1,2$ be two admissible Riemannian metrics with pinched negative curvature on $M$. 
Let $\mu$ be a  geodesic current and $m^{g_i}_\mu$ the associated invariant measure on $S^{g_i}M$ under the geodesic flow $(g_i^t)$. Assume that these measures are finite and ergodic.
Then their local entropies are related as follows.
 $$
h^{loc}_{ \Gamma}(m^{g_2}_\mu,g_2)=I_\mu(g_2, g_1)\,.\, h^{loc}_{ \Gamma}(m^{g_1}_\mu,g_1)\,.
%=\int_{S^{g_2}M}\mathcal{E}^{g_2\to g_1}(v)\frac{dm_\mu^{g_2}(v)}{\norm{m_\mu^{g_2}}}\,. \,h_{loc}(m^{g_1}_\mu, g_1)\,.
$$ 
%The  Kolmogorov-Sinai entropy of the normalized measures 
%$m^{g_i}_\mu/\|m^{g_i}_\mu\|$ also satisfies 
%$$
%h_{KS}(\frac{m^{g_2}_\mu}{\|m_\mu^{g_2}\|},g_2)=
%I_\mu(g_2, g_1)\,.\, h_{KS}(\frac{m^{g_1}_\mu}{\|m_\mu^{g_1}\|},g_1)\,.
%$$
\end{theo}

Thanks to Proposition \ref{prop:transformation-mass}, the corollary below follows. 
\begin{coro}\label{coro:invariant} Under the same assumptions, we have 
$$ 
h^{loc}_{ \Gamma}(m^{g_2}_{\mu},g_2)\times\|m_\mu^{g_2}\|=
h^{loc}_{ \Gamma}(m^{g_1}_{\mu},g_1)\times\|m_\mu^{g_1}\|\,.
$$ 
\end{coro}

%%%%%%%%%%%%%%%%%%%%%%%%%% J'en suis la dans mes revisions 

Let us prove  Theorem \ref{entropy-transformation}. 

\begin{proof} Without loss of generality, we will assume $\mu$ to be ergodic.
It follow from Lemma \ref{eq:LocalEntropy}  that the entropy may be computed as
$$
h^{loc}_\Gamma(m_\mu^{g_2},g_2)=
\sup_K\supess_{v\in K}   \limsup_{T+T'\to \infty, g_2^T v\in K, \,g_2^{-T'}v\in K}
-\frac{1}{T+T'}\log m_\mu^{g_2}(B^{g_2} (v;-T,T',\e))
$$
for some fixed $\e>0$,
the essential supremum being relative to $m^{g_2}_\mu$.
The above limsup is constant along $(g_2)$-orbits, 
so that by ergodicity, it is $m_\mu^{g_2}$-almost surely constant. 
Observe also that when $K$ grows, the quantity on the right also grows.

Choose some large compact set $K\subset TM$ large enough to
contain an open subset of $\Omega^{g_i}\cap S^{g_i}M$ for $i=1,2$, and to have positive $m_\mu^{g_i}$ measure.
Choose it large enough so that
it allows to estimate entropies $h^{loc}_\Gamma(m_\mu^{g_i},g_i)$, up to some small arbitrary $\alpha$. 
In other words, 
$$
\left|h^{loc}_\Gamma(m_\mu^{g_i},g_i)-\supess_{v\in K\cap S^{g_i}M}  
 \limsup_{T+T'\to \infty, g_2^T v\in K, \,g_2^{-T'}v\in K}-\frac{1}{T+T'}\log m_\mu^{g_i}(B^{g_2} (v;-T,T',\e))\right|\le \alpha\,.
$$

Choose a typical $v\in S^{g_1}M\cap K$, which realizes the above essential supremum on $K$, and the almost sure conclusion of Corollary \ref{coro:GeodStretch} when $T\to \pm\infty$. 
With the notations of Corollary \ref{coro:comparison}, let $w=\Psi^{g_1\to g_2}(v)$. As observed in the preceding section, we have $m^{g_2}_\mu=\mathcal{E}^{g_1\to g_2}\times \Psi^{g_1\to g_2}_* m^{g_1}_\mu$. But $\mathcal{E}^{g_1\to g_2}$ is uniformly close to $1$ on $B^{g_1}(v,\varepsilon)$. 

Thus, up to some constants $e^{\pm c(v,\epsilon)}$, by corollary \ref{coro:comparison},  we have
\begin{eqnarray*}\label{eqn:comparison-measures}
  e^{-c(v,\varepsilon)} m_\mu^{g_2}(B^{g_2} (w;S+C,S'+C,\epsilon)\le m_\mu^{g_1}\left(B^{g_1}(v;T,T',\e )\right) \le e^{c(v,\varepsilon)}
m_\mu^{g_2}\left(B^{g_2} (w;S,S',\varepsilon'\right)   
\end{eqnarray*}
with $w=\Phi^{g_1\to g_2}(v)$, $S =d^{g_2}(\pi(w),\pi(g_2^{ T}w))\pm C_3(g_1,g_2)$ and 
$S'=d^{g_2}(\pi(v),\pi(g_2^{-T'}w))\pm C_3(g_1,g_2)$.

Observe also that the condition $g_1^T v\in K$ (resp. $g_1^{-T'}v\in K$ ) implies
that $g_2^S w$ (resp $g_2^{-S'} w $) belongs to the $C_3(g_1,g_2)$-neighbourhood of $K$ for any of 
the two metrics $g_1$ or $g_2$.  It remains true for $g_2^{S+C} w$ and $g_2^{-S'-C} w $ inside the $C_3(g_1,g_2)+C$-neighbourhood of $K$ for the metric $g_2$.
Set $K'=\mathcal{V}_{C_3(g_1,g_2)+C}(K)\supset K$.

By definition of $T, T', S, S'$, we also have $T+T'\to +\infty$ iff $S+S'\to \infty$.

Therefore, taking the limsup of $\frac{1}{S+S'}\log $ of the above quantity,
we get
 \begin{eqnarray*}
& \quad&\limsup_{S+S'\to \infty, g_2^S w\in K', \,g_2^{-S'}w\in K'}-\frac{1}{S+S'}\log m_\mu^{g_2}(B^{g_2}(w,S,S',\e')\\
&=&
\limsup_{T+T'\to \infty, g_1^T v\in K ,g_1^{-T'}v\in K }\frac{T+T'}{S+S'}
\times \frac{-1}{T+T'}\log m_\mu^{g_1}(B^{g_1}(v,T,T',\e))
\end{eqnarray*}

By Corollary \ref{coro:GeodStretch} we know that
$$
\frac{T+T'}{S+S'}=\frac{T+T'}{\calB^{g_2}_{v_+^{g_1}}(\pi(g_1^{-T'}v),\pi(g_1^Tv))}=
\frac{T+T'}{\int_{-T'}^T\mathcal{E}^{g_1\to g_2}(g_1^t v)\,dt}
 $$ 
converges when $T+T'\to +\infty$ to 
$$
\frac{1}{\int_{S^{g_1}M}\mathcal{E}^{g_1\to g_2} \,\frac{dm_\mu^{g_1}}{\|m_\mu^{g_1}\|}}=
\int_{S^{g_2}M}\mathcal{E}^{g_2\to g_1} \,\frac{dm_\mu^{g_2}}{\|m_\mu^{g_2}\|}\,.
$$
We deduce easily, by taking the supremum in $K$, that
$$
h^{loc}_\Gamma(m_\mu^{g_2},g_2)=
\int_{S^{g_2}M}\mathcal{E}^{g_2\to g_1} \,\frac{dm_\mu^{g_2}}{\norm{m^{g_2}_\mu}}\,\,h^{loc}_\Gamma(m_\mu^{g_1},g_1)=
I_\mu(g_2,g_1)\times h^{loc}_\Gamma(m_\mu^{g_1},g_1)\,.
$$
\end{proof}

%%%%%%%%%%%%%%%%%%%%%%%%%%%%%%%%%%%%%%%%%%%%%%%%%%%%%%%%%%%

\subsection{Bowen-Margulis measures and comparison of topological entropies}\label{BM}

We define now the so-called {\em Bowen-Margulis measure}, 
and use it to deduce from Theorem \ref{entropy-transformation} 
a corollary about the comparison of topological entropies of two metrics $g_1$ and $g_2$. 
The construction below is due to Patterson \cite{Patterson} for compact surfaces, 
to Sullivan \cite{Sull, Sull84} for geometrically finite hyperbolic manifolds, 
and Yue  \cite{Yue} extended Sullivan's work in variable negative curvature.  

Let $(M,g)$ be a negatively curved manifold, with pinched negative curvature.
Choose some point $o\in\tilde M$.
Consider the Poincar\'e series 
$$
P^g_\Gamma(s)=\sum_{\gamma\in \Gamma} e^{-sd^g(o,\gamma o )}\,.
$$
Let $\delta(g)$ be its critical exponent. 
This exponent is finite, and when $\Gamma$ is nonelementary, it is positive. 
The following lemma is immediate from the definition of $\delta$. 
%%%%%% HERE Lemma 5.13 
\begin{lemm}\label{lem:continuity-exponent} Let $(g_\epsilon)_{-1\le \epsilon\le 1}$ be a family of negatively curved metrics on $M=\widetilde{M}/\Gamma$, such that 
$e^{-\epsilon} g_0\le g_\epsilon\le e^\epsilon g_0$. Then $e^{-\epsilon/2} \delta(g_0)\le \delta(g_\epsilon)\le e^{\epsilon/2} \delta(g_0)$.
\end{lemm}

We need to ensure that the above series diverges at $s=\delta(g)$, 
which could be false. 
We will modify $P^g_\Gamma(s)$ into $\tilde P^g_\Gamma(s)$ as follows. 
The Patterson trick \cite{Patterson} is the following. 
Define a continuous map $h:(0,+\infty)\to (0,+\infty)$ as  
the exponential of continuous piecewise affine maps\,: 
$h(x)=\exp(\epsilon_k x)$ on the interval $I_k$, with $\epsilon_k\to 0$  
and $I_k$ a sequence of adjacent intervals of increasing length. 
It is possible to do it in such a way that $h$ is positive, increasing, 
continuous,  with slow growth, \,: $\frac{h(t_0+t)}{h(t_0)}$ is bounded 
by $\exp(\varepsilon_k t)$, and therefore converges to $1$ when $t\to +\infty$ 
uniformly in  $t_0>0$. 
Moreover, $\epsilon_k$ and $I_k$ can be chosen in order to ensure that 
$$
\tilde P^g_\Gamma(s) =\sum_{\gamma\in \Gamma} h(d^g(o,\gamma \cdot o)) e^{-sd^g(o,\gamma o )}
$$
 has exponent $\delta(g)$ but now diverges at $s=\delta(g)$.  

Define for all $x\in\tilde M$ and $s>\delta(g)$ a probability measure 
$$
\nu^s_x=
\frac{1}{\tilde P^g_\Gamma(s)} \sum_{\gamma\in \Gamma} h(d^g(o,\gamma \cdot o))
e^{-sd^g(x,\gamma o )}\Delta_{\gamma o }\,
$$
on $\overline{M}=\tilde M\cup \partial\tilde M$, 
where $\Delta_x$ denotes the Dirac mass at the point $x$. 
Choose a decreasing sequence $s_k\to \delta(g)$ such that 
$\nu^{s_k}_o$ converges to a probability measure 
$\nu_o^g$ on $\overline{M}$. 
Choose for all $x\in\tilde M$ a subsequence $s_{k_j}$ of $s_k$ 
such that $\nu_x^{s_{k_j}}$ converges to a measure $\nu^g_x$ on $\overline{M}$. 
By construction, as $\tilde P(\delta(g))$ diverges, 
all these measures are equivalent finite measures supported 
on $\Lambda_\Gamma\subset \partial\tilde M$, 
the measure $\nu^g_o$ is a probability measure, 
and this family $(\nu^g_x)_{x\in\tilde M}$ satisfies two crucial properties 
for all $x,y\in\tilde M$, almost all $\xi\in\partial\tilde M$ and all $\gamma\in \Gamma$\,:
$$
\frac{d\nu^g_x}{d\nu^g_y}(\xi)=\exp\left(-\delta(g)\calB_\xi(x,y)\right) 
\quad\mbox{and}\quad \gamma_*\nu^g_x=\nu^g_{\gamma x} \,.
$$
From these properties follows the Sullivan's Shadow Lemma. 

\begin{prop}[Sullivan \cite{Sull}]\label{ShadowLemma} 
Let $(\nu^g_x)$ be a family of measures on $\Lambda_\Gamma$ obtained as above. 
Then for all $R>0$ large enough, there exists a constant $c=c(R)>0$ such that 
$$
\frac{1}{c}\exp\left(-\delta(g)d^g(o,\gamma o)\right)\,\le\,\nu^g_o\left(\calO_o(B(\gamma o ,R))\right)\,
\le\, c \exp\left(-\delta(g)d^g(o,\gamma o)\right)\,.
$$
\end{prop}

A {\em Bowen-Margulis measure} on $S^gM$ is a measure obtained 
from such a family $(\nu^g_x)$ by the following formula on 
$S^g\tilde M$, with $v=(H^g)^{-1}(v^g_-,v^g_+,t)$
\begin{equation}
d\tilde{m}^g_{BM}(v)=
\exp\left(\delta(g)\calB^g_{v^g_+}(o ,\pi(v) )+\delta(g)\calB^g_{v^{g}_-}(o,\pi(v)) \right) 
d\nu^g_o(v^g_+)d\nu^g_o(v^g_-)dt\,.
\end{equation}
This formula being $\Gamma$-invariant, 
it induces on the quotient a Bowen-Margulis measure $m^{g}_{BM}$ on $S^gM$. 

It is well known (see the above references, or Roblin \cite{Roblin} for the most general version)
 that $P^g_\Gamma$ diverges at $s=\delta(g)$ iff the Bowen-Margulis measure is ergodic and conservative, 
and in this case, the family of measures $(\nu^g_x)$ is in fact unique. 
In particular, when this measure $m^{g}_{BM}$ is finite, it is ergodic and conservative and $P^g_\Gamma$ diverges at $\delta(g)$. 

 Otal-Peign\'e proved the following result, due to Sullivan in the case of geometrically finite hyperbolic manifolds. 

\begin{theo}[Sullivan \cite{Sull84}, Otal-Peign\'e \cite{OP}]\label{th:OP} 
Let $(M,g)$ be a manifold with pinched negative curvature and bounded derivatives of the curvature. 
Then 
$$
\delta(g)=h_{top}(g)
$$ is the topological entropy of $g$.
 Moreover, when $m^{g}_{BM}$ is finite
and normalized into a probability measure,
 it is the unique measure maximizing entropy in the sense that 
$h_{KS}(\frac{m^g_{BM}}{\|m^{g}_{BM}\|},g)=h_{top}(g)$. 
When $m^g_{BM}$ is infinite, there is no probability measure maximizing entropy. 
\end{theo}

It follows from \cite[Prop. 3.16]{PPS}  and \cite{OP} that, finite or not, 
the Bowen-Margulis measure satisfies the following equality. 
\begin{prop}\label{entropie-BM} Let $(M,g)$ be a negatively curved manifold  
with pinched negative curvature and bounded derivatives of the curvature.  Let $m^g_{BM}$ be a Bowen-Margulis measure. Then 
\begin{equation}\label{VP-loc}
h^{loc}_\Gamma(m^g_{BM},g)=\delta(g)=h_{KS}(m^g_{BM},g) \,.
\end{equation}\end{prop}
\begin{proof} The first equality is a computation done in \cite[Prop. 3.16]{PPS}, 
the second is one of the main results of \cite{OP}. 
\end{proof}

This equality suggests that we could be able to prove a variational principle 
for infinite measures, using local entropies instead of Kolmogorov-Sinai entropies. We will not do it here.

\begin{coro} \label{coro:EntBM} Let $(M,g_i)$, $i=1,2$ be two admissible Riemannian metrics 
with pinched negative curvature on $M$. 
Let $\mu$ be a  geodesic current and $m^{g_i}_\mu$ the associated invariant 
measure on $S^{g_i}M$ under the geodesic flow $(g_i^t)$. 
Assume that these measures are finite and ergodic. Then 
\begin{eqnarray*}
 h_{top}(g_2)=\delta(g_2)=h^{loc}_\Gamma(m^{g_2}_{BM},g_2)&=&I_{\mu^{g_2}_{BM}}(g_2,g_1)\times h^{loc}_\Gamma(m^{g_1}_{\mu^{g_2}_{BM}},g_1)\\
&\le&
 I_{\mu^{g_2}_{BM}}(g_2,g_1)\,\times \, h_{top}(g_1)\,.
\end{eqnarray*}
\end{coro}
\begin{proof} 
Let us first note that by Theorem \ref{theo:Gibbs}, the measure $m^{g_1}_{\mu^{g_2}_{BM}}$ is a Gibbs measure. Moreover, \cite[Thm 1.3]{PPS} ensures that the Gibbs measure associated to a given potential, when finite, is the unique equilibrium measure of this potential. Therefore $h^{loc}_\Gamma(m^{g_1}_{\mu^{g_2}_{BM}})=h_{KS}(m^{g_1}_{\mu^{g_2}_{BM}})$, and the variational 
principle ensures that $h_{KS}(m^{g_1}_{\mu^{g_2}_{BM}})\le h_{top}(g_1)$, which gives the last inequality. 
\end{proof}

In the compact case, the inequality 
$h_{top}(g_2)\le I_{\mu^{g_2}_{BM}}(g_2,g_1)\times h_{top}(g_1)$ is due to 
Knieper \cite{Kni95}. 
Katok had a similar weaker inequality \cite{Katok82}, proving that 
$$
h_{top}(g_2)\le \int_{S^{g_2}M} \|v\|^{g_1}dm_{BM}^{g_2}\times h_{top}(g_1)\,.
$$
Our inequality above is valid on any manifold, compact or not, with finite Bowen-Margulis measure. 
 It follows from lemma \ref{lem:Ineq-Knieper} that it implies Katok's inequality. 
Let us mention however that it is this weaker version 
which is really used in the proof of our main theorem of differentiability of entropy.

%%%%
 %%%%%%%%%%%%%%%%%%%%%%%%%%%%%%%%%%%%%%%%%%%%%%%%%%%%%%%%%%%%%%%%%%%%%%%%%%%%%%%
%%%%%%%%%%%%%%%%%%%%%%%%%%%%%%%%%%%%%%%%%%%%%%%%%%%%%%%%%%

\section{Gibbs measures}\label{sec:Gibbs}

This section, particularly Theorem \ref{theo:Gibbs}, is crucial in the proof of Corollary \ref{coro:EntBM}, and
therefore in our approach of Theorem \ref{theo:main}. 

Theorem \ref{theo:Gibbs} is new on noncompact manifolds, the explicit change of potential being new even on compact manifolds. 
Corollary \ref{coro:Gibbs} is new even on compact manifolds. 

Gibbs measures are, for a hyperbolic dynamical system, 
a family of measures with strong stochastic properties, 
each one associated to a weight, i.e. some H\"older continuous potential, 
describing somehow that all possible dynamical behaviours can happen. 
For the geodesic flow on the unit tangent bundle of a compact manifold, their geometric 
construction, adapted from the Patterson-Sullivan construction described in the above section, 
has been done by Ledrappier in \cite{Ledrappier}. 
He proved there, on compact manifolds, that being a Gibbs measure does not depend on the metric. 
In other words, 
if $g_1$ and $g_2$ are negatively curved metrics on $M$, 
an invariant measure $m^{g_1}_\mu$ on $S^{g_1}M$ is a
Gibbs measure iff the measure $m^{g_2}_\mu$ on $S^{g_2}M$ is also a Gibbs measure. 
However, his proof strongly relies on the compactness of $M$. 
Our goal in this section is to prove this result differently on noncompact manifolds. 

 %%%%%%%%%%%%%%%%%%%%%%%%%%%%%%%%%%%%%%%%%%%%%%%%%%%%%%%%%%%%%%%%%%%%%%%%%%%%%%%

\subsection{Definitions}

We refer to \cite{PPS} for details on all notions presented here. 
Let $(M,g)$ be a negatively curved manifold, with pinched negative curvatures and bounded derivatives of the curvature. Let $F:S^gM\to \bbR$ be a H\"older continuous map. 
The {\em pressure of $F$} is the quantity 
\begin{equation}\label{pressure}
P^g(F)=\sup_{m\in\mathcal{M}^1(g)}\left( h_{KS}(m,g)+\int_{S^gM} F\,dm\right)\,, 
\end{equation}
the supremum being considered over all invariant probability measures $m\in\mathcal{M}^1(g)$. 
An invariant probability measure $m$ is {\em an equilibrium state for $F$} if it realizes the above supremum. 

Assume that $P^g(F)$ is finite. 
An invariant measure $m$ under the geodesic flow $(g^t)$ satisfies the {\em Gibbs property} 
for the potential $F$ if for all compact sets $K\subset S^gM$ and $\e>0$ there exists 
a constant $C(K,\epsilon)>0$ such that for all $v\in K$ and $T>0$ with $g^Tv\in K$, we have
\begin{eqnarray}\label{Gibbs}
 \frac{1}{C(K,\epsilon)}\,\exp\left(\int_0^T F(g^t v)\,dt -T P^g(F)\right)\, &\le \,
m\left(B^g (v,T,\e)\right) \nonumber\\
 \le  \,C(K,\epsilon)& \exp\left(\int_0^T F(g^t v)\,dt -T P^g(F)\right)\,.
\end{eqnarray}

A variant of the Patterson-Sullivan construction presented in subsection \ref{BM} provides a measure
$m_F$ which satisfies (\ref{Gibbs}) see \cite[Prop. 3.16]{PPS}. Moreover, when finite and normalized into a probability measure, it is the unique equilibrium state, i.e. the unique measure realizing the supremum in (\ref{pressure}) 
(see \cite[Th. 6.1]{PPS}). When this measure $m_F$ is infinite, there is no equilibrium state for $F$. 
Let us summarize what is useful in the present work in the following proposition. 

\begin{prop} Let $(M,g)$ be a negatively curved manifold  
with pinched negative curvature and bounded derivatives of the curvature. 
Let $F:S^g M \to \mathbb{R}$ be a H\"older potential. If the measure $m_F$ is finite and normalized, then 
$$
P^{g}(m_F)=h_{KS}(m_F,g)+\int_{S^gM} F\,dm_F=h^{loc}_\Gamma(m_F,g)+\int_{S^gM} F\,dm_F\,.
$$
\end{prop}

 %%%%%%%%%%%%%%%%%%%%%%%%%%%%%%%%%%%%%%%%%%%%%%%%%%%%%%%%%%%%%%%%%%%%%%%%%%%%%%%

\subsection{Being a Gibbs measure does not depend on the metric}

\begin{theo}\label{theo:Gibbs} Let $(M,g_i)$ be two admissible metrics 
with pinched negative curvature and bounded derivatives of the curvature on $M$. 
Let $F:S^{g_1}M\to \bbR$ be a  
H\"older map, and $m_F^{g_1}$ the associated Gibbs measure. 
We assume 
$m_F^{g_1}$ ergodic and conservative. 
Let $\mu_F^{g_1}$ be the associated current on $\partial^2\widetilde{M}$. 
Let  $m^{g_2}_{\mu_F^{g_1}}$ be the $g_2$-invariant measure
associated to the same current. 

Then $m^{g_2}_{\mu_F^{g_1}}$ is also ergodic and conservative, and 
satisfies the Gibbs property (\ref{Gibbs}) for 
the H\"older potential 
$$
G=\left(F-P^{g_1}(F)\right)\circ  \Psi^{g_2\to g_1}  \times\mathcal{E}^{g_2\to g_1}  \,.
$$
Moreover, $P^{g_2}(G)=0$. 
In other words,   for all compact subsets $K\subset S^{g_2}M$ and $\e>0$
there exists $C>0$ such that for all $w\in K$   and $S>0$ with $g^Sw\in K$, we have
$$
\frac{1}{C}\,e^{\int_0^S G(g_2^s w)ds } \,\le \,m_{\mu_F^{g_1}}^{g_2}(B^{g_2}_\Gamma(w,S,\e)) 
\,\le \, C e^{\int_0^S G(g_2^s w)ds }\,.
$$  

If we assume moreover that the measure $m_F^{g_1}$ is finite, 
and is therefore the  equilibrium 
measure associated to $F$, then $m_{\mu_F^{g_1}}^{g_2}/\|m_{\mu_F^{g_1}}^{g_2} \|$ 
is the  equilibrium measure associated to $G$.  
\end{theo}

\begin{rema}\rm Reversing the role of $g_1$ and $g_2$, we observe that 
the same result holds with the potential $H=\left((F-P^{g_1}(F) )\times (\mathcal{E}^{g_1\to g_2} )^{-1}\right) \circ \left(\Psi^{g_1\to g_2}\right)^{-1}$. 
Therefore, they must be cohomologous. 
\end{rema}
  
\begin{proof}  Conservativity and ergodicity depend only on the current at infinity and not on the (admissible) metric, as said in Proposition \ref{easy-comparison}.

Gibbs property for the potential $G$ follows from Corollary \ref{coro:comparison}. 
Let us explain it more in details. 
We stated Theorem \ref{theo:Gibbs} in the most natural way, 
starting from $g_1$ and going to $g_2$, but in view of all the
 statements proved above that we shall use, we will reverse the role of $g_1$ and $g_2$, 
$F$ and $G$, in the proof below. 
Assume that $m^{g_2}_G$ is a Gibbs measure w.r.t. the potential $G$ on $S^{g_2}M$, 
let $\mu=\mu^{g_2}_G$ be its current at infinity, and let us prove that $m^{g_1}_\mu$ 
is a Gibbs measure w.r.t. the potential $\ds F=\left(G- P^{g_2}(G)\right)\circ \Psi^{g_1\to g_2}\times \mathcal{E}^{g_1\to g_2}$.

First choose some compact set $K^{g_1}\subset S^{g_1}M$ and some $\epsilon>0$. 
Let $v\in K^{g_1}$ and $T>0$ such that $g^Tv\in K^{g_1}$. Define a compact set $K^{g_2}$ as the $C$-neighbourhood of $\Psi^{g_1\to g_2}(K^{g_1})\cup (\Psi^{g_2\to g_1})^{-1}K^{g_1}$, where $C$ is given by Corollary \ref{coro:comparison}. 

We will use Corollary \ref{coro:comparison} and first part of Proposition \ref{prop:Change-Var-Flow}, and the fact 
that $m^{g_2}_\mu=\Psi^{g_1\to g_2}_*(\mathcal{E}^{g_1\to g_2}\times m^{g_1}_\mu)$. 

As $\mathcal{E}^{g_1\to g_2}$ is continuous, it is uniformly continuous on $K^{g_1}$ so that for all 
$v\in K^{g_1}$ and $u\in B^{g_1}(v,\epsilon)$, $\mathcal{E}^{g_1\to g_2}(u)=e^{\pm c(K^{g_1},\epsilon)}\mathcal{E}^{g_1\to g_2}(v)$. We deduce that 
\begin{eqnarray*}
\frac{e^{-c(K^{g_1},\epsilon)}}{\mathcal{E}^{g_1\to g_2}(v)}\, m^{g_2}_\mu(\Psi^{g_1\to g_2}(B^{g_1}  (v,T,\epsilon))&\le & \
m^{g_1}_\mu\left(B^{g_1} (v,T,\epsilon) \right)\\
&\le & \frac{e^{  c(K^{g_1},\epsilon)}}{\mathcal{E}^{g_1\to g_2}(v) }\,m^{g_2}_\mu(\Psi^{g_1\to g_2}(B^{g_1}_\Gamma(v,T,\epsilon))\,.
\end{eqnarray*}

Now, using Corollary \ref{coro:comparison}, with $w=\Psi^{g_1\to g_2}v$, 
and $S=\calB^{g_2}_{v_+^{g_1}}(\pi(v),\pi(g_1^Tv))$, we get
\begin{eqnarray*}
\frac{e^{-c(K^{g_1},\epsilon)}}{\mathcal{E}^{g_1\to g_2}(v)}\,
 m^{g_2}_\mu( B^{g_2} (w:S+C,C,\epsilon)) &\le &
m^{g_1}_\mu(B^{g_1} (v,T,\epsilon)\\
&\le& \frac{e^{  c(K^{g_1},\epsilon)}}{\mathcal{E}^{g_1\to g_2}(v) }
m^{g_2}_\mu( B^{g_2} (w,S,\epsilon'))\,.
\end{eqnarray*}
As $m_\mu^{g_2}$ is a Gibbs measure, and $w$, $g_2^Sw$, but also $g_2^{-C}w$ 
and $g_2^{S+C}w$  belong to $K^{g_2}$, there exists a constant 
$C(G,K^{g_2},\epsilon')$ coming from the Gibbs property, such that 
\begin{eqnarray*}
\frac{e^{-c(K^{g_1},\epsilon)}}{\mathcal{E}^{g_1\to g_2}(v)}\,
\frac{e^{\int_{-C}^{S+C}(G-P^{g_2}(G))(g_2^s w)\,ds}}{C(G,K^{g_2},\epsilon')} &\le &
m^{g_1}_\mu(B^{g_1} (v,T,\epsilon)\\
&\le& \frac{e^{  c(K^{g_1},\epsilon)}}{\mathcal{E}^{g_1\to g_2}(v) }
C(G,K^{g_2},\epsilon') e^{\int_{0}^{S }(G-P^{g_2}(G))(g_2^s w)\,ds} \,.
\end{eqnarray*}
As $G$ is (H\"older) continuous, it is bounded on $K^{g_2}$, so that the integral 
$\int_{-C}^{S+C}(G-P^{g_2}(G))(g_2^s w)\,ds$ is, up to a constant $c$, uniformly close to 
$\int_0^S(G-P^{g_2}(G))(g_2^s w)\,ds$. 
The next ingredient is Proposition \ref{prop:Change-Var-Flow}, which gives 
\begin{eqnarray*}
\frac{e^{-c(K^{g_1},\epsilon)}}{\mathcal{E}^{g_1\to g_2}(v)}\,
\frac{e^{-c}}{C(G,K^{g_2},\epsilon')}e^{\int_{0}^{T}F(g_1^t w)\,dt}&\le&
m^{g_1}_\mu(B^{g_1} (v,T,\epsilon)\\
&\le& \frac{e^{  c(K^{g_1},\epsilon)}}{\mathcal{E}^{g_1\to g_2}(v) }
C(G,K^{g_2},\epsilon') e^{\int_{0}^{T } F(g_1^t w)\,dt} \,,
\end{eqnarray*}
with $F=(G-P^{g_2}(G))\circ\Psi^{g_1\to g_2}\times\mathcal{E}^{g_1\to g_2}$. 
It is exactly the Gibbs property for $m^{g_1}_\mu$ w.r.t. $F$. \\

It remains to show that $P^{g_1}(F)=0$. 
To simplify notations, let us assume that $P^{g_2}(G)=0$. 
Let $\rho$ be any geodesic current on $\partial^2\tilde M$. 
By definition, 
$$
P^{g_1}(F)=\sup_{\rho} \left(h_{KS}(m_\rho^{g_1},g_1)+\int_{S^{g_1}M} F dm^{g_1}_\rho\right)\,,$$
the supremum being taken over all currents $\rho$ such that $m^{g_1}_\rho$ is an invariant probability measure. 
The change of mass and change of entropy (Corollary \ref{prop:transformation-mass}  
and Theorem \ref{entropy-transformation}) give
$$
P^{g_1}(F)=
\sup_{\rho} I_\rho(g_2,g_1)\left(h(m_\rho^{g_2}/\|m_\rho^{g_2}\|,g_2)+\int Gdm_\rho^{g_2}/\|m_\rho^{g_2}\|\right)\le 0\,.
$$
The same computations with $\rho=\mu=\mu_G^{g_2}$ give $P^{g_1}(F)=0$. 
\end{proof}

%%%%%%%%%%%%%%%%%%%%%%%%%%%%%%%%%%%%%%%%%%%%%%%%%%%%%%%%%%%%%%%

\subsection{Length spectrum and change of metrics }

Let $g_1$ and $g_2$ be two quasi-isometric negatively curved metrics. 
There is a particular case where the above results have an easy but striking illustration 

\begin{coro}\label{coro:Gibbs} Let $(M,g_i)$ be two quasi-isometric complete 
negatively curved metrics on the same connected manifold $M$.
 Assume that the Bowen-Margulis measure of $g_1$ is ergodic and conservative, 
and let $\mu_{BM}^{g_1}$
be the associated geodesic current. 
Then the measure $m_{\mu_{BM}^{g_1}}^{g_2}$ is also ergodic and conservative. 
It is a Gibbs measure associated with the potential $G=-h_{top}(g_1)\mathcal{E}^{g_2\to g_1}$. 

Moreover, for all primitive hyperbolic elements $\gamma \in\Gamma$, if $w_\gamma$ is a
 periodic vector of $S^{g_2}M$ associated to $\gamma$, 
for all $\varepsilon>0$ there exists $C\ge 1$ such that for all $T>0$, we have
$$
\frac{1}{C} e^{-h_{top}(g_1)T \frac{\ell^{g_1}(\gamma)}{\ell^{g_2}(\gamma)} }\,\le \,
 m_{\mu_{BM}^{g_1}}^{g_2}\left(B^{g_2}(w_\gamma,T,\varepsilon)\right) \,\le\, C e^{-h_{top}(g_1)T \frac{\ell^{g_1}(\gamma)}{\ell^{g_2}(\gamma)} }\,.
$$
\end{coro}

\begin{proof} It is an immediate application of Theorem \ref{theo:Gibbs} with $F=0$. 
First write $T$ as $T=n \ell^{g_2}(\gamma)+r$, with $0\le r<\ell^{g_2}(\gamma)$. 
The only thing to notice is that 
$\int_{0}^{\ell^{g_2}(\gamma)} \mathcal{E}^{g_2\to g_1}(g_2^s w_\gamma)\,ds=
 \ell^{g_2}(\gamma)\times e^{g_2\to g_1}(\gamma)$
so that $$
-\int_0^T h_{top}(g_1)\mathcal{E}^{g_2\to g_1}(g_2^s v_\gamma)\,ds= 
-h_{top}(g_1) \times T\times \frac{\ell^{g_1}(\gamma)}{\ell^{g_2}(\gamma)}\pm \quad\mbox{constant}\,,
$$
the error term in the above inequality being smaller than $h_{top}(g_1)\ell^{g_2}(\gamma)\|\mathcal{E}^{g_2\to g_1}\|_\infty$. 
\end{proof}

%%%%%%%%%%%%%%%%%%%%%%%%%%%%%%%%%%%%%%%%%%%%%%%%%%%%%%%%%%%%%%%%%%%%%%%%%%%%%%%%%%%%%
%%%%%%%%%%%%%%%%%%%%%%%%%%%%%%%%%%%%%%%%%%%%%%%%%%%%%%%%%%%%%%%%%%%%%%%%%%%%%%%%%%%%%%

\section{Convergence of geodesics, Busemann functions  and invariant measures}\label{sec:ConvGeod}

In this section, we study the continuity of geodesics, Busemann functions, 
 and Bowen-Margulis measures under a Lipschitz perturbation 
of the metric with uniform negative curvatures. 

Let $(g_\epsilon)_{-1 \leq \epsilon \leq 1}$ be a family of metrics on $\tilde M$ with sectional curvatures satisfying $K_{g_\epsilon}\leq -a^2$, such that $\forall \epsilon>0$, at all $x\in \tilde M$, $\ds e^{-\epsilon} g_0 \leq g_\epsilon \leq e^\epsilon g_0$.

We first show that  the $g_\epsilon$-geodesic between two points at infinity converge uniformly in the Hausdorff topology of $\tilde M$ to the $g_0$-geodesic with same extremities, and that the Busemann functions of $g_\epsilon$ converge uniformly on compact sets to the Busemann functions of $g_0$.

When the variation of metrics is continuous in $\calC^1$-topology, 
this also implies that the Morse-correspondances $\Phi^{g_0\to g_\epsilon}$ and $\Psi^{g_0\to g_\epsilon}$ converge to the identity uniformly on compact sets in the $\calC^0$-topology of $S^g\tilde M$, and that the geodesic stretch $\mathcal{E}^{g_0\to g_\e}$ converges to $1$.  

Eventually, we show that under suitable assumptions, 
the Bowen-Margulis measures vary continuously in the weak-* topology.

%%%%%%%%%%%%%%%%%%%%%%%%%%%%%%%%%%%%%%%%%%%%%%%%%%%%%%%%%%%%%%%ùù
\subsection{Convergence of geodesics and Busemann functions }

The following lemma is a   classical and very useful consequence of the uniform upper bound on the  curvature.

\begin{lemm}\label{lem:Control-Hyperb}
Let $a>0$ and $(\tilde M,g)$ be a complete simply connected manifold 
with sectional curvatures satisfying $K_g\leq -a^2$. 
\begin{enumerate}
\item For all $C>0$, all $\xi \in \bd \tilde M$,
  $x,y\in \tilde M$ with $d^g(x,y)\leq C$, and $t\geq C$, if $x_t = \gamma_{x,\xi}(t)$, we have
$$
\left|\calB^g_\xi(x,y) - \left(d^g(x,x_t) - d^g(y,x_t)\right)\right|\le 2Ce^{-at}\,.
$$ 
\item For all $T, K, \alpha>0$,  
 for all $R\geq R_0=T-\frac{1}{a}\ln\frac{\alpha}{4Ke^K }$, 
if $(\gamma_1(t))_{t\in \bbR}$ and $(\gamma_2(t))_{t\in \bbR}$ are $g$-geodesics with 
$$
d^g(\gamma_1(-R), \gamma_2(-R))\leq K \quad \mbox{and} \quad d^g(\gamma_1(R), \gamma_2(R))\leq K\,,
$$
then  %there exists $s_0\in [-K, K]$ such that 
for all $t\in [-T, T]$,
$$
d^g(\gamma_1(t), \gamma_2 ) \leq \alpha\,.
$$
\end{enumerate}
\end{lemm}

\begin{proof} We will omit the subscript $g$ in the proof.
Let us first prove 1.  

Assume $d(x,y)\leq C$. We can also assume that $\calB_\xi(x,y)\geq 0$. Denote by $x'$ the unique point on $[x,\xi)$ such that $\calB_\xi(x', y)=0$. 
By convexity of the horoball, $d(x,x')\leq C$ and $d(x',y)\leq C$. 
Let $x_s$ (resp. $y_s$) be the points on $ [x',\xi)$ (resp. $[y,\xi)$ 
at distance $s$ of $x'$ (resp. $y$).
 It follows from \cite{HIH77} that for all $s\ge C$, 
$$
d(x_s, y_s) \leq d(x',y) e^{-a s}\le Ce^{-as}\,.
$$
Observe also that $\left|\calB^g_\xi(x,y) - \left(d^g(x,x_s) - d^g(y, y_s )\right)\right|=|\calB_\xi(x_s,y_s)|\le d(x_s,y_s)$, 
so that $\left|\calB^g_\xi(x,y) - \left(d^g(x,x_s) - d^g(y, x_s )\right)\right|\le 2d(x_s,y_s)\le 2Ce^{-as}$. 
 
To prove 2, denote by $x_s$ the point of $[\gamma_1(-R),\gamma_1(R)]$ at distance $s$ from $\gamma_1(-R)$, 
$y_s$ the point of $ [\gamma_1(-R),\gamma_2(R)]$ at distance $t$ 
from $\gamma_1(-R)$ and distance say $d_s$ from $\gamma_2(R)$ and $z_s$ 
the point of $ [\gamma_2(-R),\gamma_2(R)]$ at distance  $d_s$ from $\gamma_2(R)$. 
Observe immediately that $|d_s-2R+s|\le K$. 

By the above, we have $d(x_s,y_s)\le  d(x_{2R},y_{2R})e^{-as}$. 
But elementary considerations in the triangle $(x_{2R},y_{2R},\gamma_{2}(R))$ lead to 
$$
d(x_{2R},y_{2R})=d(\gamma_1(R),y_{2R})\le d(\gamma_1(R),\gamma_2(R))+d(\gamma_2(R),y_{2R})\le 2 K\,.
$$
Thus $d(x_s,y_s)\le 2Ke^{-as}$. 

Similarly we get $d(y_s,z_s)\le 2Ke^{-ad_s}\le 2Ke^Ke^{-a(2R-s)}$.
We deduce that 
$$
d(x_s,\gamma_2)\le d(x_s,z_s)\le 2Ke^{K}(e^{-as}+e^{-a(2R-s)}\,.
$$
Now, choose $R_0=T-\frac{1}{a}\ln\frac{\alpha}{4Ke^K}$. 
For $t\in [-T,T]$, we have $\gamma_1(t)=x_{R+t}$ and $R+t\ge R_0-T$ and $2R-(R+t)\ge R_0-T$, 
so that 
$$
d(\gamma_1(t),\gamma_2)\le d(\gamma_1(t),z_{R+t})\le 4Ke^K e^{-a(R_0-T)}\le \alpha\,.
$$
\begin{figure}[ht!]\label{fig:dist }
\begin{center}
 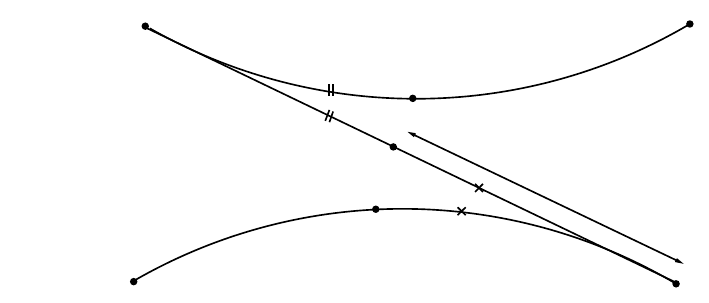 
\caption{Proof of Lemma \ref{lem:Control-Hyperb} } 
\end{center}
\end{figure} 
\end{proof}

Let us now show that the $g_\e$-geodesic segments converge to 
the $g_0$-geodesic segments in the Hausdorff topology of $\tilde M$. 

\begin{prop}\label{prop:Conv-Geod1}
Let $g_0$ be a complete metric on $\tilde M$ with $K_{g_0}\leq 0$. 
For all $0< \e \le 1$ small enough let $g_\e$ be a complete metric on $\tilde M$ 
such that at all $x\in \tilde M$,  $\ds e^{-\epsilon} g_0 \leq g_\epsilon \leq e^\epsilon g_0$.

Then for all $x,y\in \tilde M$, the  minimizing $g_\e$-geodesic $\gamma_\e$ 
joining $x$ to $y$  is contained in the $D_\epsilon$-neighbourhood 
of the $g_0$-geodesic $[x,y]_0$ from $x$ to $y$, 
with $\ds D_\e \leq \min(\sqrt{\e }d^{g_0}(x,y),C(g_1,g_0))$.
\end{prop}

\begin{proof} 
Let $g_0$ and $g_\e$ as above, and   $x,y\in \tilde M$. 
Set $L_0 = d^{g_0}(x,y)$ and $L_\e = d^{g_\e}(x,y)$. 
Let $\gamma_0 : [0, L_0]\to \tilde M$ and $\gamma_\e : [0, L_\e]\to \tilde M$ be minimizing geodesics 
from $x$ to $y$ respectively for $g_0$ and $g_\e$, 
parametrized with unit speed. Note that $\gamma_0$ is unique. 
Let $l\in [0, L_\e]$ be  such that
 $$\ds d^{g_0}(\gamma_\e(l), [x,y]_0) = \max_{t\in [0,L_\e]}d^{g_0}(\gamma_\e(l), [x,y]_0) = D_\e \,.$$ 
We call $z = \gamma_\e(l)\in \tilde M$. 
Consider the $g_0$-geodesic triangle with vertices $x, y, z$. 
Set $l_1 = d^{g_0}(x,z)$ and $l_2 = d^{g_0}(z,y)$. 

%\begin{figure}[ht!]\label{fig: }
%\begin{center}
 %\input{ .pdf_t} 
%\caption{ } 
%\end{center}
%\end{figure} 
%\end{proof}

We have
\begin{equation}\label{eq:ContGeod}
l_1 \leq \int_0^l \norm{ \dot \gamma_\e(t)}_{g_0} dt \quad 
\mbox{and} \quad l_2 \leq \int_l^{L_\e} \norm{\dot \gamma_\e(t)}_{g_0} dt.
\end{equation}
Since $\ds e^{-\epsilon} g_0 \leq g_\epsilon \leq e^\epsilon g_0$, 
we have $\norm{\dot \gamma_\e(t)}_{g_0}\leq e^{\e/2}$ for all $t\in [0, L_\e]$ 
and $\norm{\dot \gamma_0(t)}_{g_\e}\leq e^{\e/2}$ for all $t\in [0, L_0]$. 
Therefore, by Equation (\ref{eq:ContGeod}),
$$
L_\e \leq \int_0^L \norm{\dot\gamma_0(t)}_{g_\e}dt \leq e^{\e/2} L \quad \mbox{ and } 
\quad l_1 + l_2 \leq \int_0^{L_\e}\norm{\gamma_\e(t)}_{g_0} dt \leq e^\e L.
$$

Since $K_{g_0}\leq 0$, the distance $d^{g_0}$ satisfies $CAT(0)$-triangle comparison property 
(cf \cite{BH99} p161)\,: $D_\e$ is less than the height $\bar D$ from 
$\bar z$ of the comparison triangle $(\bar x, \bar y, \bar z)$ in the Euclidean plane 
with side lengths $d^{eucl}(\bar x, \bar y) = L_0$, $d^{eucl}(\bar x, \bar z) = l_1$ 
and $d^{eucl}(\bar y, \bar z) = l_2$. Moreover, for all such Euclidean triangles 
with $l_1+l_2 \leq e^\e L_0$, the height $\bar D$ is maximal 
if and only if $\ds l_1 = l_2 = \frac{e^\e L_0}{2}$. 
Therefore,
$$
D_\e^2 \leq \bar D^2  \leq \frac{e^{2\e} L_0^2}{4} - \frac{L_0^2}{4} \leq \e L_0^2
$$
as soon as $e^{2\e}-1\leq 4\e$, which ends the proof of Proposition \ref{prop:Conv-Geod1}.
\end{proof}

Proposition \ref{prop:Conv-Geod1} together with Lemma \ref{lem:Control-Hyperb} 
imply that when the curvatures have a uniform negative upper bound, 
the complete geodesics on $\tilde M$ converge uniformly for 
the $g_0$-Hausdorff topology under a variation of the metric. 
Let $a>0$ be fixed.

\begin{prop}\label{prop:Conv-Geod2}
Let $(g_\epsilon)_{-1 <\epsilon <1}$ be a family of metrics on $\tilde M$ 
with sectional curvatures satisfying $K_{g_\epsilon}\leq -a^2$, 
such that $\forall \epsilon\in (-1, 1)$, at all $x\in \tilde M$, 
$\ds e^{-\epsilon} g_0 \leq g_\epsilon \leq e^\epsilon g_0$.
Then there exists $\alpha : (-1, 1)\to [0, +\infty)$, 
with $\ds \lim_{\e \to 0} \alpha(\e) = 0$, such that for all $\e\in (-1, 1)$ 
and all $(\eta, \xi)\in \bd^2 \tilde M$, the $g_\e$-geodesic 
with extremities $\eta$ and $\xi$ is contained in the $\alpha(\e)$-neighbourhood 
of the $g_0$-geodesic with extremities $\eta$ and $\xi$.
\end{prop}

\begin{proof}First, recall (see section \ref{subsec:hopf}) 
that the geodesics for $g_0$ and $g_\e$ are at uniform bounded distance 
$C_1(g_0,g_\e)\le C_1(g_0,g_1)$. 
Let $\gamma_0$ be the $g_0$-geodesic from $\xi$ to $\eta$. 
Choose its origin $\gamma_0(0)$ arbitrarily. 
For any large $\rho>0$, we have 
 $d(\gamma_0(\pm \rho),\gamma_\e)\le C_1(g_0,g_1)$. 
Consider the $g_0$-geodesic segment $\gamma_1$ joining the nearest point 
to $\gamma_0(\rho)$ on $\gamma_\e$ with the nearest point to $\gamma_0(-\rho)$ on $\gamma_\e$. 
This geodesic segment has $g_0$-length equal to 
$2R=2\rho\pm 2C_1(g_0, g_1)$. 
Choose its origin in such a way that
$d^{g_0}(\gamma_0(\pm R),\gamma_1(\pm R))\le  2C_1(g_0,g_1)$.

For all $\alpha>0$,  Lemma \ref{lem:Control-Hyperb} applied with $K=2C_1(g_0,g_1)$, $\alpha/2$ 
and $T=1$ gives some $R_0>0$ 
such that when $R\ge R_0$, for all $t\in[-1,1]$, $d^{g_0}(\gamma_1(t),\gamma_0(t))\le \alpha/2$. 

By Proposition \ref{prop:Conv-Geod1}, $d^{g_0}(\gamma_1(0),\gamma_\e))\le 2R\sqrt{\e}$. 

Therefore, 
$$
d^{g_0}(\gamma_0(0),\gamma_\e)\le d^{g_0}(\gamma_0(0),\gamma_1(0))+d^{g_0}(\gamma_1(0),\gamma_\e)) 
\le \alpha/2+2R\sqrt{\e}\,.
$$

Choose $R\ge R_0$ and $\e>0$ such that $2R_0\sqrt{\e}\le \alpha/2$ to get $d^{g_0}(\gamma_0(0),\gamma_\e)
\le\alpha$. 
As the origin on $\gamma_0$ is arbitrary, the result follows.   
\end{proof}

Observe that, in the above proof, $\epsilon$ can be made relatively explicit. 
For $K=C_1(g_0,g_1)$, $T=1$ and $\alpha/2$ 
we get $R_0=1+\frac{2C_1(g_0,g_1)}{a}\ln \frac{2C_1(g_0,g_1)}{\alpha}$ 
and $\epsilon=\frac{\alpha^2}{16 R_0^2}$.

Moreover, our proof only uses $K_{g_0}\leq -a^2 <0$ and that for all $g_\e$, 
the $g_\e$-geodesic between two points at infinity is unique. 
The negative upperbound on the $K_{g_\e}$ does not need to be uniform.

\begin{prop}\label{prop:Conv-Busemann}
Let $(g_\epsilon)_{-1 \leq \epsilon \leq 1}$ be a family of complete metrics 
on $\tilde M$ with $K_{g_\epsilon}\leq -a^2$, such that for all $  \epsilon>0$, 
at all $x\in \tilde M$, $\ds e^{-\epsilon} g_0 \leq g_\epsilon \leq e^\epsilon g_0$.

Then the map $\calB^{g_\epsilon} : (x,y,\xi) \mapsto \calB^{g_\epsilon}_\xi(x,y)$ 
converges to $\calB^{g_0}$ as $\epsilon\to 0$, uniformly on compact sets of 
$\tilde M\times \tilde M \times \bd \tilde M$.
\end{prop}

\begin{proof} Any compact set $K\subset \tilde M$ is contained 
in some (noncompact) set of the form 
$\ds H_C = \left\{(x,y,\xi) \in \tilde M \times \tilde M \times \bd \tilde M ; d^{g_0}(x,y)\leq C\right\}$, 
for some $C>0$.   
It is enough to show that $\calB^{g_\epsilon} \to \calB^{g_0}$ as $\e\to 0$,
 uniformly on each $H_C$. 

Let $C>0$ be fixed. For all $\e\in(-1,1)$ and all $(x,y,\xi)\in H_C$,
$$
d^{g_\epsilon}(x,y)\leq 2d^{g_0}(x,y) \leq 2C.
$$
Let $\eta>0$ be fixed. Choose $x_t$ at distance $t$ from $x=x_0$ 
on the $g_0$-geodesic $(x,\xi)$, and let
$y_t$ be the point on the $g_0$-geodesic  $(y,\xi)$ 
such that $\calB^{g_0}_\xi(x,y)=0$. 
Let $x^\e_t$ be the projection of $x_t$ on the $g_\e$-geodesic 
from $x$ to $\xi$. Proposition \ref{prop:Conv-Geod2} ensures 
that $d^{g_\e}(x_t,x_t^\e)\le \alpha(\e)$.  
Let us write 
\begin{eqnarray*}
\left|\calB^{g_\e}_\xi(x,y)-\calB^{g_0}_\xi(x,y)\right| &\le&
\left|\calB^{g_\e}_\xi(x,y)-d^{g_\e}(x,x_t^\e)+d^{g_\e}(y,x_t^\e)\right|\\
&+&
\left| d^{g_\e}(x,x_t^\e)-d^{g_\e}(y,x_t^\e) -d^{g_\e}(x,x_t)+d^{g_\e}(y,x_t)\right|\\
&+&
\left|d^{g_\e}(x,x_t)-d^{g_0}(x,x_t)-d^{g_\e}(y,x_t)+d^{g_0}(y,x_t)\right| \\
&+& 
\left|\calB^{g_0}_\xi(x,y)-d^{g_0}(x,x_t)+d^{g_0}(y,x_t)\right|\,.
\end{eqnarray*}
For $t\ge 2C$, by Lemma \ref{lem:Control-Hyperb}, 
the last term on the right hand side is bounded from the above by 
$4Ce^{-at}$. 
For $t\ge 2Ce^\e+\alpha(\e)$, we also have $d^{g_\e}(x,x_t^\e)\ge 2C $ 
so that again by Lemma  \ref{lem:Control-Hyperb}, the first term is bounded from the above by 
$4Ce^{-ad^{g_\e}(x,x_t^\e}\le 4Ce^{\alpha(\e)}e^{-at/2}$. 
By triangular inequality, the second term is bounded from the above by $2\alpha(\e)$. 
The inequality $e^{-\e}g_0\le g_\e \le e^\e g_0$ allows to bound the third term by
$2(e^{\e}-1)(t+C)$.  

At last, we get 
$$
\left|\calB^{g_\e}_\xi(x,y)-\calB^{g_0}_\xi(x,y)\right| 
\le 4Ce^{\alpha(\e)}e^{-at/2}+2\alpha(\e)+2(e^{\e}-1)(t+C)+ 4Ce^{-at}\,.
$$
Let $\eta>0$ be fixed. 
Choose first $\e_0$ so that for $\e\le \e_0$, $\alpha(\e)\le 1$. 
Chose $t\ge 2C$ large enough to guarantee that the first and the last term are each 
bounded from the above by $\eta/4$. 
Choose $\e_1\le\e_0$ small enough to guarantee that for $\e\le \e_1$, $\alpha(\e)\le\eta/4$ and 
$2(e^{\e}-1)(t+C)\le\eta/4$.
Thus, $|\calB^{g_\e}_\xi(x,y)-\calB^{g_0}_\xi(x,y)|\le \eta$. 
This gives the desired result. 
\end{proof}
 
\begin{rema}\rm
  Eventhough   this section is written in a Riemannian setting, 
all the previous proofs apply verbatim to a family of distances 
$(d_\e)_{-1<\e<1}$ on $X$ such that for all $\e\in (-1, 1)$, 
the metric space $(X,d_\e)$ is CAT$(-1)$ and $\ds e^{-\e} d_0 \leq d_\e \leq e^\e d_0$.
 \end{rema}

%%%%%%%%%%%%%%%%%%%%%%%%%%%%%%%%%%%%%%%%%%%%%%%%%%%%%%%%%%%%%%%%%%%%%%%%%%%%%%ù
\subsection{Higher regularity, Morse correspondances and geodesic stretch}

In this section, we consider metrics $g_\e \to g_0$ in the $C^1$-topology. 
To emphasize the necessity of this assumption, 
observe that $g_\e\to g_0$ in the $C^0$-topology 
does not imply the convergence of the curvatures nor the convergence of the geodesic flow. 

In particular, one can "add mushrooms" on a hyperbolic manifold, 
and make the mushrooms as small as we want, and build a sequence 
of manifolds with many points of nonnegative curvature converging to a hyperbolic manifold. The geodesic flow of such $g_\epsilon$ will not converge in general to the geodesic flow of $g_0$.

%One can also build oscillating diffeomorphisms {\bf Samuel!! COMPLETER..}\\

In view of its importance in the sequel, recall the convergence that we shall use. 

\begin{defi}\label{def:Unif-Cont1}
A family $(g_\e)_{-1 \leq \e \leq 1}$ of complete Riemannian metrics on $\tilde M$ (or $M$)   
\emph{converges in the $C^1$-topology, uniformly on compact sets, to $g_0$ } if:
\begin{enumerate}
  \item $(g_\e)$ converges to $g_0$ uniformly on compact sets, i.e. for all compact sets $K\subset T\tilde M$, 
  $$\lim_{\e\to 0} \sup_{v\in K}\left|g_\e (v,v) - g_0(v,v)\right| = 0;$$
  \item the first derivatives of $g_\e$ also converge uniformly on compact sets to those of $g_0$.
%, i.e.  for all compact sets $K\subset T\tilde M$,
%  $$
%\lim_{\e \to 0} \sup \left\{ \left|\nabla^{g_\e}_X Y - \nabla^{g_0}_X Y\right|\; ; \; X,Y \mbox{ vect. fields with image in } K\right\}\,.
%$$
\end{enumerate}
\end{defi}
 By Theorem 2.79 of \cite{GHL}, it implies for all fixed $T>0$ 
the uniform convergence on compact sets of the geodesic flows $v \mapsto g_\e^T v$. 
As a consequence, we get the following result. 

\begin{theo}\label{theo:Conv-Param}
Let $(g_\epsilon)_{-1 < \epsilon < 1}$ be a family of metrics on $\tilde M$ 
with sectional curvatures satisfying $K_{g_\epsilon}\leq -a^2$, such that 
for all $\epsilon\in (-1, 1)$, at all $x\in \tilde M$,
$e^{-\epsilon} g_0 \leq g_\epsilon \leq e^\epsilon g_0$, and $g_\e\to g_0$
in the $\calC^1$ topology, 
uniformly on compact sets.  

Let $\ds \widetilde\Phi^{g_0\to g_\e}  $ and $\ds\widetilde\Psi^{g_0\to g_\e}$ 
be the Morse correspondances between $S^{g_0}\tilde M$ and $S^{g_\e} \tilde M$ 
defined in section \ref{subsec:Morse}. 
Then $\ds \Phi^{g_0\to g_\e} \to {\rm Id}$ and $\ds \Psi^{g_0\to g_\e}\to {\rm Id}$ 
uniformly   on all compact sets $K\subset S^{g_0}\tilde M$ in 
the uniform topology of $\calC^0(K, T\tilde M)$.
\end{theo}

\begin{proof} Let $K$ be a fixed compact set of $S^{g_0}\tilde{M}$ and 
$v\in K$, with $v_\pm^{g_0}$ the endpoints of its $g_0$-geodesic in 
$\partial\tilde{M}$. Denote by $(\gamma_0(t))_{t\in\mathbb{R}}$ the
 parametrization of this geodesic such that $\gamma_0'(0)=v$. 
Let $\gamma_\epsilon$ be the parametrization of the $g_\epsilon$-geodesic
 with same endpoints, with $v_\epsilon=\gamma_\epsilon'(0)=\Phi^{g_0\to g_\epsilon}(v)$. 

By Proposition \ref{prop:Conv-Busemann} and definitions 
from Section \ref{subsec:Morse}, uniform convergence of 
$\Psi^{g_0\to g_\epsilon}$ on compact sets will follow from 
the convergence of $\Phi^{g_0\to g_\epsilon}$. 
So let us prove the latter. 

We will  use the distance $d(w,w')=\sup_{t\in[0,1]} d^{g_0}(\pi(g_0^t w,g_0^t w')$ 
on $T\tilde{M}$ and show that for all $\alpha>0$, 
if  $\epsilon $ is small enough, for all $v\in K$ and 
$t\in[0,1]$, $d^{g_0}(\pi(g_0^t v), \pi(g_0^t v_\epsilon))\le \alpha$. 

Choose some $\alpha>0$. 
By Propositions \ref{prop:Conv-Geod2} and   \ref{prop:Conv-Busemann}, 
for $\epsilon$ small enough, uniformly in $v\in K$, and $t\in [-1,1]$, 
we know that 
$\gamma_\epsilon$ is in the $\alpha/2$-neighbourhood of $\gamma_0$,  
and $\gamma_\epsilon(t)$ is uniformly  close to $\gamma_0(t)$. 
It implies that $v_\epsilon=\gamma_\epsilon'(0)$ and $v_0=\gamma_0'(0)$ 
are uniformly close. 
 As $g_\epsilon\to g_0$ in the $C^1$-topology, uniformly on compact sets, 
it implies that for $\epsilon$ small enough, for all $t\in [-1,1]$, 
$ \pi(g_\epsilon^t(v_\epsilon))$ and $\pi(g_0^t(v_\epsilon))$ will stay $\alpha/2$-close. 
In particular, $\pi(g_0^t(v_\epsilon))$ will stay $\alpha$-close 
from $\gamma_0(t)$, for $t\in [-1,1]$. That is the desired convergence. 
\end{proof}

\begin{rema}\rm 
Adapting Theorem \ref{theo:Conv-Param} and the definition of 
the geodesic stretch in the setting of CAT$(-1)$ spaces would
 require a careful definition of the tangent bundle on such 
spaces with its topology, which we will not do here.
\end{rema}

%%%%%%%%%%%%%%%%%%%%%%%%%%%%%%%%%%%%%%%%%%%%%%%%%%%%%%%%%%%%%%%%%%%%%%%%%%%%%

Let us conclude this section by a key technical ingredient. 

\begin{theo}\label{theo:Conv-Stretch}
Let $(g_\epsilon)_{-1 < \epsilon < 1}$ be a family of metrics 
on $\tilde M$ with sectional curvatures satisfying $K_{g_\epsilon}\leq -a^2$, 
such that for all $ \epsilon\in (-1, 1)$, 
at all $x\in \tilde M$, $\ds e^{-\epsilon} g_0 \leq g_\epsilon \leq e^\epsilon g_0$, 
and $g_\epsilon\to g_0$ in the $\calC^1$-topology, uniformly on   compact sets.

Then uniformly on compact sets of $S^{g_0}\tilde M$, we have
$$
\limsup_{\e \to 0} \calE^{g_0\to g_\e}(v) \le 1\,.
$$
Moreover,  $$\calE^{g_0\to g_\epsilon}\to 1\quad m_\mu^{g_0}-\mbox{almost surely}\,.$$ 
\end{theo}

\begin{proof}
Observe that Lemma \ref{lem:Ineq-Knieper} gives the obvious upper bound 
$\ds \limsup_{\e\to 0}\calE^{g_0\to g_\e}\le 1$, uniformly on $S^{g_0}\tilde{M}$. 
For the same reason, $\ds \limsup_{\e\to 0}\calE^{g_\e\to g_0}\le 1$, 
uniformly on $S^{g_\epsilon}\tilde{M}$. 
By Corollary \ref{prop:transformation-mass}, one easily deduces that 
\begin{equation}\label{eqn:convergence-of-mass} 
\frac{\|m_\mu^{g_\epsilon}\|}{\|m_\mu^{g_0}\|}\to 1\quad\mbox{when}\quad \e\to 0
\end{equation}
Combined with the fact that $\ds \limsup_{\e\to 0}\calE^{g_0\to g_\epsilon}\le 1$,
 this implies in turn that $\calE^{g_0\to g_\epsilon}\to 1$ $m_\mu^{g_0}$-almost surely. 
\end{proof}

%%%%%%%%%%%%%%%%%%%%%%%%%%%%%%%%%%%%%%%%%%%%%%%%%%%%%%%%%%%%%%%%%%%%%%%%%%%%%%%%%%

\subsection{Narrow convergence of measures associated to a fixed geodesic current}

Recall that if $\mu$ is a $\Gamma$-invariant geodesic current and $g$ an admissible 
metric on $M$, we denote by $m_\mu^g$ the locally finite Radon measure 
on $S^gM$ whose lift to $S^g\tilde M$ is given by
$$
d\tilde m_\mu^g(v) = (H^g)^*( \mu\times dt).
$$ 
The results of the previous paragraph imply   the following fact.

\begin{prop}\label{coro:ConvEtroiteMmu}
Let $(g_\epsilon)_{-1 < \epsilon < 1}$ be a family of metrics on $\tilde M$ 
whose sectional curvatures satisfy $K_{g_\epsilon}\leq -a^2$, and such that 
for all $\epsilon\in (-1, 1)$, at all $x\in \tilde M$, 
$\ds e^{-\epsilon} g_0 \leq g_\lambda \leq e^\epsilon g_0$, and 
$g_\epsilon\to g_0$ in the $\calC^1$-topology, uniformly on compact sets. 
Let $\mu$ be a $\Gamma$-invariant geodesic current. 
Then the measures $ m_\mu^{g_\epsilon}$ 
converge to $m_\mu^{g_0}$   in the dual of bounded  continuous maps 
on $TM$ (i.e. in the narrow topology).
\end{prop}

\begin{proof}
By definition, for all $\epsilon\in (-1, 1)$ we have 
$\ds m_\mu^{g_\epsilon} = (\Psi^{g_0\to g_\epsilon})_*\left(\mathcal{E}^{g_0\to g_\epsilon}\times m_\mu^{g_0}\right)$. 
Therefore the weak-* convergence (in the dual of continuous compactly supported functions)  is an immediate consequence of Theorem \ref{theo:Conv-Param} and dominated convergence Theorem. 
%Theorem \ref{theo:Conv-Stretch}.

We also showed that $\|m^{g_\epsilon}_\mu\|\to \|m^{g_0}_\mu\|$, see Equation (\ref{eqn:convergence-of-mass}). 
It is classical that it implies the convergence of the above measures in the dual of
 bounded continuous functions. The result follows.  
\end{proof}

%%%%%%%%%%%%%%%%%%%%%%%%%%%

\subsection{Weak convergence of Bowen-Margulis measures}\label{ssec:Conv-BM}

We now show that, provided they are unique, the Bowen-Margulis measures 
are continuous in the weak-* topology under Lipschitz deformations of the metric.
  
\begin{prop}\label{prop:Conv-PattSull}
Let $(g_\epsilon)_{-1 \leq \epsilon \leq 1}$ be a family of metrics on $M$ 
with sectional curvatures satisfying $K_{g_\epsilon}\leq -a^2$, such that 
 for all $\epsilon\in (-1,1)$, $(\Gamma, g_\e)$ is divergent and  at all 
$x\in \tilde M$,$\ds e^{-\epsilon} g_0 \leq g_\epsilon \leq e^\epsilon g_0$.

Then for all $x\in \tilde M$, the Patterson-Sullivan measures for $g_\e$
 normalized at $o$ converge: $\ds \lim_{\e\to 0} \nu_x^{g_\e} = \nu_x^{g_0}$ 
in the weak-* topology, uniformly in $x$ on  compact sets of $\tilde M$.
\end{prop}

\begin{proof}
For all $\e\in (-1,   1)\setminus\{0\}$, the measure $\nu_o^{g_\e}$ is a 
probability measure  on $\Lambda_\Gamma$.
 Let $\ds \tilde \nu_o=\lim_{\e_i\to 0}\nu_0^{g_{\e_i}}$ be any of its weak limits. 
Define for all $x\in \tilde M$ a measure  $\tilde \nu_x$  on $\Lambda_\Gamma$ by 
$$
\frac{d\tilde \nu_x}{d\tilde \nu_o}(\xi) = e^{-\delta(g_0)\calB^{g_0}_\xi(o,x)}.
$$
It is a $\Gamma$-invariant, $\delta(g_0)$ conformal family of measures, 
normalized at $o$. By uniqueness of such a family, 
it coincides with $(\nu^{g_0}_x)_{x\in\tilde{M}}$. 
\end{proof}

Recall that $\mu_{BM}^g$ denotes the \emph{$g$-Bowen-Margulis geodesic current} 
on $\bd^2 \tilde M$ given by
$$
d\mu_{BM}^g(\eta, \xi) = 
d\nu^g_{x}(\eta)d\nu^g_{x}(\xi) = e^{-\delta(g)(\calB^g_\xi(o,x) + 
\calB^g_\eta(o,x))}d\nu^g_o(\eta)d\nu_o^g(\xi)\,,
$$
where $x$ is any point on the $g$-geodesic with endpoints $(\eta, \xi)$. 
We get the immediate corollary of Propositions \ref{prop:Conv-Busemann} 
and \ref{prop:Conv-PattSull}.

\begin{coro}\label{coro:Conv-MBM1} Under the same assumptions, 
in the weak-* topology of $\bd^2 \tilde M$, 
$\ds \lim_{\e\to 0} \mu_{BM}^{g_\e} = \mu_{BM}^{g_0}$.
\end{coro}

\begin{rema}\rm 
Once again,  Proposition \ref{prop:Conv-PattSull} and Corollary \ref{coro:Conv-MBM1} 
are still valid if we consider a family of $\Gamma$-invariant distances 
$(d_\e)_{\e\in(-1, 1)}$ on $\tilde M$ such that $(\tilde M, d_\e)$ 
is a CAT$(-1)$ and $\ds e^{-\e}d_0 \leq d_\e \leq e^\e d_0$ forall $\e\in (-1, 1)$.
\end{rema}

We end this section by  the convergence of  Bowen-Margulis measures. 

  \begin{theo}[Convergence of Bowen-Margulis measures]\label{theo:ContMBM}
Let $(g_\epsilon)_{-1 < \epsilon < 1}$ be a family of metrics on $\tilde M$ 
with sectional curvatures satisfying $K_{g_\epsilon}\leq -a^2$, 
such that for all $\epsilon\in (-1, 1)$, 
at all $x\in \tilde M$, $\ds e^{-\epsilon} g_0 \leq g_\epsilon \leq e^\epsilon g_0$ 
and $g_\epsilon\to g_0$ in the $\calC^1$-topology, uniformly on  compact sets. 
Assume that $\Gamma$ is divergent for all metrics $g_\e$. Then in the weak-* topology of $TM$,
$$
\lim_{\e \to 0} m_{BM}^{g_\e} = m_{BM}^{g_0}.
$$
  \end{theo}
   
\begin{proof}
Let $ \varphi$ be a continuous map with compact support on $TM$.
Write the difference
 $\ds \int_{TM}\varphi dm_{\mu_{BM}^{g_\epsilon}}^{g_\epsilon}-\int_{TM}\varphi dm_{\mu_{BM}^{g_0}}^{g_0}$ as 
$$
\left(\int_{TM}\varphi dm_{\mu_{BM}^{g_\epsilon}}^{g_\epsilon}-\int_{TM}\varphi dm_{\mu_{BM}^{g_\e}}^{g_0}\right)+
\left(\int_{TM}\varphi dm_{\mu_{BM}^{g_\epsilon}}^{g_0}-\int_{TM}\varphi dm_{\mu_{BM}^{g_0}}^{g_0}\right)
$$
By Corollary \ref{coro:Conv-MBM1}, the second difference converges to $0$. 

 Proposition \ref{prop:Change-Var-Flow} allows to rewrite the first difference as
 $$
\int_{TM}\varphi dm_{\mu_{BM}^{g_\epsilon}}^{g_\epsilon}-\int_{TM}\varphi dm_{\mu_{BM}^{g_\e}}^{g_0}
=\int_{TM}\left(\varphi\circ\Psi^{g_0\to g_\epsilon}\times{\cal E}^{g_0\to g_\epsilon}-\varphi\right)\, dm_{\mu_{BM}^{g_\epsilon}}^{g_0} \,.
$$ 

 By Corollary \ref{coro:Conv-MBM1}, 
$ m_{\mu_{BM}^{g_\epsilon}}^{g_0}$ converges weakly to 
$ m_{\mu_{BM}^{g_0}}^{g_0}$ in the dual of continuous functions with compact support. 

 By Theorem \ref{theo:Conv-Param}, as  $\ds \limsup_{\epsilon\to 0} {\cal E}^{g_0\to g_\epsilon}\le 1$, 
if $\varphi\ge 0$, we have 
$\ds \limsup_{\varepsilon\to 0}\varphi\circ\Psi^{g_0\to g_\epsilon}\times{\cal E}^{g_0\to g_\epsilon}-\varphi\le 0$. 
As the support of these maps is included in a fixed compact set,  
we deduce that 
$$
\limsup_{\epsilon\to 0}\int_{TM}\left(\varphi\circ\Psi^{g_0\to g_\epsilon}\times{\cal E}^{g_0\to g_\epsilon}-\varphi\right)\, dm_{\mu_{BM}^{g_\epsilon}}^{g_0} 
\le 0\,.
$$

Now, rewrite this first difference as 
$$
-\left(\int_{TM}\varphi dm_{\mu_{BM}^{g_\epsilon}}^{g_0}-\int_{TM}\varphi dm_{\mu_{BM}^{g_\e}}^{g_\epsilon}\right)
=-\int_{TM}\left(\varphi\circ\Psi^{g_\epsilon\to g_0}\times{\cal E}^{g_\epsilon\to g_0}-\varphi\right)\, dm_{\mu_{BM}^{g_\epsilon}}^{g_\epsilon} \,.
$$ 

Observe first that, by the same arguments used in the proof 
of Equation (\ref{eqn:convergence-of-mass}), 
the ratios of masses $\frac{\| m_{\mu_{BM}^{g_\epsilon}}^{g_0}\|}{\| m_{\mu_{BM}^{g_\epsilon}}^{g_\epsilon}\|}$ goes to $1$ when $\epsilon\to 0$.

For the same reason as above, $\limsup_{\epsilon\to 0}{\cal E}^{g_\epsilon\to g_0}\le 1$, 
so that for $\varphi\ge 0$, by Theorem \ref{theo:Conv-Param}, uniformly on $TM$, 
the limsup of $\varphi\circ\Psi^{g_\epsilon\to g_0}\times{\cal E}^{g_\epsilon\to g_0}-\varphi$ 
is nonpositive. By convergence of the ratio of masses mentioned above, 
and by convergence of $ m_{\mu_{BM}^{g_\epsilon}}^{g_0}$ to $ m_{\mu_{BM}^{g_0}}^{g_0}$, 
its integral also has a nonpositive limsup, and the sign minus in the above expression gives 
$$
\liminf_{\epsilon\to 0}\int_{TM}\varphi dm_{\mu_{BM}^{g_\epsilon}}^{g_\epsilon}-\int_{TM}\varphi dm_{\mu_{BM}^{g_\e}}^{g_0} \ge 0\,.
$$
The result follows. 
\end{proof}

%%%%%%%%%%%%%%%%%%%%%%%%%%%%%%%%%%%%%%%%%%%%%%%%%%%%%%%%%%%%%%%%%%%%%%%%%%%%%%%%%
%%%%%%%%%%%%%%%%%%%%%%%%%%%%%%%%%%%%%%%%%%%%%%%%%%%%%%%%%%%%%%%%%%%%%%%%%%%%%%%%%

 \section{Differentiability of the metric and topological entropies}\label{sec:DiffEnt}
  
In this section, we   show  differentiability of  topological and 
measure theoretic entropies  at $\e=0$, when 
along a variation   $(g_\e)_{\e\in(-1, 1)}$ of metrics of a negatively curved 
Riemannian manifold $(M = \tilde M/\Gamma, g_0)$. 
We will focus on two distinct situations. 

First, let $\mu$ be a $\Gamma$-invariant geodesic current on $\bd^2 \tilde M$, 
and for all $\e\in(-1, 1)$, let $m_\mu^{g_\e}$ be the associated invariant measure 
for the geodesic flow $(g_\e^t)$ (see Section \ref{sec:Hopf}). 
Assume that the total mass of $m_\mu^{g_0}$ is finite. 
We will show that the measure theoretic entropy 
$\e \mapsto h(m_\mu^{g_\e}, g_\e)$ is $\calC^1$, with explicit derivatives. 

We then focus on the topological entropy. 
Provided that  Bowen-Margulis measures of each geodesic flow $(g_\e^t)$ are finite, 
and that their masses vary  continuously, 
we show that  the topological entropy is also $\calC^1$, 
with a similar  formula for its derivative. 
The proofs are  similar in both situations, and  inspired from \cite{KKW91} and \cite{Tap11}.

\begin{defi}\label{def:unifC1}
Let $M$ be a (non-compact) manifold. We   say that a family of complete 
Riemannian metrics $(g_\e)_{\e\in (-1,1)}$ on $M$ converges to $g_0$ \emph{uniformly in the $\calC^1$-topology} 
  if\,:
\begin{enumerate}
\item $g_\e \to g_0$ in the $\calC^1$ topology, uniformly on compact sets, 
as in Definition \ref{def:Unif-Cont1};
\item there exists $\kappa>0$ such that for all $\e\in (-1, 1)$ 
and all $v\in TM$ with $\norm{v}_{g_0}\leq 1$,
$$
\left| \left.\frac{d}{ds}\right|_{s=\e}g_s(v,v)\right|\leq \kappa ;
$$
%  \item the family $\ds \e \mapsto \left.\frac{d}{ds}\right|_{s=\e}g_s$ 
%is uniformly continuous on $\ds \left\{v\in TM \; ; \; \norm{v}_{g_0}\leq 1\right\}$. 
%[Samuel] Inutile, ce item est une scorie de l'époque où j'utilisais une convergence uniforme à la place de la convergence dominée dans l'argument de continuité.
\end{enumerate}
\end{defi}

A $\calC^2$ variation of metric with compact support, 
or with non-compact support but uniformly bounded second derivative, 
is a typical example of such uniformly $\calC^1$ family. 
If $(g_\epsilon)_{\epsilon\in (-1, 1)}$ is such a uniformly $\calC^1$ family 
of complete metrics on $M$, one immediately see that there 
exists $B = B(C_1, \e)>0$ such that at all $x\in M$ and for all $\epsilon\in (-1, 1)$,
$$
e^{-B\epsilon}g_0 \leq g_\epsilon \leq e^{B\epsilon}g_0\,,
$$
which allows us to apply the results shown in the previous section.

%%%%%%%%%%%%%%%%%%%%%%%%%%%%%%%%%%%%%%%%%%%%

\subsection{Variation of metric entropy}

This paragraph is devoted to the proof of the following result, 
which seems to us new even in the compact case. 

\begin{theo}\label{theo:VarEntMet}
Let $b>a>0$, $\e>0$ and let $(g_\epsilon)_{\epsilon\in (-1, 1)}$ 
be a  family of complete metrics on $M = \tilde M/\Gamma$ 
 whose curvatures and derivatives of curvatures are uniformly bounded, and moreover such that 
for all $\epsilon\in (-1, 1)$ and at all points, $-b^2\leq K_{g_\epsilon}\leq -a^2$. 
Assume that
$g_\e\to g_0$ uniformly $\calC^1$. 
Let $\mu$ be a $\Gamma$-invariant geodesic current on $\bd^2 \tilde M$ 
such that $m_\mu^{g_0}$ is finite.
 
 Then the local  entropy $\ds \e\mapsto h_\Gamma^{loc}(m_\mu^{g_\e}, g_\e)$ 
of the $(g_\e^t)$-invariant measures $m_\mu^{g_\e}$ is differentiable at $\e=0$ 
with derivative given by
 $$
\left.\frac{d}{d\e}\right|_{\e = 0} h_\Gamma^{loc}(m_\mu^{g_\e}, g_\e) = 
-h_\Gamma^{loc}(m_\mu^{g_0}, g_0)\times 
\int_{S^{g_0}M} \left.\frac{d}{d\e}\right|_{\e = 0} \norm{v}^{g_\e}\frac{dm_\mu^{g_0}(v)}{\norm{m_\mu^{g_0}}}\,.
$$
\end{theo}

\begin{proof}
Let $\mu$ be a $\Gamma$-invariant geodesic current on $\bd^2 \tilde M$ 
such that $m_\mu^{g_0}$ is finite. 
It follows from Proposition \ref{prop:transformation-mass} that for all $\e\in (-1, 1)$, 
the measures $m_\mu^{g_\e}$ are finite. 
Moreover, by Corollary \ref{coro:ConvEtroiteMmu}, 
$\ds \lim_{\e\to 0} m_\mu^{g_\e} = m_\mu^{g_0}$ and 
$\ds \lim_{\e\to 0}  \frac{m_\mu^{g_\e}}{\norm{m_\mu^{g_\e}}}=
\frac{m_\mu^{g_0}}{\norm{m_\mu^{g_0}}}$  in the narrow topology.

By Theorem \ref{th-entropies},   if $g_1$ and $g_2$ are admissible metrics on $M$, 
we know that 
\begin{equation}\label{eq:ChangeEnt}
h_\Gamma^{loc}(m_\mu^{g_2}, g_2)= \int_{S^{g_2}M}\calE^{g_1\to g_2}(v)d\bar m_\mu^{g_2}(v) . h_\Gamma^{loc}(m_\mu^{g_1}, g_1).
\end{equation}

By Theorem \ref{theo:Conv-Stretch}, this implies that the local entropy 
$\ds   h_\Gamma^{loc}(m_\mu^{g_\e}, g_\e)$ converges to 
$h_\Gamma^{loc}(m_\mu^{g_0},g_0 )$ when $\e\to 0$.
 Moreover, (\ref{eq:ChangeEnt}) and Lemma \ref{lem:Ineq-Knieper} also imply that 
 $$
h_\Gamma^{loc}(m_\mu^{g_2}, g_2)\leq 
\int_{S^{g_2}M}\norm{v}_{g_1}d\bar m_\mu^{g_2}(v). h_\Gamma^{loc}(m_\mu^{g_1}, g_1).
$$
Applying it with $g_1=g_0$ and $g_2=g_\e$ first, and second with $g_1=g_\e$ and $g_2=g_0$, we get 
\begin{eqnarray*}
  h_\Gamma^{loc}(m_\mu^{g_0}, g_0)\left(\frac{1}{\int_{S^{g_0}M}\norm{v}^{g_\e} d\bar m_\mu^{g_0}(v)}-1\right)
&\leq&h_\Gamma^{loc}(m_\mu^{g_\e}, g_\e) - h_\Gamma^{loc}(m_\mu^{g_0}, g_0))\\
&\leq& h_\Gamma^{loc}(m_\mu^{g_0}, g_0)\left(\int_{S^{g_\e} M}\norm{v}^{g_0}d\bar m_\mu^{g_\e}(v)-1\right),
\end{eqnarray*}
which yields to 
\begin{eqnarray*}
h_\Gamma^{loc}(m_\mu^{g_0}, g_0)\frac{\int_{S^{g_0}M} \frac{\norm{v}^{g_0} - \norm{v}^{g_\e}}{\e} d\bar m_\mu^{g_0}(v)}{\int_{S^{g_0}M} \norm{v}^{g_\e}d\bar m_\mu^{g_0}(v)}
 &\leq & \frac{h_\Gamma^{loc}(m_\mu^{g_\e}, g_\e) - h_\Gamma^{loc}(m_\mu^{g_0}, g_0)}{\e} \\
& \leq & h_\Gamma^{loc}(m_\mu^{g_0}, g_0)\int_{S^{g_\lambda}M} \frac{\norm{v}^{g_0} - \norm{v}^{g_\e}}{\e} d\bar m_\mu^{g_\e}(v).
\end{eqnarray*}
Now, dominated convergence theorem, continuity of $\e\to h_\Gamma^{loc}(m_\mu^{g_\e},g_\e)$ at $\e=0$, and narrow convergence of $\frac{m^{g_\e}_\mu}{\|m^{g_\e}_\mu\|}$ towards $\frac{m^{g_0}_\mu}{\|m^{g_0}_\mu\|}$
give  
  $$
\left.\frac{d}{d\e}\right|_{\e = 0} h_\Gamma^{loc}( m_\mu^{g_\e}, g_\e) = 
-h_\Gamma^{loc}( m_\mu^{g_0}, g_0)\times \int_{S^{g_0}M} \left.\frac{d}{d\e}\right|_{\e = 0} \norm{v}^{g_\e}d\bar m_\mu^{g_0}(v)\,.
$$
This is the desired result. 
\end{proof}

\begin{rema}\rm In the above approach, we used lemma \ref{lem:Ineq-Knieper} 
to bound from the above $\calE^{g_1\to g_2}$ by the
smooth quantity $\|v\|^{g_2}$. A slightly more direct approach could have 
been to use in the above computations the fact that
 $$
\calE^{g_0\to g_\e}(v) = 
1 + \e \left.\frac{d}{d\e}\right|_{\e = 0} \norm{v}^{g_\e} + \e\alpha(\e, v),
$$
where $\alpha$ is uniformly bounded on $(-1, 1)\times S^{g_0}M$ and 
for all $v\in S^{g_0}M$, $\ds \lim_{\e\to 0} \alpha(v, \e) = 0$.
However, the proof of such estimate would be  long and technical, 
and would require additional assumptions of regularity on the variation of metric. 
\end{rema}  

%%%%%%%%%%%%%%%%%%%%%%%%%%%%%%%%%%%%%%%%%%%%%%
\subsection{Variation of topological entropy}

We now show differentiability of the topological entropy $h_{top}(g_\e)$ at $\e=0$. 
It  is not a corollary of Theorem \ref{theo:VarEntMet} since we have to consider 
 Bowen-Margulis geodesic currents $\mu_{BM}^{g_\e}$ depending on the metric $g_\e$. 
However, the strategy of proof is very similar, as by   Theorem \ref{theo:ContMBM}, 
$m_{BM}^{g_\e}\to m_{BM}^{g_0}$ in the weak-* topology. 
The only missing ingredient is the convergence of 
Bowen-Margulis measures in the dual of bounded  continuous functions. 
It is therefore required in the assumptions of Theorem \ref{theo:VarEnt}. 
We refer to Section \ref{sec:SPR} for the study of the large class 
of the so-called SPR manifolds, which will
satisfy this assumption.

\begin{theo}\label{theo:VarEnt}
Let $b>a>0$, and let $(g_\e)_{\e\in (-1, 1)}$ be a  family of complete metrics on $M$ such that
 \begin{enumerate}
  \item for all $\e\in (-1,1)$ and at all point, $-b^2\leq K_{g_\e}\leq -a^2$ ;
  \item $g_\e\to g_0$ uniformly in the $\calC^1$ topology as in Definition \ref{def:unifC1} ;
  \item for all $\e\in (-1, 1)$, the Bowen-Margulis measure $m_{BM}^{g_\e}$ of 
the geodesic flow $(g_\e^t)_{t\in \bbR}$ on $S^{g_\e} M$ has finite mass;
 \item the map $\e\to \norm{m_{BM}^{g_\e}}$ is continuous at $\e=0$.
 \end{enumerate}
 
 Then the entropy $\ds \e\mapsto h_{top}(g_\e)$ is $\calC^1$ at $\e=0$ with derivative given by
 $$
 \left.\frac{d}{d\e}\right|_{\e = 0} h_{top}(g_\e) =
 -h_{top}(g_0)\int_{S^{g_0}M} \left.\frac{d}{d\e}\right|_{\e = 0} \norm{v}^{g_\e}\frac{dm_{BM}^{g_0}(v)}{\norm{m_{BM}^{g_0}}}.
 $$
\end{theo}

\begin{proof} As the preceding one, our strategy of proof is inspired 
from \cite{KKW91} and \cite{Tap11}.  Corollary \ref{coro:EntBM} 
shows that if $g_1$ and $g_2$ are admissible metrics $M$ with finite 
Bowen-Margulis measures, then 
 $$
 h_{top}(g_2)\leq
 \int_{S^{g_2}M}\calE^{g_1\to g_2}(v)\,\frac{d  m_{BM}^{g_2}(v)}{\|m_{BM}^{g_2}\|} \cdot h_{top}(g_1)\leq
\int_{S^{g_2}M}\norm{v}_{g_1}\frac{d  m_{BM}^{g_2}(v)}{\|m_{BM}^{g_2}\|}\cdot h_{top}(g_1)\,,
$$
where the last inequality follows from Lemma \ref{lem:Ineq-Knieper}. 
Applying it to $g_\e$ and $g_0$ on both sides, we get   for all $\e\in (-1, 1)$,
%$$
%  h_{top}(g_0)\left(\frac{1}{\int_{S^{g_0}M}\norm{v}^{g_\e} d\bar m_{BM}^{g_0}(v)}-1\right)\leq h_{top}(g_\e) - h_{top}(g_0)\leq h_{top}(g_0)\left(\int_{S^{g_\e} M}\norm{v}^{g_0}d\bar m_{BM}^{g_\e}(v)-1\right),
%$$
%which yields 
\begin{eqnarray*}
h_{top}(g_0)\times \frac{\int_{S^{g_0}M} \frac{\norm{v}^{g_0} - \norm{v}^{g_\e}}{\lambda} \frac{d m_{BM}^{g_0}(v)}{\|m_{BM}^{g_0}\|}}{\int_{S^{g_0}M} \norm{v}^{g_\e}\frac{d m_{BM}^{g_0}(v)}{\|m_{BM}^{g_0}\|}}
& \leq & \frac{h_{top}(g_\lambda) - h_{top}(g_0)}{\lambda}\\
 &\leq& h_{top}(g_0)\times \int_{S^{g_\e}M} \frac{\norm{v}^{g_0} - \norm{v}^{g_\e}}{\e} \frac{d m_{BM}^{g_\e}(v)}{\|m_{BM}^{g_\e}\|} \,.
\end{eqnarray*}
 The assumptions of the Theorem are now exactly done to make the above integrals converge. 
We deduce that topological entropy is differentiable at $\e=0$, with 
  $$
\left.\frac{d}{d\e}\right|_{\e = 0} h_{top}(g_\e) =
 -h_{top}(g_0)\times \int_{S^{g_0}M} \left.\frac{d}{d\e}\right|_{\e = 0} \norm{v}^{g_\e}d\bar m_{BM}^{g_0}(v)\,.
$$ 
\end{proof}

%%%%%%%%%%%%%%%%%%%%%%%%%%%%%%%%%%%%%%%%%%%%%%%%%%%%%%%%%%%%%%
%%%%%%%%%%%%%%%%%%%%%%%%%%%%%%%%%%%%%%%%%%%%%%%%%%%%%%%%%%%%%%

\section{Entropy at infinity and Strongly Positively Recurrent groups}
\label{sec:SPR}

In this section, our goal is to propose a wide class of manifolds 
and metrics to which Theorem \ref{theo:VarEnt}
will apply. In view of this goal, proving differentiability of entropy, 
this section is apparently technical. 
However, the definition of this class of manifolds, and the related concepts studied here, 
is  probably one of the main novelties in our paper. 
We refer to \cite{STdufutur} for further results on these manifolds. 

We define the
  \emph{entropy at infinity} $\delta_\infty(M,g)$ of a 
negatively curved manifold $(M,g)$ 
(see Definition \ref{def:entropy-at-infini}),  
as the maximal exponential growth of the dynamics away
from any given (large) compact set. 
In particular, it is invariant under any $\calC^2$  compact 
perturbation of a negatively curved metric. 

We introduce the class of \emph{strongly positively recurrent} manifolds 
$(  M, g)$, defined as those negatively curved manifolds whose
 entropy at infinity is strictly smaller than the total topological 
entropy of the geodesic flow. 

As said in the introduction, the notion of {\em strong positive recurrence} 
appeared in \cite{Sarig01} in the
context of symbolic dynamics over an infinite alphabet, and has
been used later by some other authors among which \cite{BBG}. 
A former terminology  due to \cite{GS} was {\em stable positive recurrence}. 
This terminology could be more adapted to the kind of results that we prove here. 
In any case, as will be seen below and in \cite{STdufutur}, 
the acronyme SPR is perfectly adapted to the concept.  

The simplest nontrivial examples are 
geometrically finite hyperbolic manifolds, but this class also 
includes most  known examples of non-compact manifolds 
with negative curvature whose geodesic flow has a finite 
Bowen-Margulis measure, and many new ones (see section \ref{ssec:SPR-Ex}).  

Our main result is the following. 

\begin{theo}\label{theo:SPR-FinBM}\label{theo:ConvMassSPR} 
Let $(M,g_0)$ be a manifold with pinched negative curvature 
and bounded derivatives of the curvature. 

If $(M,g_0)$ is a strongly positively recurrent manifold, 
then the Bowen-Margulis measure of its geodesic flow is finite. 

Moreover, if $(g_\epsilon)_{\epsilon\in (-1,1)}$ is a uniformly 
$\mathcal{C}^1$-variation of smooth complete metrics on $M$ 
with pinched negative curvature and bounded derivatives of metrics, 
satisfying all $K_{g_\epsilon}\le -a^2<0$, then 
\begin{enumerate} 
\item For $\epsilon\in (-\epsilon_0,\epsilon_0)$ small enough,
all metrics $g_\epsilon$ are strongly positively recurrent. 
\item The mass of the associated (finite) Bowen-Margulis $m_{BM}^{g_\epsilon}$ 
varies continuously on $
(-\epsilon_0,\epsilon_0)$. 
\end{enumerate}
\end{theo}

The first part of this theorem (finiteness of Bowen-Margulis) has been 
proven independently and simultaneously by
A. Velozo \cite{Velozo} by a different approach. 

As a corollary, all assumptions of Theorem \ref{theo:VarEnt} hold for 
such a variation of metrics, so that we get the following result, 
which answers positively the question at the origin of this work. 

\begin{coro} 
Let $(g_\epsilon)_{\epsilon\in (-1, 1)}$ be a uniformly $\mathcal C^1$ 
family of complete metrics on the manifold $M$ with pinched negative curvature 
and bounded derivatives of the curvature. Assume that for all $\epsilon\in (-1, 1)$, 
$-b^2 \leq K_{g_\epsilon} \leq  -a^2$ for some $b>a>0$.
Then the entropy $\epsilon \mapsto h_{top}(g_\epsilon)$ is 
$\mathcal C^1$ around $\epsilon = 0$, and its derivative is given by 
 $$
 \left.\frac{d}{d\epsilon}\right|_{\epsilon = 0} h_{top}(g_\epsilon) =
 -h_{top}(g_0) \times 
\int_{S^{g_0}M} \left.\frac{d}{d\epsilon}\right|_{\epsilon = 0} \norm{v}^{g_\epsilon}\,
\frac{dm_{BM}^{g_0}(v)}{\norm{m_{BM}^{g_0}}}\,.
 $$
\end{coro}

In view of the length of this section, let us present the strategy of the proof. 

Heuristically, the SPR assumption allows to neglect the dynamical contribution of the complement of a large compact set to the dynamics.
 We develop this idea in  two introductive parts  
\ref{ssec:entropy-at-infinity} and \ref{subsec:entropy-at-infinity}, 
defining the growth of the fundamental group outside a compact set, 
the entropy at infinity and 
the class of Strongly Positively Recurrent manifolds.

In Subsection \ref{ssec:SPR-Ex} we provide an illustration of this concept, 
by describing different families of examples of SPR manifolds. 

A criterion of finiteness of the Bowen-Margulis measure from \cite{PS16} 
is used to prove the first part of Theorem \ref{theo:SPR-FinBM}. 
Subsection \ref{ssec:finite-bm} is devoted to this proof.

All entropies considered here are continuous for a negatively curved perturbation $(g_\epsilon)_{-1\leq \e\le 1}$ satisfying $e^{-\e} g_0 \leq g_\e \leq e^{\e}g_0$. 
Thus, the SPR assumption, which is the existence of a critical gap 
between the entropy at infinity and the topological entropy 
is stable under such small perturbations. 
And the existence of a large compact set concentrating the most part of the dynamics allows to prove that
its complement is of small Bowen-Margulis measure, uniformly in the perturbation. 
These ideas are developped in subsection \ref{ssec:entropy-variation}, where we prove that 
for a variation of a SPR metric as above, 
the mass of the Bowen-Margulis measures varies continuously. 

As said in the introduction, these results imply all Theorems stated in the introduction.
Theorem \ref{examples-SPR}  is an immediate consequence of section \ref{ssec:SPR-Ex} 
and the first part of Theorem \ref{theo:SPR-FinBM}. 
Theorem \ref{theo:SPRStable} is a reformulation of the second part of
Theorem \ref{theo:SPR-FinBM}. At last, our main result, Theorem \ref{theo:main}, 
follows from Theorems \ref{theo:VarEnt} and \ref{theo:SPRStable} (or \ref{theo:SPR-FinBM}).

%%%%%%%%%%%%%%%%%%%%%%%%%%%%%%%%%%%%%%%%%%%%%%%%%%%%%%%%%%%%%%
\subsection{Fundamental group outside a given compact set}
\label{ssec:entropy-at-infinity}

Let $(M,g)$ be a complete Riemannian manifold with pinched negative 
curvature $-b^2 \leq K_g \leq -a^2 <0$, 
whose fundamental group $\Gamma = \pi_1(M)$ is non-elementary.
 Let $p_\Gamma:\tilde M \to M$ be the universal covering map. 
Let $o\in \tilde M$ be a point, fixed once for all. 
For any set $W\subset M$, we will write $W^c = M\backslash W$.
 
\begin{defi}\label{def:nicepreimage}
Let $W\subset M$ be a compact pathwise connected set 
which is the closure of its interior, and
 whose boundary is piecewise $\calC^1$. 
A {\em nice preimage of $W$} is a 
 compact set   $\tilde W\subset \tilde M$  such that 
\begin{enumerate}
\item $p_\Gamma(\tilde W) = W$ and the restriction of $p_\Gamma$ 
to the interior of $\tilde W$ is injective ;
\item $\tilde W$   has a piecewise $\calC^1$ boundary.
\end{enumerate}
\end{defi}

\begin{rema}\rm We will often refer to and use results of \cite{PS16}. 
In this reference, ${\cal W}$ is a subset of $S^gM$ and $\tilde{\cal W}$ 
is an open set inside $p_\Gamma^{-1}({\cal W})$ such that 
$p_\Gamma:\tilde{\cal W}\to {\cal W}$ is onto. 
As we deal with several metrics and several unit tangent bundle, 
it is better  here to work with $W\subset M$.
 The reader can think to ${\cal W}$ as $S^gW$. 
The fact that $W$ is compact here, and ${\cal W}$ open in \cite{PS16} 
is  just a matter of taste in some arguments. 
%We will emphasize the places where we need to be careful with these hypotheses. 
%{\bf ENONCER QQ part}
\end{rema}

We gather in the following lemma elementary useful facts. 
\begin{lemm} Let $W$ be a compact pathwise connected set 
with piecewise $\calC^1$ boundary, 
which is the closure of its interior. 
\begin{enumerate}
\item A nice preimage $\tilde{W}$ of $W$  exists. 
\item If $W_2\supset W_1$, 
then they admit nice preimages $\tilde{W}_2\supset\tilde{W}_1$. 
\item  If $\gamma\neq {\rm id}$ then 
$\ds \gamma \cdot \inter{\tilde W}\cap \inter{\tilde W} = \emptyset$
\item The set 
$\{\gamma\in \Gamma \; ; \; \gamma \cdot \tilde W \cap \tilde W \neq \emptyset\}$ 
is finite. 
We call such $\gamma \tilde W$  the  \emph{adjacent elements} of $\tilde W$. 
\end{enumerate}
\end{lemm}
\begin{proof} 
 Choose some $w\in W$, lift it to $\tilde w\in p_\Gamma^{-1}(W)$ 
and construct the \emph{Dirichlet domain}  
$$\tilde W=\{z\in p_\Gamma^{-1}(W), \,\, 
\forall \gamma\in\Gamma, \,\,d^{g}(z,\tilde{w})\le d^{g}(z,\gamma \tilde{w})\}\,.
$$ 
It is a compact set with $\calC^1$-boundary which satisfies the properties stated in the lemma. 
If $W_1\subset W_2$, choose some 
$w\in p_\Gamma^{-1}(W_1)\subset p_\Gamma^{-1}(W_2)$. 
For $i = 1, 2$   the Dirichlet domains $\tilde{W}_1\subset \tilde{W}_2
\subset p_\Gamma^{-1}(W_i)$ satisfy Fact 2.  
\end{proof}

The following notion was introduced in  \cite{PS16}.

\begin{defi}\label{def:GammaW} Let $W\subset M$ be a 
compact set and $\tilde{W}$ a nice preimage of $W$. 
The \emph{fundamental group of $M$ out of $\tilde W$} is the set
$\ds \Gamma^g_{\tilde W} $ of elements $\gamma\in \Gamma$ 
such that there exists  $x,y\in \tilde W$ and a $g$-geodesic 
segment $c_{\gamma}$ joining $x$ to $\gamma y$ such that for all $h\in \Gamma$ ,
$$
c_\gamma\cap p_\Gamma^{-1}W=
c_\gamma\cap\Gamma.\tilde W \subset \tilde W \cup \gamma \cdot \tilde W\,.
$$
By compactness of $\tilde W$ we will always assume that $x,y\in \partial \tilde W$. 
\end{defi}

Heuristically, as explained in \cite{PS16}, $\Gamma_{\tilde W}^g$ 
represents loops $p_\Gamma([x,\gamma y])$ which go outside $W$ at the beginning, 
and come back to $W$ only at the end. 
This heuristics does not work perfectly, depending on the topology of $W$, 
for example when it has holes. 

The set $\Gamma^g_{\tilde W}$ will  help controlling 
what happens far at infinity. 
In particular it follows immediately from the definition 
that it is not sensitive to small compact perturbations 
of the metric $g$, as stated in the proposition below.  

\begin{prop}\label{prop:GammaWPerturb}
Let $(M,g_0)$ be a complete negatively curved metric 
and $W\subset M$ be a compact set,
with nice preimage $\tilde W$. 
For any proper compact subset $K\subset \inter{W}$ and 
any metric $g$ such that $g_1 = g_2$ outside $K$, 
we have $\Gamma_{\tilde W}(g_1) = \Gamma_{\tilde W}(g_2)$.
\end{prop}

By definition, $id_\Gamma\in \Gamma_{\tilde W}^g$, 
and $\gamma\in \Gamma_{\tilde W}^g$ 
iff $\gamma^{-1}\in \Gamma_{\tilde W}^g$. 
When $(M,g)$ is a geometrically finite manifold,
for suitable choice of $W$, 
$\Gamma^g_{\tilde W}$ is a union of 
groups. 
But in general, $\Gamma_{\tilde W}^g$ is not a group at all, 
as shown in the following proposition. 

\begin{prop}\label{prop:Basic-GammaW}
With the previous notations, let $W\subset M$ be a compact 
pathwise connected set with piecewise $\calC^1$ boundary and 
$\tilde W$ be a nice preimage of $W$.
  If $\gamma\in \Gamma_{\tilde W}$ is a hyperbolic element 
whose axis $A_\gamma$ intersects the interior of $\tilde{W}$, 
then there exists $N = N(\gamma)>0$ such that 
for all $n\geq N$, $\gamma^n\notin \Gamma^g_{\tilde W}$.
\end{prop}

\begin{proof} Let $\gamma\in \Gamma_{\tilde W}^g$ be such an hyperbolic element. 
Its axis $A_\gamma$ intersects $\tilde{W}$, and therefore also $\gamma\tilde{W}$ 
and all iterates $\gamma^n\tilde W$. 
Choose some $x_0\in A_\gamma\cap \inter{\tilde W}$ 
and let $d_0=d^g(x_0,\partial\tilde W)>0$. 
Let $x,y\in \tilde W$. 
By lemma \ref{lem:Control-Hyperb} (2), with $K=diam(\tilde W)$, $\alpha=d_0/2$, 
 we know that if 
$d^g(x,\gamma^n y)=n\ell^g(\gamma)\pm 2diam\tilde W \ge 2R_0$, 
all points in the middle interval of length $2T=\ell^g(\gamma)$ of the 
$g$-geodesic segment from $x$ to $\gamma^n y$ would be at distance
 less than $d_0/2 $ from $A_\gamma$, and therefore some of them would
 be inside $\gamma^k \tilde W$, for some $1\le k\le n-1$. 
This implies $n\ell^g(\gamma)\le 2R_0+2diam(\tilde W)$, 
which proves the proposition. 
\end{proof}

The set $\Gamma_{\tilde W}$ depends on $W$ and the choice of its preimage $\tilde W$, 
but not too strongly as illustrated by the following proposition. 

\begin{prop}\label{prop:GammaWPreimage}
\begin{enumerate}
\item Let $W\subset M$ be a compact set (with piecewise $\calC^1$ boundary), 
and $\tilde{W}$ be a nice preimage. let $\alpha\in\Gamma$.  
Then $\ds \Gamma^g_{\alpha \tilde W}=\alpha \Gamma^g_{\tilde W}\alpha^{-1}$. 
\item  If $W_1$ and $W_2$ are compact sets of $M$ 
(with piecewise $\calC^1$ boundary) 
such that $W_1 \subset \inter{W_2}$ with respective nice preimages
$\tilde W_1\subset \tilde W_2$, 
there exists $k\geq 1$ and $\alpha_1,..., \alpha_k\in \Gamma$ such that 
$$
\Gamma_{\tilde W_2}\subset \bigcup_{i, j = 1}^k \alpha_i\Gamma_{\tilde W_1}(\alpha_j)^{-1}\,.
$$
\item If $\tilde W_1$ and $\tilde W_2$ are nice preimages of $W$, 
then there exists a finite set 
$\{\alpha_1,..., \alpha_k\}\subset \Gamma$ such that
$$
\Gamma_{\tilde W_2}\subset \bigcup_{i,j = 1}^p \alpha_i \Gamma_{\tilde W_1}(\alpha_j)^{-1}.
$$
\end{enumerate}
\end{prop}

%As a  corollary, we get 
%\begin{coro} Let $W_1$ and $W_2$ be two compact sets of $M$
%(with piecewise $\calC^1$ boundary) such that $W_1 \subset \inter{W_2}$.  
%Then for \emph{all} nice preimages $\tilde W_1$ of $W_1$ and $\tilde W_2$ of $W_2%$, 
%there exists $k\geq 1$ and $\alpha_1,..., \alpha_k\in \Gamma$ such that 
%$$
%\Gamma_{\tilde W_2}\subset \bigcup_{i, j = 1}^k \alpha_i\Gamma_{\tilde W_1}(\alpha_j)^{-1}.
%$$
%\end{coro}

\begin{proof} The first item of the proposition  is obvious. 
Let us show 2. Set 
$$
D = 2{\rm diam}(W_2) \quad \mbox{and} \quad \eta =
 \inf\left\{d^g(w, \partial W_2) \; ; \; w\in W_1\right\}>0.
$$
Let $\gamma\in \Gamma^g_{\tilde W_2}$.
There exist $x_2,y_2 \in \partial \tilde W_2$ 
such that the $g$-geodesic segment 
$[x_2,\gamma y_2]$ intersects $\Gamma \tilde W_2$ 
only in $\tilde W_2\cup \gamma\tilde W_2$. 
Now, choose some $x_1,y_1\in \partial \tilde W_1$.  

By Lemma \ref{lem:Control-Hyperb}, 
there exists $L = L(D, \eta)>0$ and 
$R = R(D, \eta)>2L$ such that for all $x_1, y_1, x_2,  y_2\in \tilde M$ 
with $d^g(x_1, x_2)\leq D$, $d^g(y_1, y_2)\leq D$ and $d^g(x_2, y_2)\geq R$, 
the $g$-geodesic segment $(x_1, y_1)$ is contained 
in the $\frac \eta 2$-neighbourhood of $(x_2, y_2)$ 
except inside the balls $B^g(x_1, L)$ and $B^g(y_1, L)$. 

Let $\alpha_1,..., \alpha_k\in \Gamma$ be the (finitely many)  elements 
such that $d^g(\tilde W_1, \alpha_i  \tilde W_1)=
\inf\{d^g(a,b),\,a\in\tilde W_1,b\in \alpha_i\tilde W_1\} \leq L$.

Let $\tilde W_1 \subset \tilde W_2$ be included nice preimages 
of $W_1$ and $W_2$, and let $\gamma\in \Gamma_{\tilde W_2}$ 
such that $d^g(o,\gamma  o)\geq R + 2D$. 
Then there exists $x_2,y_2\in \bd \tilde W_2$ such that 
$(x_2, \gamma  y_2)$ does not intersect $p_\Gamma^{-1}(W_2)$. 
By construction there exists $x_1, y_1 \in \tilde W_1$ such that 
$d^g(x_1, x_2)\leq D$ and $d^g(y_1, y_2)\leq D$. 
The geodesic $(x_1, \gamma\cdot y_1)$ is $\frac \eta 2$-close to the 
geodesic $(x_2, \gamma\cdot y_2)$ outside 
the balls $B^g(x_1, L)$ and $B^g(\gamma\cdot y_1, L)$, 
hence does not intersect $p_\Gamma^{-1}(W_1)$ except maybe in these balls.  
Thus, there exist $\alpha_i,\alpha_j$ in the above finite set, 
such that the geodesic segment $(x_1,\gamma y_1)$
does not intersect $\Gamma \tilde{W}_1$ between $\alpha_i\tilde{W}_1$ 
and $\gamma\alpha_j\tilde{W}_1$. Therefore, 
$\alpha_i^{-1}\gamma\alpha_j\in\Gamma_{\tilde W_1}$ or in other words, 
$$
\gamma\in \alpha_i \Gamma_{\tilde W_1} \alpha_j^{-1}\,.
$$ 
The proof of the last item is similar, and we let it to the reader. 
\end{proof}

%%%%%%%%%%%%%%%%%%%%%%%%%%%%%%%%%%%%%%%%%%%%%%%%%%%%%%%%%%%%%%%%%ù

\subsection{Entropy at infinity}\label{subsec:entropy-at-infinity}

\begin{prop}\label{prop:EntW}
Let $W\subset M$ be a compact set and $\tilde W$ a nice preimage of $W$. 
The critical exponent $\delta_{W^c}(g)$
of the Poincar\'e series 
$\ds \sum_{\gamma \in \Gamma_{\tilde W}}e^{-s d^g(o,\gamma \cdot o)}$ is equal to  
$$
 \delta_{W^c}(g)=
\limsup_{R\to\infty} \frac{\log\#\left\{\gamma \in \Gamma_{\tilde W} ,\,\, 
d^g(o, \gamma  o)\leq R\right\}}{R}
$$
 and does not depend on the choice of a nice preimage
 $\tilde W\subset \tilde M$ of $W$ nor $o\in \tilde M$. 
We  call it the \emph{entropy out of $W$} of $(M,g)$.
\end{prop}

\begin{proof}  
It follows from the triangular inequality that $\delta_{W^c}(g)$ 
does not depend on the choice of $o$. 
Let us show that it does not depend on the choice of preimage. 
Let $\tilde W_1$ and $\tilde W_2$ be two nice preimages of $W$. 
By Proposition \ref{prop:GammaWPreimage}, there exists $k\geq 0$ and 
$\alpha_1,..., \alpha_k\in \Gamma$ such that 
\begin{equation}\label{eq:GammaIncl}
\Gamma_{\tilde W_2} \subset \bigcup_{i,j = 1}^k \alpha_i  \Gamma_{\tilde W_1}  \alpha_j ^{-1}.
\end{equation}
Set 
$$
D = \max \left\{d^g(w, o) \; ; \; 
w\in \tilde W_2 \cup \bigcup _{i,j = 1}^k \alpha_i  \Gamma_{\tilde W_1} (\alpha_j)^{-1}\right\}\,.
$$
Define  for $i = 1, 2$ and $R>0$,
$$
\Gamma_{\tilde W_i}(R) = 
\left\{\gamma\in \Gamma_{\tilde W_i} \; ; \; d^g(o, \gamma  o)\leq R\right\}\,.
$$
It follows from (\ref{eq:GammaIncl}) and   triangular inequality that for all $R>0$,
$$
\Gamma_{\tilde W_2}(R) \subset 
\bigcup_{i,j = 1}^k \alpha_i  \Gamma_{\tilde W_1}(R+2D) (\alpha_j)^{-1}, 
$$ and 
therefore $\ds\#\,\Gamma_{\tilde W_2}(R) \leq k^2 \#\,\Gamma_{\tilde W_1}(R+2D)$.
This gives immediately  
$$
\limsup_{R\to +\infty} \frac{1 }{R} \log \#\Gamma_{\tilde W_2}(R)  
 \leq \limsup_{R\to +\infty} \frac{1 }{R} \log \#\Gamma_{\tilde W_1}(R)\,.
$$
By symmetry, the reverse inequality also holds, and the result follows. 
\end{proof}
 
\begin{prop}\label{prop:DeltaWProp}
Let $(M,g)$ be a complete negatively curved metric. 
\begin{enumerate}
\item For any proper compact subset $K\subset \inter{W}$ and any metric $g_2$ 
such that $g_1 = g_2$ outside $K$, we have $\delta_{W^c}(g_1) = \delta_{W^c}(g_2)$.
\item For all compact sets $W_1,W_2$ such that $W_1\subset \inter{W_2}\subset M$, 
we have $\delta_{W_1^c}\geq \delta_{W_2^c}$.
\end{enumerate}
\end{prop}

\begin{proof}
Item 1 follows from Proposition \ref{prop:GammaWPerturb}. 
Item 2  can be proven similarly to  Proposition \ref{prop:EntW}, 
thanks to Proposition \ref{prop:GammaWPreimage}.
\end{proof}

For a global variation of the metric (i.e. beyond $W$), even small, 
the behaviour of $\delta_{W^c}(g)$ is not clear 
since the set $\Gamma_{\tilde W}$ depends on the metric. 

\begin{defi}\label{def:entropy-at-infini}
The \emph{entropy at infinity} of $(M,g)$ is
$$
 \delta_\infty(g) = \inf\,\{ \,\delta_{W}(g), \,\,W\subset M \,\, \,\, \mbox{compact set}\,\,\}\,.
$$
\end{defi}

Proposition \ref{prop:DeltaWProp} implies the following natural 
characterization of the entropy at infinity.

\begin{prop}
Let $(M,g)$ be a complete negatively curved manifold and 
$(W_i)_{i\in \mathbb N}$ be an increasing exhaustion of $M$ by compact sets. 
% : for all $i\geq 0$, $W_i\subset \inter{W_{i+1}}$ and $M = \cup_i W_i$. 
Then
$$
\delta_\infty(g) = \lim_{i\to \infty}\delta_{W_i^c}(g)\,.
$$
Moreover, it is invariant under any negatively curved 
perturbation of the metric with compact support.
\end{prop}

This entropy at infinity is a dynamical analogous to the bottom of 
the essential spectrum of the Laplacian in spectral geometry. 
We will use this fact in some of the examples given in Section \ref{ssec:SPR-Ex}.

\begin{defi}
The complete manifold $(M,g)$ is called \emph{strongly/stably positively recurrent (SPR)}, 
if $\delta_\infty(g) < \delta_\Gamma(g)$. 
We will also call this property a \emph{critical gap at infinity}.
\end{defi}

By definition, if $(M,g)$ is strongly positively recurrent, 
there exists a compact set $W\subset M$ such that $\delta_{W}< \delta_\Gamma$.

\begin{rema}
The reader may have noticed that the definition of $\Gamma_{\tilde U}$ 
given in \cite{PS16} p.4 is slightly different to ours, 
since it is written for an open set $\tilde U$ which projects onto $U = \inter W$. 
Nevertheless, these definitions almost coincide in the following sense. 
Let $W\subset M$ be a compact set with nice preimage 
$\tilde W\subset \tilde M$, let $\tilde U\subset \tilde M$ 
be an open set which projects onto $U = \inter{W}$. 
Let $\Gamma_{\tilde U}$ be defined as in \cite{PS16}, 
and $\Gamma_{\tilde W}$ be defined as above. 
Then there exists $\alpha_1,..., \alpha_k\in \Gamma$ such that 
$$
\Gamma_{\tilde U}\subset \bigcup_{i,j = 1}^k \alpha_i \Gamma_{\tilde W}(\alpha_j)^{-1} \quad \mbox{and} \quad 
\Gamma_{\tilde W}\subset \bigcup_{i,j = 1}^k \alpha_i \Gamma_{\tilde U}(\alpha_j)^{-1}.
$$
Therefore $\Gamma_{\tilde U}$ and $\Gamma_{\tilde W}$ 
have the same critical exponent and all results stated in \cite{PS16} 
to characterize the finiteness of Gibbs measures in terms 
of $\Gamma_{\tilde U}$ are also valid for our definition of $\Gamma_{\tilde W}$.
\end{rema}
 
%%%%%%%%%%%%%%%%%%%%%%%%%%%%%%%%%%%%%%%%%%%%%%%%%%%%%%%%%%%

%%%%%%%%%%%%%%%%%%%%%%%%%%%%%%%%%%%%%%%%%%%%%%%%%%%%%%%%%%%%%

\subsection{Examples of SPR manifolds}\label{ssec:SPR-Ex}

We present here three classes of SPR manifolds. 
The first examples are geometrically finite manifolds 
with critical gap studied in \cite{DOP}. 
Schottky products furnish also plenty of examples, 
generalizing the examples of \cite{Peigne}. 
At last, we describe examples inspired by Ancona's examples in \cite{Ancona}. 

These examples are almost the only known examples 
of non-compact manifolds with finite Bowen-Margulis measure.
To our knowledge, the only exception is the construction of 
Peign\'e \cite {Pei11} of  geometrically finite manifolds 
with finite Bowen-Margulis measure but without critical gap, see \cite{Pei11}.

\medskip
\subsubsection{Geometrically finite manifolds with critical gap}

The convex core $CC(M)\subset M$ is the image on $M$ of 
the convex hull of the limit set $\Lambda_\Gamma$ inside $\tilde M$. 
The nonwandering set $\Omega\subset S^gM$  of the geodesic flow is
 the set of vectors $v\in S^gM$ such that $v^\pm \in \Lambda_\Gamma$. 
By  definition, $\Omega\subset S^g CC(M)$.
A parabolic subgroup $\calP$ of $\Gamma$ is a subgroup 
which fixes a point at infinity, and therefore
stabilizes any horoball ${\cal H}$ centered at this point. 

A cusp is the image on $M$ of such a horoball. 

The manifold $M$ is \emph{geometrically finite} if its convex core 
can be written as a finite union 
$$CC(M)=C_0\sqcup C_1\sqcup \dots \sqcup C_K\,,$$
where $C_0$ is a compact set and the $C_i$ are finitely many cusps, 
images through $p_\Gamma$ of 
horoballs ${\calH}_i$ stabilized by parabolic subgroups $\calP_i$ of $\Gamma$.  
The complete reference
on such manifolds is \cite{Bowditch}. 
Parabolic subgroups have a positive critical  exponent.
The preimage on $\tilde{M}$ of a cusp $C_i$ is  the orbit of an horoball ${\calH}_i$, 
and the stabilizer of any horoball $\gamma {\calH}_i$ is conjugated 
to the stabilizer ${\calP}_i$ of ${\calH}_i$ in $\Gamma$. 

A \emph{convex-cocompact manifold} is a geometrically finite manifold without cusps; 
in other words, it is a manifold whose convex core is compact. 

\begin{prop}\label{prop:DeltaInfty-GeoFin} 
Let $(M,g)$ be a  manifold with pinched negative curvature. 
If $(M,g)$ is convex-cocompact, its entropy at infinity vanishes.  
If $(M,g)$ is geometrically finite with $k$ cusps represented by parabolic subgroups 
$\calP_1,..., \calP_k\subset \Gamma$, then
$$
\delta_\infty(g) = 
\max\left\{\delta_{\calP_1}(g),..., \delta_{\calP_k}(g)\right\}\,.
$$
In particular, a geometrically finite manifold  is strongly positively recurrent if and only if 
$$
\max\left\{\delta_{\calP_1}(g),..., \delta_{\calP_k}(g)\right\} < \delta_\Gamma\,.
$$
\end{prop}

This condition is precisely the \emph{critical gap} criterion 
introduced by Dalbo, Otal and Peign\'e in \cite{DOP}. 
It is satisfied  in particular by 
locally symmetric geometrically finite manifolds 
and their small compact $\calC^2$ perturbations. 
The notion of SPR manifold allows to  generalize
many  results of \cite{DOP} and others on geometrically finite manifolds 
to all strongly positively recurrent manifolds.
\medskip

The result follows immediately from  Proposition \ref{prop:Gamma-Infty-GeoFin} below.

\begin{prop}\label{prop:Gamma-Infty-GeoFin}  Let $(M,g)$ be a  manifold with pinched negative curvature. 

 If $(M,g)$ is convex-cocompact and $W$ is a compact set such that $CC(M)\subset \inter{W}$,
 then $\ds \Gamma_{\tilde W}$ is finite.

If $(M,g)$ is geometrically finite with $k$ cusps,
then  there exists a compact set $W\subset  M$   with nice preimage $\tilde W$, 
 a finite set $\Gamma^0_{\tilde W}$, finitely many elements $\alpha_1,\dots,\alpha_N\in \Gamma$, 
 and parabolic subgroups 
$\calP_1,..., \calP_k\subset \Gamma$  such that
$$
\Gamma_{\tilde W}(g) = 
\Gamma_0(\tilde W)\cup \bigcup_{i,j}\alpha_i\left(\calP_1\cup...\cup\calP_k\right)\alpha_j^{-1}\,.
$$
\end{prop}

\begin{proof} Assume first   $(M,g)$ be convex-cocompact and 
$CC(M)\subset \inter{W}$. Let $D$ be the diameter of $W$ and 
$\ds \eta =
 \inf\left\{d^g(w, \partial W) \; ; \; w\in {\rm CC}(M)\right\}>0$.
Let $\gamma\in\Gamma_{\tilde{W}}$, $x,y\in\partial\tilde W$  and 
choose $x_1,y_1\in \tilde{CC(M)}$ such that $d(x,x_1)\le D$ and $d(y,y_1)\le D$. 
By Lemma \ref{lem:Control-Hyperb}, there exists some $R_0$ depending on $D,\eta$ such that if 
$\ell^g(\gamma)\ge R_0$, there exists some $z\in(x,\gamma y)$, $z_1\in (x_1,\gamma y_1)$ 
 such that $d^g(z,z_1)\le \eta/2$. But $\tilde{CC(M)}$ is convex, 
so that $z_1\in \tilde{CC(M)}$ and $z$ is at distance $\eta/2$ 
of $\tilde{CC(M)}$ and therefore inside $\Gamma\tilde W$. 
Thus, $\gamma\notin \Gamma_{\tilde W}$.  
Therefore, all elements of $\Gamma_{\tilde W}$ have bounded length less than $R_0$, 
so that $\Gamma_{\tilde W}$ is finite, included in $\{\gamma\in \Gamma,\ell^g(\gamma)\le R_0\}$. 
 
\medskip
Assume now that $M$ is geometrically finite with cusps, 
and let $CC(M)=C_0\sqcup(\sqcup_{i=1}^k) $ be 
a decomposition of the convex core into a compact part and finitely many disjoint cusps. 
Let $W\subset M$ be a compact set such that $\inter{W}\supset CC(M)$. 
Choose some nice preimage $\tilde{W}$ and disjoint 
horoballs ${\cal H}_i$, $1\le i\le k$ whose boundary intersect $\tilde W$. 
Let $\calP_i$ be the stabilizer of $\calH_i$ in $\Gamma$.
  
Let  $\gamma\in\Gamma_{\tilde W}$ be such that $\ell^g(\gamma)\ge R_0$ 
and $x,y\in\partial\tilde W$. As noticed above, by Lemma \ref{lem:Control-Hyperb}, 
the geodesic segment $(x,\gamma y)$ is (except at the beginning and the end, 
inside balls $B^g(x,L)$ and $B^g(\gamma y,L)$) in the $\eta/2$ neighbourhood of $\tilde{CC(M)}$. 
As already said in \cite{PS16}, if $\gamma\in\Gamma_{\tilde W}$, 
except for a bounded amount of time at the beginning and the end, the geodesic segment 
$p_\Gamma(x,\gamma y)$ has to leave the compact part $C_0$ 
and enter in some  cusp $C_i$. 
Therefore, there exists a finite set $\{\alpha_1,\dots\alpha_N\}$ 
such that for some $1\le i,j\le N$, 
the geodesic segment $(\alpha_i x, \gamma\alpha_j y)$ 
stays in some horoball ${\calH}_l$.  
As in the proof of Proposition \ref{prop:GammaWPreimage}, 
one deduces that  $\Gamma_{\tilde W}\subset \Gamma_{\tilde W}^0\cup \cup_{i,j,l}\alpha_l\calP_i\alpha_j^{-1}$ 
with $\Gamma_{\tilde W}^0\subset \{\gamma\in \Gamma,\ell^g(\gamma)\le R_0\}$ 
as in the convex-cocompact case. 
\end{proof}

%%%%%%%%%%%%%%%%%%%%%%%%%%%%%%%%ù

\subsubsection{Schottky products}

We present now a family of geometrically infinite examples first studied in \cite{Peigne}.  
Let $G$ and $H$ be discrete groups of isometries of 
a complete manifolds $(\tilde M,g)$ with pinched negative curvature. 
They are {\em in Schottky position} if there exists disjoint compact 
sets $U_G, U_H\subset \tilde M\cup \bd \tilde M$ such that for all 
$g\in G\backslash \{{\rm id}\}$ and all $h\in H\backslash \{{\rm id}\}$, we have
$$
g\left((\tilde M\cup \bd \tilde M)\backslash U_G\right) \subset U_G \quad \mbox{and} \quad h\left((\tilde M\cup \bd \tilde M)\backslash U_H\right)\subset U_H.
$$
In particular,  by Klein's ping-pong argument, 
 they generate  a free product: $\Gamma = \langle G, H \rangle = G*H$. 
The entropy at infinity behaves nicely under Schottky products, 
as shown by the following theorem.

\begin{theo}\label{theo:Schottky}
Let $G$ and $H$ be discrete groups of isometries of a complete manifold 
$(\tilde M,g)$ with pinched negative curvature which are in Schottky position.
Let $\Gamma =  \langle G, H \rangle = G*H$. 
Denote respectively by $M_\Gamma=\tilde{M}/\Gamma$, $M_G=\tilde{M}/G$ and $M_H=\tilde{M}/H$ the associated quotient manifolds endowed with the quotient metric induced by $g$. 
Then 
$$
\delta_\infty(M_\Gamma) = 
\max\left\{\delta_\infty(M_G), \delta_\infty(M_H)\right\}\,.
$$
\end{theo}

As an immediate corollary, we get the following result. 

\begin{coro}
Let $G$ and $H$ be discrete groups of isometries of a complete manifold 
 $(\tilde M,g)$ with pinched negative curvature which are in Schottky position. 
Let $M_G, M_H$, and $M_{G*H}$ be
the quotient manifolds.  Their critical exponents satisfy 
\begin{equation}\label{eqn:schottky}
\delta_{G*H}\ge \max\left\{\delta_G,\delta_H\right\}\ge
 \max\left\{\delta_\infty(M_G),\delta_\infty(M_H)\right\}=\delta_\infty(M_{G*H})
\end{equation} In particular, 
\begin{enumerate}
\item if $G$ and $H$ are Strongly Positively Recurrent, then $G*H$ is also.

\item if $\ds \delta_{G*H}>\max\{ \delta_\infty(M_G), \delta_\infty(M_H)\}$, 
then $G*H$ is strongly positively recurrent.
\end{enumerate}
In  both cases $(\tilde M/\Gamma, g)$ has a finite Bowen-Margulis measure.
\end{coro}

It was originally shown by M. Peign\'e in \cite{Pei11} that 
if $\ds \delta_\Gamma>\max\{ \delta_G, \delta_H\}$ 
then $(\tilde M/\Gamma, g)$ has a finite Bowen-Margulis measure. 
The above corollary with Theorem \ref{theo:SPR-FinBM}
 guarantees this finiteness under a weaker condition. 

It was shown in \cite{DOP} that if $G\subset \Gamma$ is a {\em divergent} subgroup, 
then $\delta_G < \delta_\Gamma$. We get therefore the following corollary.
\begin{coro}
Let $G,H$ be discrete \emph{divergent} groups of isometries 
of a complete manifolds $(\tilde M,g)$ with pinched negative curvature 
which are in Schottky position. 
Then $\Gamma =  \langle G, H \rangle = G*H$ is strongly positively recurrent.
\end{coro}

This last corollary allows a lot of topologically infinite examples. 
For instance, if $G$ and $H$ are discrete subgroups of the group of isometries of 
the hyperbolic space, whose limit set are not the whole boundary, 
they can be settled in Schottky position by taking suitable conjugation 
with hyperbolic elements. If $G$ and $H$ are $\bbZ$-covers of convex-cocompact groups, 
they are divergent and their Schottky product gives   
a SPR manifold, hence with finite Bowen-Margulis measure, 
whose fundamental group is not even finitely generated.

\medskip

\begin{figure}[ht!]\label{fig:schottky} 
\begin{center}
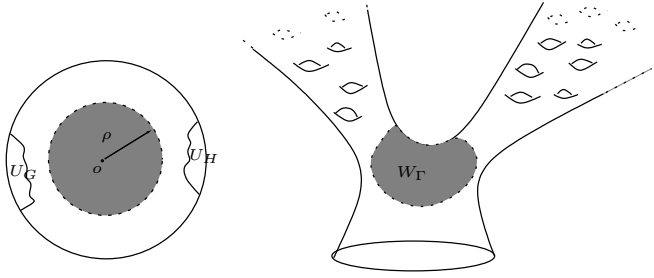 
\caption{Schottky manifold} 
\end{center}
\end{figure}

\begin{proof}
Our proof relies on the ideas of Section 9 of \cite{PS16}. 
Let $U_G$ and $U_H$ be the sets ensuring the Schottky position of $G$ and $H$. 
Since they are compact in $\tilde M \cup \bd\tilde M$ 
and since $K_g\leq -a^2 < 0$, a key point is that there exists $\rho>0$ such that 
all geodesics from $U_G$ to $U_H$ intersect the ball $B^g(o, \rho)$. 
Moreover, without loss of generality, we can assume that 
the point $o$ is neither in $U_G$ nor in $U_H$. 
%Denote respectively by $\Gamma_G$ and $\Gamma_H$ the subsets of elements of $\Gamma$ which can be written respectively as 
%$\gamma=g\gamma'$ or $\gamma=h\gamma'$, for some $g\in G, h\in H,\gamma'\in \Gamma$, so that $\Gamma=\Gamma_G\sqcup\Gamma_H\sqcup\{id\}$. 
%As in \cite[Lemma 9.4]{PS16}, we can take $R>0$ large enough so that for all $g\in\Gamma_G$, $h\in\Gamma_H$, 
%the geodesic segment $(go,ho)$ intersects the ball $B^g(o,R)$. 

Let $M_\Gamma = \tilde M/\Gamma$, $M_G = \tilde M/G$ and 
$M_H = \tilde M/H$. 
Let $p_\Gamma : \tilde M\to M_\Gamma$, $p_G : \tilde M\to M_G$ and 
$p_H : \tilde M\to M_H$ be the associated covering maps. 

For all $R\ge \rho$, define   $W^R_\Gamma = p_\Gamma(B^g(o,R)) \subset M_\Gamma$,
$W^R_G= p_G(B^g(o,R)) \subset M_G$ and $W^R_H= p_H(B^g(o,R)) \subset M_H$. 
Let $\tilde W^R_\Gamma,\tilde W^R_G,\tilde W^R_H\subset \tilde M$ be 
nice preimages (Dirichlet domains viewed from $o$) of $W^R_\Gamma,W^R_G,W^R_H$, 
respectively for the actions of $\Gamma$, $G$, $H$. 
By definition of $W^R_\Gamma,W^R_G,W^R_H$ 
and of a Dirichlet domain, 
one easily checks that they all lie inside $B^g(o,R)$. 
Moreover, as $p_G$ and $p_H$ are intermediate covers between $\tilde M$ and $M_\Gamma$, we have
$o\in \tilde W^R_\Gamma\subset \tilde W^R_G\cap \tilde W^R_H\subset B^g(o,R)$. 
  
%As they are fundamental domains for the respective actions of $\Gamma, G,H$ on their orbits, by definition of a Schottky product, we can assume {\bf NON!! UTILE ???} \\
%that 
%$$
%o\in\tilde W_\Gamma\subset \tilde W_G\cap \tilde W_H\subset B^g(o,R)\setminus (U_G%\cup U_H)\,.
%$$
 
Let $\Gamma_{\tilde W^R_\Gamma}\subset \Gamma$, $G_{\tilde W^R_G}\subset G$ 
and $H_{\tilde W^R_H}\subset H$ be the fundamental groups  respectively of 
$\Gamma, G$ and $H$ respectively out of $\tilde W^R_\Gamma, \tilde W^R_G, \tilde W^R_H$, 
according to Definition \ref{def:GammaW}.  

A key fact is the following.
\begin{lemm}[Pit-Schapira, \cite{PS16}]\label{lem:SchottkyDom}
For all $R>0$, there exists a finite set $S\subset \Gamma$ such that
$$
\Gamma_{\tilde W^R_G}\subset G \cup H \cup S.
$$
\end{lemm}
It implies that $\ds 
\delta_\infty(M_\Gamma) \leq \delta_{(W^R_\Gamma)^c}(\Gamma) \leq \max\{\delta_G, \delta_H\}\,,
$
and therefore, if $\delta_\Gamma >  \max\{\delta_G, \delta_H\}$ 
then $\Gamma$ is strongly positively recurrent. 

We precise this inclusion in the following lemma, 
which implies immediately Theorem \ref{theo:Schottky}.

\begin{lemm}
With the previous notations, for all $R\ge \rho +1$, there exist  
%a finite set $\alpha_1,..., \alpha_k\in \Gamma$
a finite set $F$ such that
$$
 \Gamma_{\tilde W^{2R}_\Gamma}\subset S\cup G_{\tilde W^{2R}_G}\cup H_{\tilde W^{2R}_H}\subset 
S\cup F \cup \bigcup_{\alpha,\beta\in F} \alpha \Gamma_{\tilde W^R_\Gamma} \beta.
$$ 
\end{lemm}

\begin{proof}
Let us first show the left inclusion. 
It follows from the previous lemma that 
$
\ds \Gamma_{\tilde W^{2R}_\Gamma} \subset G \cup H \cup S.
$
Moreover, 
$$
G\cdot \tilde W^{2R}_G = G\cdot B^g(o,2R)\subset \Gamma\cdot \tilde W^{2R}_\Gamma 
= \Gamma\cdot B^g(o,2R).
$$
For each $\gamma\in \Gamma_{\tilde W^{2R}_\Gamma}\cap  G$,
 there exist $x,y\in \tilde W^{2R}_\Gamma\subset \tilde W^{2R}_G$ such that 
$$
[x, \gamma  y]\cap \Gamma. \tilde W^{2R}_\Gamma \subset 
\tilde W^{2R}_\Gamma \cup \gamma. \tilde W^{2R}_\Gamma\,,\quad \mbox{whence} \quad 
[x, \gamma y]\cap G\cdot \tilde W^{2R}_G \subset \tilde W^{2R}_G \cup \gamma  \tilde W^{2R}_G\,,
$$
so that $\gamma\in G_{\tilde W^{2R}_G}$. 
It shows that $\Gamma_{\tilde W^{2R}_\Gamma}\cap G\subset G_{\tilde W^{2R}_G}$. 

Similarly, $\ds \Gamma_{\tilde W^{2R}_\Gamma}\cap H \subset H_{\tilde W^{2R}_H}$.

\medskip

Let us now prove the right inclusion. 
We want to show that there exists a finite set $F\subset \Gamma$ such that 
$\ds 
 G_{\tilde W^{2R}_G}\subset F\cup \bigcup_{\alpha,\beta\in F} \alpha \Gamma_{\tilde W^R_\Gamma} \beta,
$
the case of $H_{\tilde W^{2R}_H}$ being similar. 

Define $F_{\lambda R}$ as $F_{\lambda R}=\{\gamma\in\Gamma, \gamma B^g(o,\lambda R)\cap B^g(o,\lambda R)\neq \emptyset\}$. 

First observe that for $\lambda \ge 2$, we have
\begin{equation}\label{inclusion-technique}
\tilde W^{2R}_\Gamma \subset \tilde W^{2R}_G\subset \cup_{\alpha\in F_{\lambda R}} \alpha.\tilde W^{R}_\Gamma\,.
\end{equation}
Let $g\in G_{\tilde W^{2R}_G}$, $g\notin F$.  
By definition, there exist $x,y\in \tilde W^{2R}_G$
such that $(x,gy)$ intersects $G.\tilde W^{2R}_G$ only 
in $\tilde W^{2R}_G$ and $g \tilde W^{2R}_G$. 
We will show that $(x,gy)$ intersects $\Gamma.\tilde W^R_\Gamma=\Gamma.\tilde W^R_G$ 
only inside $\tilde W^{2R}_G$ and $g\tilde W^{2R}_G$. 
By  equation (\ref{inclusion-technique}), 
as in the proof of Proposition \ref{prop:GammaWPreimage}, it will imply that
$g\in \cup_{\alpha,\beta\in F_{\lambda R}} \alpha\Gamma_{\tilde W^R_\Gamma}\beta$. 
In fact, we will show that if $(x,gy)$ intersects some $\gamma.\tilde W^R_\Gamma$, 
then either  $\gamma$ or $g^{-1}\gamma$ is in the finite 
set $F$, so that by the same argument, 
$g\in \cup_{\alpha,\beta\in F_{\lambda R}} \alpha \Gamma_{\tilde W^R_\Gamma}{\beta}$. 
%{\bf Il ne faut pas passer \`a $2\lambda R$ dans cette inclusion finale ??}\\

By contradiction, assume that
the geodesic segment  $(x,gy)$ intersects $\gamma \tilde W^R_\Gamma$,
 with $\gamma\neq id, g$, and $\gamma, g^{-1}\gamma\notin F_{\lambda R}$. 
In particular, $d(o,\gamma o)> 2\lambda R$ and $d(go,\gamma o)>2\lambda R$. 
As $g\in G_{\tilde W^{2R}_G}$, 
we know that $\gamma \notin G$. 
Denote by $z_\gamma\in (x,gy)$ the closest point to 
$\gamma o$ in $(x,y)\cap \tilde W^R_\Gamma$. 
By the above, we have $\ds d(x,z_\gamma)\ge d(o,\gamma o)-3R\ge (2\lambda -3)R$. 

By definition of a Schottky product, as $o\notin U_G\cup U_H$,
either $\gamma o \in  U_G$ or
 $\gamma o\in U_H$.
Assume first that $\gamma o\in U_H$. Recall that $g o\in U_G$. 
Therefore, the geodesic segment $(\gamma o,go)$ 
intersects the ball $B(o,\rho)$. 
As $d(\gamma o,go)\ge 2\lambda R$ and $d(\gamma o, z_\gamma)\le R$, 
$d(go, gy)\le 2R$, the geodesic segment $(z_\gamma,gy)$ 
intersects the ball $B(o,\rho+2R)$. Let $w_\gamma$ be a point in this intersection.  
Therefore, we get 
$\ds d(x,w_\gamma)\le d(x,o)+d(o,w_\gamma)\le 4R+\rho\le 5R$. 
However, $d(x,w_\gamma)\ge d(x,z_\gamma)> (2\lambda-3)R$, 
which leads to a contradiction as soon as $\lambda\ge 4$.\\

Therefore, the first case holds, $\gamma o\in U_G$, so that $\gamma$ has a reduced form as
$\gamma=g'h'\gamma'$, with $g'\in G\setminus\{id\}$, $h'\in H\setminus\{id\}$, $\gamma'\in \Gamma$. We will distinguish the cases $g'\in F$ and $g'\notin F$. 

If $g'\notin F$, consider the segment $[(g')^{-1}o,h'\gamma'o]$. 
It goes from $U_G$ to $U_H$ so that it intersects the ball $B^g(o,\rho)$. 
It follows that $[o,\gamma o]$ intersects $g'.B(o,\rho)$ at a point $y$ with 
$d(o,y)\ge 2\lambda R-\rho$. By Lemma \ref{lem:Control-Hyperb}, for $\lambda$ large enough, 
the point $y$ is at distance less than $\rho$ from the geodesic segment $(x,\gamma o)$, 
and therefore at distance less than $R+1$ from the geodesic segment $(x,z_\gamma)$. 
Thus, we deduce that $(x,z_\gamma)$ intersects the ball $g'B(o,\rho+R+1)$. 
As we assumed $R\ge \rho+1$, this ball is included in 
$g'B(o,2R)\subset G.\tilde W^{2R}_G$. 
Moreover, as $\gamma'\notin F$, this intersection $(x,z_\gamma)\cap g'B(o,2R)$ is 
disjoint from $\tilde W^{2R}_G$, and as $\gamma \notin F$, 
and the intersection is between $x$ and $z_\gamma$, 
this intersection is also disjoint from $g.\tilde W^{2R}_G$. 
This is a contradiction with the hypothesis $g\in G_{\tilde W^{2R}_G}$. 

It remains the case $g'\in F$, which implies in particular $g'\neq g$. 
Consider in this case the 
geodesic segment $[h'\gamma' o,(g')^{-1}go]$. 
It goes from $U_H$ to $U_G$, so that it intersects the ball $B^g(o,\rho)$. 
It follows that $[\gamma o,go]$ intersects $g'B(o,\rho)$. 
The same arguments on $[z_\gamma, gy]$ instead of $[x,z_\gamma]$ 
lead once again 
to a contradiction with the hypothesis $g\in G_{\tilde W^{2R}_G}$.
 %{\bf Faut-il en dire plus ?}\\

It concludes the proof, for $F=F_{\lambda R}$, for some $\lambda\ge 4$ determined by the use of Lemma \ref{lem:Control-Hyperb}.  
\end{proof}
\end{proof}

%%%%%%%%%%%%%%%%%%%%%%%%%%%%%%%%%%%%%%%%%%%%%%%%%%%%%%%%%%%

\subsubsection{Ancona-like examples}

On hyperbolic manifolds, the dynamics is strongly related 
to the spectrum of the Laplacian.
 In particular, a well-known theorem of Patterson and Sullivan
 relates the entropy $\delta_\Gamma$ of $M = \bbH^{n+1}/\Gamma$ 
with the bottom of the spectrum of the Laplacian $\lambda_0(M)$:

\begin{theo}[Patterson 1976, Sullivan 1979 - 1987]
Let $M = \bbH^{n+1}/\Gamma$ be a complete hyperbolic manifold. 
If $\delta_\Gamma<\frac n 2$, then $\lambda_0(M) = \frac{n^2}{4}$. 
If $\delta_\Gamma>\frac n 2$, then $\ds \lambda_0(M) = \delta_\Gamma(n-\delta_\Gamma)$.
\end{theo}
 
We present now a family of surfaces inspired by examples of  
Ancona \cite{Ancona}, which is particularly easy to handle 
using the entropy at infinity introduced before. 

\begin{theo} There are plenty of examples of geometrically infinite hyperbolic surfaces 
$S=\bbH^2/\Gamma$ with  $\delta_\infty(S)<\delta_\Gamma(S)$. 
\end{theo}
By Theorem \ref{theo:SPR-FinBM}, all these examples have 
finite Bowen-Margulis measure. 
We postpone a more detailed study of these examples and 
relations between SPR property and spectrum of the Laplacian to our future paper  \cite{STdufutur}.

\begin{figure}[ht!]\label{fig:ancona} 
\begin{center}
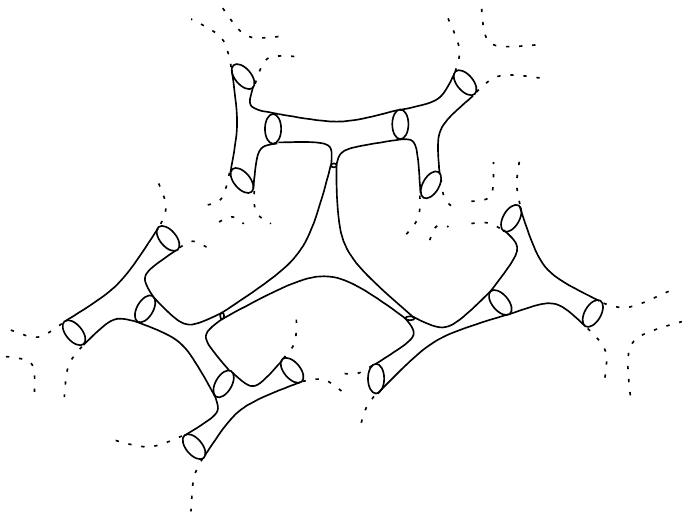 
\caption{SPR surface} 
\end{center}
\end{figure} 

\begin{proof}
Let $N= \bbH^2 / \Gamma$ be a complete hyperbolic surface with $\frac{n}{2}<\delta_\Gamma<n$.  
Denote by $g_0$ its metric.
For example, 
$N$ can be build as a nonamenable regular cover of a compact hyperbolic surface $N_0$. 
In any pair of pants decomposition of $N$, choose finitely many pairs of pants $P_1,\dots P_K$. 
Change the metric of $N$ to a metric $g_\epsilon$, which is equal to $g_0$ far from the pants $P_i$, 
and modified in the neighbourhood of the $P_i$ 
 by shrinking the lengths of the boundary geodesics of the pants $P_i$ to a length $\epsilon$. 
Let $\Gamma_\epsilon$ be a discrete group such that the new hyperbolic surface $(N,g_\epsilon)$ is isometric
to $\bbH^2/\Gamma_\epsilon$. 

As the perturbation is compact, for all $\e>0$, $\delta_\infty(g_\e) = \delta_\infty(g_0)<n$.
An elementary computation (see for example \cite[Prop. II.2 (ii)]{CoCo} )
gives $\ds \lim_{\e \to 0} \lambda_0(N, g_\e) = 0$, 
therefore $\ds \lim_{\e\to 0} \delta_{\Gamma_\e} = n$. 
This implies that for $\e>0$ small enough, $(N, g_\e)$ has a critical gap at infinity~: 
$\delta_{\Gamma_\e} > \delta_{\Gamma_0}\geq \delta_\infty(g_\e)$.
\end{proof}

%%%%%%%%%%%%%%%%%%%%%%%%%%%%%%%%%%%%%%%%%%%%%%%%%%%%%%%%%%

\subsection{SPR manifolds have finite Bowen-Margulis measure}
\label{ssec:finite-bm}
 This paragraph is devoted to the proof of the first part of
 Theorem \ref{theo:SPR-FinBM}\,: if  $(M,g)$ is a strongly positively recurrent manifold, then the Bowen-Margulis measure of its geodesic flow has finite mass. 

This finiteness result had been shown in \cite{DOP} on geometrically finite manifolds, under the assumption that $ \ds \max\left\{\delta_{\calP_1}(g),..., \delta_{\calP_k}(g)\right\} < \delta_\Gamma$, which is exactly the SPR assumption in the geometrically finite context, although they did not introduce this concept.

As said earlier, this result (finiteness of Bowen-Margulis measure) has been obtained independently, by a different approach, in \cite{Velozo}.

Our proof will rely on the following theorem shown in \cite{PS16}.

\begin{theo}[Pit-Schapira \cite{PS16}]\label{theo:Fin-BM-PS16}
Let $(M,g)$ be a complete manifold with negative curvatures. 
Then the Bowen-Margulis measure of $(M,g)$ is finite if and only if 
 $\Gamma = \pi_1(M)$ is divergent and
there exists a compact set $W\subset M$ with nice preimage $\tilde W$ 
such that $\Gamma_{\tilde W}$ satisfies 
$$
\sum_{\gamma\in \Gamma_{\tilde W}}d(o,\gamma  o)e^{-\delta_\Gamma d(o,\gamma  o)}<+\infty\,.
$$
\end{theo}

Let $(M,g)$ be a complete strongly positively recurrent manifold: 
there exists a compact set $W\subset M$ such that $\delta_{W^c}(g)< \delta_\Gamma(g)$. 
The second condition 
$\ds \sum_{\gamma\in \Gamma_{\tilde W}}d(o,\gamma  o)e^{-\delta_\Gamma d(o,\gamma  o)}<+\infty $ is then automatically satisfied for $\tilde W$ a nice lift of $W$. 
Therefore, Theorem \ref{theo:SPR-FinBM}  follows immediately from the following.

\begin{theo}\label{theo:SPR-Div}
Let $(M,g)$ be a strongly positively recurrent manifold. 
Then its fundamental group $\Gamma$ is divergent.
\end{theo}
 
We give first the strategy of the proof. 
 Let $(M,g)$ be a SPR manifold, with $\Gamma=\pi_1(M)$. 
It follows from Hopf-Tsuji-Sullivan theorem (see \cite[p.18]{Roblin}) 
that $\Gamma$ is divergent if and only if any Patterson-Sullivan measure $\nu^g_o$ 
(cf Section \ref{ssec:Conv-BM}) gives full measure to the radial limit set $\Lambda_\Gamma^r$. 

Theorem \ref{theo:SPR-Div}   follows from a careful study of $\Lambda_\Gamma^r$. 
More precisely, if $W$ is a nice set with $\delta_{W^c}<\delta_\Gamma$ 
and nice lift $\tilde W$, we introduce a kind 
of limit set $\calL_{W^c }$ of the subset $\Gamma_{\tilde W }$ of $\Gamma$, 
see Definition \ref{def:limitset} and Proposition \ref{prop:Lambda-Infty}. 
We show in Proposition \ref{prop:MassPatSul} that $\nu_o^g( \calL_{W^c })=0$. 
By definition, $\partial\tilde M \setminus \calL_{W^c}$ 
consists in asymptotic directions of geodesics returning 
infinitely often in the compact set $W$. 
In particular, it is included in the radial limit set. 
We deduce that $\nu_o^g(\Lambda_\Gamma^r)=1$, 
which implies that $\Gamma$ is divergent by Hopf-Tsuji-Sullivan Theorem.

\begin{defi}\label{def:limitset}
Let $W\subset M$ be a compact subset and $\tilde W\subset \tilde M$ a nice preimage of $W$. 
Introduce the set 
$$
\Lambda_{\tilde W} = \left\{\xi\in \Lambda_\Gamma\; \mbox{s.t.}\;  
\exists x\in \tilde W, [x,\xi)\cap \Gamma \cdot \tilde W \subset \tilde W\right\}\,.
$$
We  call \emph{limit set of $\Gamma$ out of $W$} the set
$\ds{\cal L}_{W^c} = \Gamma\cdot \Lambda_{\tilde W}$.
\end{defi}

The following proposition shows that all elements of $\Lambda_{\tilde W}$ 
are limit points of $\Gamma_{\tilde W}\cdot o$ in the boundary at infinity, 
and that the only limit points of $\Gamma_{\tilde W}\cdot o$ 
which are not in $\Lambda_{\tilde W}$ are  endpoints  of geodesic  
rays  which do not come back inside the interior $\Gamma\cdot \inter{\tilde W}$, 
after leaving $\tilde W$ but touch the boundary $\bd (\Gamma\cdot \tilde W)$.

\begin{prop}\label{prop:Lambda-Infty}
Let $W\subset M$ be a compact subset and $\tilde W\subset \tilde M$ 
a nice preimage of $W$. 
Then
$$ 
\Lambda_{\tilde W} \,\,\subset\,\,
 \overline{\Gamma_{\tilde W} \cdot o}\backslash \Gamma_{\tilde W} \cdot o \,\, \subset\,\, 
 \left\{\xi\in \Lambda_\Gamma\; \mbox{s.t.}\;  \exists x\in \tilde W, [x,\xi)\cap \Gamma \cdot \inter{\tilde W} \subset \tilde W\right\}\,.
$$
\end{prop}

\begin{proof} Without loss of generality, assume $o\in \tilde W$. 
We show first the left inclusion.
Let $\xi\in \Lambda_{\tilde W}\subset \Lambda_\Gamma$. 
There exists a sequence $(\gamma_n) $ of elements of $\Gamma$ such that $\gamma_n o\to \xi$. 
Moreover, by definition of $\Lambda_{\tilde W}$, there exists $x\in \tilde W$ 
such that the  geodesic $[x,\xi)$ does not intersect $\Gamma\cdot \tilde W$ after leaving $\tilde W$.
Thus, for $n$ large enough, the geodesic segment $[x,\gamma_n o]$ also leaves $\tilde W$ before returning to 
$\gamma_n \tilde W$. 
Let $\tilde \gamma_n\cdot \tilde W$ be the first image of $\tilde W$ crossed 
by the geodesic segment $[x, \gamma_n o]$ after leaving $\tilde W$. 
By construction, $\tilde \gamma_n\in \Gamma_{\tilde W}$.
Moreover, we have 
\begin{equation}
  \lim_{n\to +\infty}d^g(x, \tilde \gamma_n o) = +\infty\,.
\end{equation} 
Indeed, for all  $R>0$, there exists $\eta>0$ such that 
inside the (compact) ball $B^g(x, R)$, 
the distance between $[x, \xi)\cap B^g(x,R)$ and 
$\Gamma\cdot \tilde W \backslash \tilde W$ is at least $\eta$. 
Moreover, it follows from Lemma \ref{lem:Control-Hyperb} 
that the sequence of geodesic segments $([x, \gamma_n  o])_{n\in \bbN}$ 
converges to the half geodesic $[x, \xi]$ uniformly on $B^g(x,R)$. 
Thus, for all $n $ large enough, $[x, \gamma_n o]\cap B^g(x, R)$ and
$[x, \xi]\cap B^g(x,R)$ are $\frac \eta 2$-close, 
so that $[x, \gamma_n\cdot o]$ does not meet 
$\Gamma\cdot \tilde W \backslash \tilde W$ on $B^g(x,R)$, whence $d^g(x, \tilde \gamma_n o)\geq R$. 

It follows from the above  that the sequence of geodesic segments 
$([x, \tilde \gamma_n o])_{n\in \bbN}$ also converges to the half geodesic $[x, \xi]$, so that 
$$
\xi =
 \lim_{n\to +\infty}\tilde \gamma_n  o \in
 \overline{\Gamma_{\tilde W} \cdot o}\backslash \Gamma_{\tilde W} \cdot o \,.
$$

\medskip

Let us now show that 
$
\ds \overline{\Gamma_{\tilde W} \cdot o}\backslash \Gamma_{\tilde W} \cdot o \subset 
 \left\{\xi\in \Lambda_\Gamma\; \mbox{s.t.}\;  \exists x\in \tilde W, [x,\xi)\cap \Gamma \cdot \inter{\tilde W} \subset \tilde W\right\}\,.
$
Let $\ds \xi \in \overline{\Gamma_{\tilde W} \cdot o}\backslash \Gamma_{\tilde W}$. 
There exists a sequence $(\gamma_n)_{n\in \bbN}$ of elements of $\tilde \Gamma_{\tilde W}$ 
such that $ \gamma_n   o\to \xi$ and $d^g(o, \gamma_n\cdot o)\to +\infty$. 
By definition of $\Gamma_{\tilde W}$, for all $n\geq 0$ there exist 
$x_n, y_n\in \tilde W$ such that 
the geodesic segment $[x_n, \gamma_n   y_n]$ intersects $\Gamma\cdot \tilde W$ only in $\tilde W$
 and $\gamma_n \cdot \tilde W$. 
Up to taking a subsequence, 
we can assume that $x_n\to x_\infty\in \tilde W$ and $y_n\to y_\infty\in \tilde W$ as $n\to +\infty$. 
Once again, it follows from the compactness of $\tilde W$ 
and Lemma \ref{lem:Control-Hyperb} that the sequence of 
geodesic segments $([x_\infty, \gamma_n y_\infty])_{n\in \bbN}$ 
converges to $[x_\infty, \xi ]$ uniformly on compact sets. 
Therefore $[x_\infty, \xi )$ cannot intersect the interior of $\Gamma\cdot \tilde W$.
\end{proof}

%The limit set out of $W$ encodes the set of endpoints of geodesic  rays
% which eventually leave $W$ and never come back. 
%The set $\Lambda_{\tilde W}$ depends on the choice of $\tilde W$, 
%but if $\tilde W_2 = \alpha \tilde W_1$ are nice isometric preimages of $W$, 
%then $\Lambda_{\tilde W_2} = \alpha \cdot \Lambda_{\tilde W_1}$, so that
%$\calL_{W^c}=\Gamma.\Lambda_{\tilde W}$ does not depend on the choice of $\tilde W%$. 
%  For any \emph{subset} $\Gamma_0\subset \Gamma$ 
%and   $x\in \tilde M$, set
%$\ds \Gamma_0 \cdot x = \{ \gamma x \; ; \; \gamma\in \Gamma_0\}$.

We gather in the following proposition elementary properties of the sets $\Lambda_{\tilde W}$ and
$\calL_{W^c}$.  
\begin{prop}\label{prop:EnsLimRad} Let $(M,g)$ be a manifold with pinched negative curvature. 
Let $W\subset M$ be a nice compact set, and $\tilde W$ a nice preimage. 
With the above notations, 
\begin{itemize}
\item 
the set $\calL_{W^c} = \Gamma\cdot \Lambda_{\tilde W}$ is 
the set of endpoints of geodesics which eventually leave  $\Gamma \cdot \tilde W$:
$$
\calL_{W^c}=\Gamma \cdot \Lambda_{\tilde W} =
 \left\{v_+\in \Lambda_\Gamma \; ; \;  
\exists v\in S^g\tilde M, \exists T>0 \mbox{ s.t. } \forall t\geq T, \pi g^tv \notin \Gamma \cdot \tilde W\right\}\,.
$$
\item  
The \emph{limit set out of $W$}, $\ds\calL_{W^c}  = \Gamma \cdot \Lambda_{\tilde W}$ 
does not depend on the choice of nice preimage $\tilde W$. 
\item 
If $W_1\subset W_2$, we have $\calL_{W_2^c}\subset \calL_{W_1^c}$.
\item  $\ds \Lambda_\Gamma \backslash \left(\calL_{W^c}\right) \subset \Lambda_\Gamma^r$, 
where $\Lambda_\Gamma^r$ is the radial limit set.
\end{itemize} 
\end{prop}

\begin{proof} The first property is left to the reader.  

The set $\ds \left\{v_+\in \Lambda_\Gamma \; ; \;  
\exists v\in S^g\tilde M, \exists T>0 \mbox{ s.t. } \forall t\geq T, 
\pi g^tv \notin \Gamma \cdot \tilde W\right\}$ only depends on 
$\Gamma\cdot \tilde W = p_\Gamma^{-1}(W)$, which is independent of the choice of $\tilde W$. 

If $W_1\subset W_2$, then for all nice preimages $\tilde W_1$ and $\tilde W_2$, we have 
$$
\Gamma\cdot \tilde W_1 = p_\Gamma^{-1}(W_1)\subset p_\Gamma^{-1}(W_2) = \Gamma\cdot \tilde W_2\,,
$$
 which shows the third point.
 
The radial limit set is the set of $\xi\in \Lambda_\Gamma$ such that 
there exists $x\in \tilde M$ and a compact set $K\subset M$ such that 
the geodesic ray $[x, \xi)$ intersects infinitely  often the preimage 
$p_\Gamma^{-1}(K)$. If $\xi \in  \Lambda_\Gamma \backslash \left(\Lambda_{\tilde W}\right)$, 
by the above proposition, the geodesic  ray $[x,\xi]$ intersects infinitely 
 often the set $\Gamma\cdot \tilde W = p_\Gamma^{-1}(W)$, which shows the last claim.
\end{proof}

As seen in Section \ref{ssec:SPR-Ex}, basic examples are given by geometrically finite manifolds. 
The following proposition is an immediate consequence of 
Propositions \ref{prop:Gamma-Infty-GeoFin} and \ref{prop:Lambda-Infty}.

\begin{prop}
Let $(M,g)$ be a geometrically finite manifold with pinched negative curvature, 
with $k$ cusps  $C_1,..., C_k$.  Let $W=B^g(x,R)$ be a large ball. 
It admits   a nice preimage $\tilde W$ such that
$$
\Lambda_{\tilde W} = \{\xi_1,..., \xi_i\}\,,
$$ each point $\xi_1,..., \xi_k\in \Lambda_\Gamma$ being a parabolic point  fixed by a parabolic group  $\calP_i<\Gamma$ representing the cusp $C_i$. 
\end{prop}

The following proposition is a detailed version of  
Theorem \ref{theo:SPR-Div}, with additional properties which will be useful in section \ref{ssec:entropy-variation}.

\begin{prop}\label{prop:MassPatSul}\label{lem:SPR-PatSul}\label{prop:SPR-PatSul}
Let $(M = \tilde M/\Gamma,g)$ be a complete manifold with $K_g\leq -a^2$, 
with $\Gamma = \pi_1(M)$ its fundamental group. 
Assume that $(M,g)$ is SPR. 
Then $\Gamma$ is divergent. 

Moreover, for all compact sets  $W\subset M$  
such that $\delta_{W^c} < \delta_\Gamma$, for all $\eta\in (0, \delta_\Gamma - \delta_{W^c})$, 
there exists $C = C(g, W, \eta, a)>0$ such that for all nice preimage $\tilde W$ of $W$ and all $T\geq 4  {\rm diam}_g(W)$, if 
$$
U_T = U_T(\tilde W, g) = \left\{ \xi\in \tilde M \cup \bd \tilde M\; ; \; \exists x\in \tilde W \mbox{ s.t. } \forall t\in [0, T], [x,\xi]_T \cap \Gamma W \subset \tilde W.\right\}\,,
$$
%Then  $\Lambda_{\tilde W}=\cap_{T>0} U_T$, the group $\Gamma$ is divergent, and
then the
unique Patterson-Sullivan density $(\nu_x^g)_{x\in\tilde M}$ on $\Lambda_\Gamma$ 
such that $\nu_o^g(\Lambda_\Gamma^r)=1$ satisfies 
$$
\nu_o^g(U_T) \leq Ce^{-(\delta_\Gamma - \delta_{W^c} - \eta)T}\,.
$$
In particular,
$$
\nu_o^g(\Lambda_{\tilde W})=
\nu_o^g(\cap_{T>0} U_T)=0\,.
$$
%Moreover, the constant $C(\eta, g, W, a)$ varies continuously in $\eta$, $a$, $g$ (for the uniform $\calC^1$ topology), and $W$ (for the Hausdorff topology).
%[Samuel:]Je ne suis pas sûr de la régularité de la variation, ce sera plus simple d'utiliser l'expression directe de la constante quand on en aura besoin.
\end{prop}

\begin{proof} We  start with any Patterson-Sullivan density $(\nu_x^g)$ on $\Lambda_\Gamma$ obtained
as a weak limit of an average as in section \ref{BM}. We will show that   there exists $C>0$ such that 
for all $T>0$ large enough, 
\begin{equation}\label{eqn:HTS}
\nu_o^g(U_T) \leq Ce^{-(\delta_\Gamma - \delta_{W^c} - \eta)T}\,.
\end{equation}
By definition, $U_T$    is the set of points joined by a geodesic from $\tilde W$ wich, 
after exiting $\tilde W$, does not enter $\Gamma \cdot \tilde W$ before time $T$, so that 
$$
\Lambda_{\tilde W} = \bigcap_{T>0} U_T\,.
$$ 
Therefore, (\ref{eqn:HTS}) implies 
$$
\nu^g_o\left(\Lambda_{\tilde W}\right)=  0 \quad 
\mbox{so that } \quad \nu_o^g(\calL_{W^c})=0 \quad \mbox{and}\quad \nu_o^g(\Lambda_\Gamma^r)=1\,.
$$ 
By Hopf-Tsuji-Sullivan Theorem, it will imply that $\Gamma$ is divergent 
and the Patterson-Sullivan density is unique. 

If $x\in \tilde M$ and $\xi \in \tilde M \cup \bd \tilde M$, 
set $[x,\xi]_T = (\pi g^t v)_{t\in [0, T]}$, 
where $v\in S^g _x \tilde M$ is the tangent vector at $x$ of the geodesic $[x,\xi]$.

Recall notations from Section \ref{BM}. 
We omit the mention of the metric $g$ here. 
As in  \cite{Patterson},   choose a positive increasing map 
$h : \bbR^+ \to \bbR^+$  such that for all $\eta>0$,
 there exist  $C_\eta>0$ and $r_\eta>0$ such that
\begin{equation}\label{eq:CroissFaible1}
\forall r\geq r_\eta, \quad \forall t\geq 0, \quad h(t+r)\leq C_\eta e^{\eta t}h(r)\,,
\end{equation}
and the series  
$\ds  \tilde P_\Gamma(s) = \sum_{\gamma\in \Gamma} h(d(o,\gamma o)e^{-s d(o, \gamma   o)}$ diverges
at the critical exponent $\delta_\Gamma$. 
Construct a Patterson-Sullivan density  $(\nu_x)$ s.t. for all $x\in \tilde M$, the measure $\nu_x$ is a weak limit as $s\to \delta_\Gamma^+$ of the positive finite measures
$$
\nu_x^{s} = \frac{1}{\tilde P_\Gamma(o,s)} \sum_{\gamma\in \Gamma}h(d(x,\gamma o))e^{-s d(x, \gamma   o)}\delta_{\gamma o}\,.
$$
%Let us note that (\ref{eq:CroissFaible1}) implies that for all $\eta>0$, there exists $C_\eta>0$ such that
%\begin{equation}\label{eq:CroissFaible2}
%\forall r\geq 0, \quad \forall t\geq 0, \quad h(t+r)\leq C_\eta e^{\eta t}h(r).
%\end{equation}

For all $\gamma \in \Gamma_{\tilde W}$, define $\calO_{\tilde W}(\gamma \cdot \tilde W)$ as
the set of  $y \in\tilde M\cup \bd \tilde M$ such that 
there exists $v\in S^g \tilde W$ such that 
the first intersection of the geodesic ray $(\pi g^tv)_{t\geq 0}$ with $\Gamma \cdot \tilde W$, after the first exit of $\tilde W$ is in $\gamma \cdot \tilde W$, and the point $y$ belongs to $(\pi g^t v)_{t\ge 0}$.

 By definition of $U_T$ and $\Gamma_{\tilde W}$, and 
triangular inequality, for all $T>0$ and   $\alpha \in \Gamma$, 
if $\alpha o\in U_T$, there exists $\gamma\in \Gamma_{\tilde W}$ 
such that $\alpha o\in \calO_{\tilde W}( \gamma \cdot \tilde W)$ 
and $d(o, \gamma   o)\geq T - 2D$, 
with $D = {\rm diam}(\tilde W)$. 
Indeed, choose $\gamma$  so that $\gamma\tilde W$ 
is the first copy of $\tilde W$ intersected by all geodesic segments 
from $\tilde W$ to $\alpha o$ after exiting $\tilde W$ inside $\Gamma\tilde W$. 

In other words, we have 
\begin{equation}\label{eq:LambdaOmbres}
\Gamma \cdot o \cap U_T \subset
 \bigcup_{\gamma \in \Gamma_{\tilde W}, d(o, \gamma \cdot o)\geq T - 2D} \calO_{\tilde W}(\gamma \cdot \tilde W).
\end{equation}
Fix $s>\delta_\Gamma$ and recall from section \ref{BM} that for all $x,y\in \tilde M$ 
and   $\xi \in \Gamma \cdot o$,
$$
\frac{d\nu_y^s}{d\nu_x^s}(\xi) =
 e^{-s(d(y,\xi) - d(x,\xi))}\frac{h(d(y,\xi))}{h(d(x,\xi))}\,.
$$
Therefore, for all $\gamma\in \Gamma_{\tilde W}$,
\begin{eqnarray*}
\nu_o^s\left(\calO_{\tilde W}(\gamma \cdot \tilde W)\right) 
& = & \nu_{\gamma^{-1} o}^s\left(\calO_{\gamma^{-1}\tilde W}(\tilde W)\right)\\
& = & \int_{\calO_{\gamma^{-1}\tilde W}(\tilde W)}e^{-s(d(\gamma^{-1} o,\xi) - d(o,\xi))}\frac{h(d(\gamma^{-1} o,\xi))}{h(d(o,\xi))}\,d\nu_o^s(\xi)\,.
\end{eqnarray*}
Moreover, there exists $C>0$ such that as soon as $d(o, \gamma  o)> 2D$, 
for all $\xi \in \calO_{\gamma^{-1}\tilde W}(\tilde W)$,
$\ds
d(\gamma^{-1} o, \xi) \geq 
d(\gamma^{-1} o, o) + d(o, \xi) - C$, 
 which implies by (\ref{eq:CroissFaible1}) that 
$$
e^{-s(d(\gamma^{-1} o,\xi) - d(o,\xi))}\frac{h(d(\gamma^{-1} o,\xi))}{h(d(o,\xi))} \leq 
e^{sC} C_\eta e^{- s(d(\gamma^{-1} o, o)}e^{\eta d(\gamma^{-1} o, o)}\,.
$$
Therefore, as $\nu_o(\bd \tilde M) = 1$, 
there exists $C_\eta>0$ such that for all $\gamma \in \Gamma_{\tilde W}$ 
with $d(o, \gamma   o)> 2D$ and   $s>\delta_\Gamma$,
$$
\nu_o^s\left(\calO_{\tilde W}(\gamma \cdot \tilde W)\right) \leq C_\eta e^{(-s + \eta)d(o, \gamma \cdot o)}\,.
$$
By (\ref{eq:LambdaOmbres}), for all $T>4D$, we get
$$
\nu_o^s(U_T) \leq
 C_\eta \sum_{\stackrel{\gamma \in \Gamma_{\tilde W}}{ d(o, \gamma \cdot o)\geq T - 2D}} e^{(-s+\eta)d(o, \gamma \cdot o)}\,.
$$
Taking any weak limit as $s\to \delta_{\Gamma}^+$, we obtain
$$
\nu_o^g(U_T) \leq 
C_\eta \sum_{\stackrel{\gamma \in \Gamma_{\tilde W}}{ d(o, \gamma \cdot o)\geq T - 2D}} e^{(-\delta_\Gamma+\eta)d(o, \gamma \cdot o)}\,.
$$
As $\delta_\Gamma-\eta>\delta_{W^c}$, the right hand side decreases exponentially fast as $T\to +\infty$. 
As mentioned as the beginning of the proof, by Hopf-Tsuji-Sullivan, we deduce that
$\Gamma$ is divergent so that Theorem \ref{theo:SPR-Div} is proven. 

Let us prove now the end of the statement of Proposition \ref{prop:MassPatSul}. 
The Patterson-Sullivan $\nu_o^g$ is \emph{the} weak limit as $s\to \delta_\Gamma^+$ 
of $\ds \nu_o^{s} = \frac{1}{P_\Gamma(o,s)} \sum_{\gamma\in \Gamma}e^{-s d(o, \gamma o)}$.
 Repeting exactly the same computations, setting $h \equiv 1$,
 we get that there exists $C_a>0$, depending only on the curvature upperbound, such that for all $T\geq 4D$
$$
\nu_o^g(U_T) \leq e^{\delta_\Gamma C_a}\sum_{\stackrel{\gamma \in \Gamma_{\tilde W}}{ d(o, \gamma   o)\geq T - 2D}} e^{(-\delta_\Gamma)d(o, \gamma   o)}.
$$
%Recall that $\delta_{W^c}$ is the critical exponent of the Poincar\'e series of $\Gamma_{\tilde W}$. 
We get therefore that for all $T\geq 4D$,
$$
\nu_o^g(U_T) \leq e^{\delta_\Gamma (C_a+2D)}e^{-(\delta_\Gamma - \delta_{W^c} -\eta)T}\sum_{\gamma \in \Gamma_{\tilde W}} e^{-(\delta_{W^c} + \eta)d(o, \gamma   o)},
$$
which is precisely the desired estimate with
\begin{equation}\label{eq:CsteMassPS}
C(\eta, g, W, a) = e^{\delta_\Gamma (C_a+2D)}\sum_{\gamma \in \Gamma_{\tilde W}} e^{-(\delta_{W^c} + \eta)d(o, \gamma  o)}.
\end{equation} 
 \end{proof}

Under the above assumption,   the   Patterson-Sullivan measure $\nu_o^g$ gives full mass to the set of endpoints of lifts of geodesics of $(M,g)$ which come  back infinitely often in $W$. This set is in general strictly smaller than the radial limit set. The product structure of the Bowen-Margulis measure (see section \ref{BM})  implies the following useful fact.

\begin{coro}\label{coro:support-BM} Under the same assumptions,   let $W\subset M$ be any compact set such that $\delta_{W^c}(g)<\delta_\Gamma(g)$. Then the Bowen-Margulis measure of $S^gM$ is finite and gives full mass to the set of bi-infinite geodesics which intersect infinitely often $W$ in the past and in the future.
\end{coro}

%%%%%%%%%%%%%%%%%%%%%%%%%%%%%%%%%%%%%%%%%%%%%%%%%%%%%%%%%%%

\subsection{Entropy variation for SPR manifolds}\label{ssec:entropy-variation}

As mentionned earlier, the original motivation of this article
 was to find reasonnable geometric assumptions 
on non-compact manifolds with negative curvature 
such that the entropy is regular under a  small  variation of the metric.
In this subsection, our aim is to finish the proof of Theorem
\ref{theo:ConvMassSPR}.

Let $(g_\epsilon)_{\epsilon\in (-1, 1)}$ be a uniformly $\mathcal C^1$ family of complete metrics on the manifold $M$ such that for all $\epsilon\in (-1, 1)$, $-b^2 \leq K_{g_\epsilon} \leq  -a^2$ for some $b>a>0$, and $(M,g_0)$ is SPR.

 Let $W\subset M$ be a compact subset such that 
$\delta_{W^c}(g_0) < \delta_\Gamma(g_0)$, 
and let $\tilde W$ be a nice preimage of $W$. 
For   $r>0$, denote by $\ds W_r = \{x\in M ; d^{g_0}(x,W)\leq r\}$ 
the $(g_0, r)$-neighbourhood of $W$.
Note that $\ds \delta_{W_r^c}(g_0)\leq \delta_{W^c}(g_0) \le \delta_\infty(g_0)< \delta_\Gamma(g_0)$. 
Denote by $\tilde W_r$  a nice preimage of $W_r$ 
such that $\tilde W\subset \tilde W_r$. 
Observe that $\gamma \cdot \tilde W_r$ is the $(g_0, r)$-neighbourhood of $\gamma \cdot \tilde W$.

\begin{lemm}\label{lem:ContGammaW}
For all $r>0$, there exists a finite set $F\subset \Gamma$ and $\epsilon_0>0$ such that for all $\epsilon\in(-\epsilon_0, \epsilon_0)$, we have 
$$
\Gamma_{\tilde W_{2r}}(g_0)\subset \bigcup_{\alpha,\beta\in F}\alpha \Gamma_{\tilde W_r}(g_\epsilon)\beta\quad \mbox{and}\quad \Gamma_{\tilde W_r}(g_\epsilon)\subset  \bigcup_{\alpha,\beta\in F}\alpha\Gamma_{\tilde W}(g_0)\beta\,.
$$
\end{lemm}

\begin{proof}  We   prove the right inclusion, the left one is   proved similarly.

Let $D = {\rm diam}_{g_0}(\tilde W)$ and $D' = e^1 (D+1)$, 
so that for all $\epsilon\in(-1, 1)$, 
${\rm diam}_{g_\epsilon}(\tilde W_r)\leq D'$. 
It follows from Section \ref{sec:ConvGeod} that there exists $\epsilon_0>0$ 
such that for all $\epsilon\in (\epsilon_0, \epsilon_0)$ 
  $x,y\in \tilde W_r$, and $\gamma\in \Gamma_{\tilde W_r}$, 
the $g_\epsilon$-geodesic between $x$ and $\gamma y$ 
is at distance less than $r$ to the $g_0$-geodesic between $x$ and $y$. 
Reasoning as in  Proposition \ref{prop:GammaWPreimage}    leads to the desired result.
\end{proof}

This lemma leads to the following corollary, which implies the first item of 
Theorem \ref{theo:SPR-FinBM}. 
\begin{coro}\label{coro:continuite-delta-W} Let $(g_\varepsilon)_{\varepsilon\in (-1, 1)}$ be a uniformly $\mathcal C^1$ family of complete metrics on the manifold $M$ such that for all $\varepsilon\in (-1, 1)$, $-b^2 \leq K_{g_\varepsilon} \leq  -a^2$ for some $b>a>0$, and $(M,g_0)$ is SPR. Then for all $\alpha>0$ and  $r>0$, there exists $\varepsilon_0>0$ such that for all $\varepsilon\in(-\varepsilon_0, \varepsilon_0)$, we have
$$
e^{-\alpha}\delta_{W_{2r}}(g_0) \leq \delta_{W_r}(g_\varepsilon) \leq e^\alpha \delta_{W^c}(g_0).
$$
In particular, the entropy at infinity $\varepsilon \mapsto \delta_\infty(g_\varepsilon)$ is continuous at $\varepsilon = 0$, and if $\alpha>0$ is small enough, $g_\varepsilon$ is SPR for $\varepsilon\in(-\varepsilon_0, \varepsilon_0)$.
\end{coro}

\begin{proof}
Let $r,\alpha>0$ be fixed, and choose $\varepsilon_0$ as in Lemma \ref{lem:ContGammaW}. 
For all $\gamma\in  \Gamma$, 
we have 
$\ds d^{g_\varepsilon}(o, \gamma   o) \geq e^{-\varepsilon/2} d^{g_0}(o, \gamma   o)$. 
Therefore,  for all $\varepsilon\in (-\varepsilon_0, \varepsilon_0)$ 
we get $\delta_{W_r}(g_\varepsilon) \leq e^{\varepsilon/2}\delta_{W^c}(g_0) \leq e^\alpha \delta_{W^c}(g_0)$ up to reducing $\varepsilon_0$. 
The other inequality is proved similarly.
\end{proof}

Let us  show now the last item of Theorem \ref{theo:SPR-FinBM}, that is 
that the mass of the Bowen-Margulis measure 
of $g_\varepsilon$ varies continuously. 
This will rely on the following estimate, which is a uniform version of Proposition \ref{prop:SPR-PatSul}.

\begin{lemm}\label{lem:UnifMassPatSul}
For all $\delta_0\in (0, \delta_\Gamma(g_0) - \delta_\infty(g_0))$ 
and $\beta\in (0, \delta_0)$, 
there exists a compact set $W\subset M$ with nice preimage $\tilde W$, 
$\varepsilon_0>0$ and $C>0$ such that for all $\varepsilon\in (-\varepsilon_0, \varepsilon_0)$, we have $\delta_\Gamma(g_\varepsilon) - \delta_{W^c}(g_\varepsilon) \geq \delta_0$ and
$$
\nu_o^{g_\varepsilon}(U_T(\tilde W, g_\varepsilon)) \leq C e^{-\beta T}\,,
$$
where $U_T(\tilde W, g_\varepsilon)$ is defined as in Proposition \ref{prop:SPR-PatSul}.
\end{lemm}

\begin{proof}
 Let $\delta_0\in (0, \delta_\Gamma(g_0) - \delta_\infty(g_0))$ be fixed. 
By the above corollary, for $|\varepsilon|$ small enough, 
$(M, g_\varepsilon)$ is SPR and has therefore a finite Bowen-Margulis. 
Choose  $\alpha>0$ small enough and a large enough compact set $W\subset M$ so that 
$\delta_\infty(g_0)\leq \delta_{W^c}(g_0) \leq e^{\alpha} \delta_\infty(g_0)$. 
Let $r>0$ small enough and $\varepsilon_0>0$ 
given by  Corollary \ref{coro:continuite-delta-W}  be such that
 for all $\varepsilon\in (-\varepsilon_0, \varepsilon_0)$, 
$$
e^{-\alpha}\delta_{W^c}(g_0)\leq \delta_{W_r}(g_\varepsilon) 
\leq e^{\alpha}\delta_{W^c}(g_0)\,.
$$
Up to decreasing $\alpha>0$, we can  therefore assume
 that for all $\varepsilon\in (-\varepsilon_0, \varepsilon_0)$,  
$$
\delta_{\Gamma}(g_\varepsilon) - \delta_{W_r}(g_\varepsilon) \geq \delta_0>0 \,.
$$
Let $\beta\in (0, \delta_0)$  and $\tilde W_r$ nice preimage of $W_r$ be fixed. Define  $D>0$ as 
$\ds D=\sup_{\varepsilon\in (-\varepsilon_0, \varepsilon_0)}\mbox{diam}( \tilde W_r)$. 
For all $\varepsilon\in (-\varepsilon_0, \varepsilon_0)$, 
let $U_T^\varepsilon = U_T(\tilde W_r, g_\varepsilon)$ be defined as in 
Proposition \ref{prop:SPR-PatSul}.
By the last estimate  in the proof of Proposition \ref{prop:MassPatSul}, 
there exists $C_a>0$, only depending on the curvature upperbound of 
the metrics $g_\varepsilon$, such that for all $T>4D$,
$$
\nu_0^{g_\varepsilon}(U_T^\varepsilon) \leq 
e^{\delta_\Gamma(g_\varepsilon) (C_a+2D)}e^{-\beta T}
\sum_{\gamma \in \Gamma_{\tilde W_r(g_\varepsilon)}} e^{-(\delta_{W_r}(g_\varepsilon) + \beta)d^{g_\varepsilon}(o, \gamma   o)}.
$$
Therefore, 
$$
\nu_0^{g_\varepsilon}(U_T^\varepsilon) \leq e^K e^{-\beta T} \sum_{\gamma \in \Gamma_{\tilde W(g_0)}} e^{-(e^{-\e}\delta_{W}(g_0) + \alpha)e^{-\varepsilon}d^{g_0}(o, \gamma \cdot o)},
$$
where $K\in \mathbb R$ is independent of $\varepsilon$.
 Up to reducing $\alpha>0$ and $\varepsilon_0>0$, 
we can suppose that $\ds e^{-\alpha}\delta_{W}(g_0) + \beta)e^{-\varepsilon} > \delta_{W}(g_0) + \frac \beta 2$. 
Therefore, we get that for all $\varepsilon\in (-\varepsilon_0, \varepsilon_0)$,
\begin{equation}\label{eq:MassPSUnif}
\nu_0^{g_\varepsilon}(U_T^\varepsilon) \leq Ce^{-\beta T},
\end{equation}
with $C>0$ being independent of $\varepsilon$. 
This concludes the proof of Lemma \ref{lem:UnifMassPatSul}, the compact set $W$ of the statement being the set  $W_r$ of the proof.
\end{proof}

Let us now conclude the proof of Theorem \ref{theo:ConvMassSPR}. 
Let $W\subset M$, $\tilde W \subset \tilde M$ and $\beta, \varepsilon_0, C>0$ satisfy the conclusion of Lemma \ref{lem:UnifMassPatSul}. 
For all $R>0$, set as usual $\ds W_R = \{x\in M ; d_{g_0}(x,W)\leq R\}$. 
We have shown in Theorem \ref{theo:ContMBM} that, 
under our current hypotheses, the Bowen-Margulis measure 
$\varepsilon \mapsto m_{BM}^{g_\varepsilon}$ varies continuously
for the weak-* convergence, i.e. on the dual of compactly supported maps. 
In particular, for all \emph{fixed compact set} $K\subset M$ 
with $m_{BM}^{g_0}(\partial S^{g_0}K)=0$, 
the map $\varepsilon \mapsto m_{BM}^{g_\varepsilon}(S^{g_\varepsilon}K)$ is continuous at $\e=0$. 
Therefore the following lemma will imply Theorem \ref{theo:ConvMassSPR}.

%For all $\e>0$, let $R>0$ satisfy the conclusion of the previous lemma. Let $\ds K_\e = \bigcup_{\varepsilon\in [-\varepsilon_0, \varepsilon_0]}S^{g_\varepsilon}W_R$. This is a compact set such that for all $\varepsilon\in (-\varepsilon_0, \varepsilon_0)$, $\ds m_{BM}^{g_\varepsilon}(S^{g_\varepsilon}(M\backslash K_\e))\leq \e$. Since $\varepsilon\mapsto m_{BM}^{g_\varepsilon}(K_\e)$ is continuous, this concludes the proof of Theorem \ref{theo:ConvMassSPR}.

\begin{lemm} With the above notations, for all $\alpha>0$, 
there exists $R_0>0$ such that for all $R\geq R_0$ and all $\varepsilon\in(-\varepsilon_0, \varepsilon_0)$, we have
$$
m_{BM}^{g_\varepsilon}\left(S^{g_\varepsilon}(M\backslash W_R)\right) \leq \alpha\,.
$$
\end{lemm}

\begin{proof}
Let $R>8 {\rm diam} (W)$ be fixed and let $O_R = M\backslash W_R$. 
By Corollary \ref{coro:support-BM},
 for all $\varepsilon\in (-\varepsilon_0, \varepsilon_0)$, 
since $\delta_{W^c}(g_\varepsilon) < \delta_\Gamma(g_\varepsilon)$,
 the Bowen-Margulis measure $m_{BM}^{g_\varepsilon}$ 
gives full mass to the set of vectors 
which hit infinitely often $W$ in the past and in the future. 
In particular,
$$
m_{BM}^{g_\varepsilon}(S^{g_\varepsilon}O_R) = 
m_{BM}^{g_\varepsilon}\left(\coprod_{n\geq R-1} O_n^\varepsilon\right),
$$
where $O_n^\e$ is defined for all integers $n\geq R-1$ by 
$$
O_n^\varepsilon = \left\{ v\in S^{g_\varepsilon}O_R\; ; \; \exists t\in [n, n+1[ \mbox{ s.t. } \forall s\in [0, t), \pi g^{-s}v\notin W \mbox{ and } \pi g_\varepsilon^{-t}v\in W\right\}\,.
$$
Therefore, since the Bowen-Margulis measure $m_{BM}^{g_\varepsilon}$ 
is invariant under the geodesic flow $(g_\varepsilon^t)$,
$$
m_{BM}^{g_\varepsilon}(S^{g_\varepsilon}O_R) =
 \sum_{n\geq R-1} m_{BM}^{g_\varepsilon}\left(O_n^\varepsilon\right)
 = \sum_{n\geq R-1} m_{BM}^{g_\varepsilon}\left(g^{-n}_\varepsilon(O_n^\varepsilon)\right)
$$
Now, by definition for all $v\in g^{-n}_\varepsilon(O_n^\varepsilon)$, 
there exists $t\in [0, 1)$ such that $w = g^{-n-t}_\varepsilon v \in S^{g_\varepsilon} W$ and for all $s\in [0, n]$, 
we have $\pi g^s w \notin W$. 

Let us write 
$$
\tilde A_n^\varepsilon = \left\{v\in S^{g_\varepsilon} \tilde W ; \exists t\in [n, n+1) \mbox{ s.t. } \forall s\in (0, t), \pi g_\varepsilon^s v\notin \gamma \cdot \tilde W \mbox{ and } \pi g^tv\in \gamma \cdot \tilde W\right\}.
$$
The reader will easily check that 
$\ds \bigcup_{s \in [0, 1)} g_\varepsilon^s A_n^{\varepsilon}\subset S^{g_\varepsilon}\tilde M$ projects onto
 $g^{-n}(O_n^\varepsilon)$. 
Moreover, as soon as $\varepsilon_0$ is small enough, 
since $g_\varepsilon\geq e^{-\varepsilon}g_0 \geq \frac 14g_0$, 
all vectors $v\in \tilde A_n^\varepsilon$ have a point at infinity $v^+$ which satisfies $v_+\in U_{n/2}(\tilde W, g_\varepsilon)$.  
As the map 
$$
v\mapsto e^{\delta_\Gamma(g_\varepsilon)(\calB_{v_+}(o, \pi v)+\calB_{v_-}(o, \pi v)}
$$
is uniformly bounded in $v\in   W$ and $\e\in (-\e_0,\e_0)$,  
the product structure of the Bowen-Margulis measure (see section \ref{BM}) 
implies
$$
m_{BM}^{g_\varepsilon}(O_n^\varepsilon) = 
m_{BM}^{g_\varepsilon}(g_\varepsilon^{-n}(O_n^\varepsilon)) \leq 
2K\nu_o^{g_\varepsilon}(\Lambda_\Gamma)\times \nu_o^{g_\varepsilon}(U_{n/2}^\varepsilon),
$$
which eventually gives by Lemma \ref{lem:UnifMassPatSul} 
$$
m_{BM}^{g_\varepsilon}(O_n^\varepsilon) \leq 2KCe^{-\frac \alpha 2 n},
$$
where $C$ and $\alpha$ do not depend on $\varepsilon$. 
Therefore, we  get
$$
m_{BM}^{g_\varepsilon}(S^{g_\varepsilon}O_R) \leq 2KC\sum_{n\geq R-1}e^{-\frac \alpha 2 n} \leq \e
$$
as soon as $R$ is large enough.
\end{proof}

%\addcontentsline{toc}{section}{References}
\bibliography{biblio}

\bibliographystyle{amsalpha}

\end{document}